\documentclass[12pt]{article}
\usepackage{xspace}
\usepackage[english,french]{babel}
\usepackage[fixlanguage]{babelbib}
\selectbiblanguage{french}
\usepackage[latin1]{inputenc}
\usepackage[T1]{fontenc}
\usepackage{pslatex}
\usepackage{amsmath,amssymb,amsfonts,amsthm}
\usepackage{url}
\usepackage{tikz-cd}
\usepackage{fullpage}
\usepackage{hyperref}
\usepackage[noend]{algorithmic}

\usepackage{caption}

\newcommand\sprime{{\mathop{}_{\mkern40mu '}}}
\def\id{\mathrm{Id}}

\def\Stab{\mathrm{Stab}}
\def\ZZ{\ensuremath{\mathbb Z}}
\def\NN{\ensuremath{\mathbb N}}
\def\QQ{\ensuremath{\mathbb Q}}
\def\CC{\ensuremath{\mathbb C}}
\def\RR{\ensuremath{\mathbb R}}
\def\PP{\ensuremath{\mathbb P}}
\def\charun{{\mathbf 1}}

\def\cC{\ensuremath{\mathcal C}}
\def\cD{\ensuremath{\mathcal D}}
\def\cE{\ensuremath{\mathcal E}}
\def\cF{\ensuremath{\mathcal F}}
\def\cH{\ensuremath{\mathcal H}}
\def\cM{\ensuremath{\mathcal M}}
\def\cO{\ensuremath{\mathcal O}}
\def\cS{\ensuremath{\mathcal S}}
\def\cV{\ensuremath{\mathcal V}}
\def\cW{\ensuremath{\mathcal W}}
\DeclareMathOperator{\Hom}{Hom}
\DeclareMathOperator{\ord}{ord}
\DeclareMathOperator{\im}{Im}
\DeclareMathOperator{\re}{Re}

\theoremstyle{plain}
\newtheorem{thm}{Théorème}[section]
\newtheorem*{thme}{Théorème}
\newtheorem{lem}[thm]{Lemme}
\newtheorem{prop}[thm]{Proposition}
\newtheorem{cor}[thm]{Corollaire}
\theoremstyle{definition}
\newtheorem{defn}[thm]{Définition}
\newtheorem{rem}[thm]{Remarque}

\def\Eisc#1#2{Eis_{#2#1}}
\def\Eiscc#1{Eis_{#1}}
\def\Eis{\mathcal{E}is}
\def\Eisk#1#2{\mathcal{E}_{#1}(#2)}
\def\wEisk#1#2{\widetilde{\mathcal{E}}_{#1}(#2)}
\def\EiskQ#1{\mathcal{E}_{#1}}
\def\wEiskQ#1{\widetilde{\mathcal{E}}_{#1}}
\def\eisenmod#1#2{G_{#1,#2}}
\def\eisen#1#2{E_{#1,#2}}

\def\Per{\mathcal{P}er}
\def\mellin{\cM}
\def\periode{\cM}
\def\cocycle{\mathcal{C}}
\def\cocycleF#1#2{\mathcal{C}(#1,#2)}
\DeclareMathOperator{\Sl}{SL}
\DeclareMathOperator{\Gl}{GL}
\DeclareMathOperator{\Psl}{PSL}

\def\CN#1#2{C_{#1,#2}}
\def\Fonct#1#2{\cF_{#1}(#2)}
\def\Wmod{\widetilde{W}}
\def\cl#1{\widetilde {#1}}
\def\cll#1#2{\tau_{#2}(#1)}
\def\permm#1{\tau_{#1}}
\def\perm#1#2{\tau_{#1}(#2)}
\def\gammap{\gamma\ '}
\def\gammapp{{\gamma\ '}}
\newcommand\concat{+}
\def\Ikf#1#2{I_{#1}(#2)}
\def\Vmod{\widetilde{\cV\ }}
\def\smallmat#1{\left(\begin{smallmatrix}#1\end{smallmatrix}\right)}
\def\KKK{\Xi}
\DeclareMathOperator{\pgcd}{}
\DeclareMathOperator{\ddivides}{\mid\!\mid}
\def\Orb#1{\textrm{Orb}_{\Gamma_0(#1)}}
\def\VN{\cV_{\Gamma_0(N)}}
\def\vN{V_{\Gamma_0(N)}}
\def\dd#1#2{d#1 \otimes d#2}
\def\hypk#1{{#1}_k'}

\title{Symboles modulaires et produit de Petersson}
\usepackage[noblocks]{authblk}

\author[*]{Dominique Bernardi}
\author[**]{Bernadette Perrin-Riou}
\affil[*]{Sorbonne Université, Institut de Mathématiques de Jussieu - Paris Rive Gauche, F-75005 Paris, France}
\affil[**]{Université Paris-Saclay, CNRS, Laboratoire de mathématiques d'Orsay, 91405, Orsay, France.}

\begin{document}
\maketitle
\selectlanguage{english}
\begin{abstract}
We revisit some papers by Eichler and Shimura in order to give an algebraic
formulation (based on Farey symbols) for the intersection product on the
space of modular symbols, as described by Pollack and Stevens.
We define the period homomorphism of an Eisenstein series (Eisenstein-Dedekind-Stevens
symbol) and extend the definition of the intersection product to these objects.
We construct a computationally convenient  basis for the space of Eisenstein series
for $\Gamma_0(N)$ with rational periods. Given a Farey symbol for a subgroup $\Gamma$
of the modular group and a subgroup $\Gamma'$ of finite index of $\Gamma$,
we give an algorithmic construction for a Farey symbol for $\Gamma'$.
\end{abstract}
\selectlanguage{french}
\begin{abstract}
On revisite des articles de Eichler et de Shimura afin de donner une formule algébrique
(basée sur les symboles de Farey) pour le produit d'intersection
sur l'espace des symboles modulaires tel qu'il est décrit par Pollack et Stevens.
On définit l'homomorphisme de périodes
d'une série d'Eisenstein (symbole d'Eisenstein-Dedekind-Stevens)
et on étend le produit d'intersection à ces objets.
On construit une base adaptée à un traitement algorithmique de l'espace
des séries d'Eisenstein de période rationnelle pour $\Gamma_0(N)$.
On donne un algorithme pour construire un symbole de Farey
d'un sous-groupe d'indice fini d'un groupe donné par un symbole de Farey.
\end{abstract}

Le but de cette note est de revisiter des articles de Eichler et de Shimura
(\cite{eichler}, \cite{shimura})
et de donner une formule pour le produit d'intersection
sur l'espace des symboles modulaires tel qu'il est décrit
dans \cite{PS}. Ce produit étendu à $\RR$ ou à $\CC$
et appliqué aux symboles modulaires associés à des formes modulaires
de poids $k$ pour un sous-groupe de congruence $\Gamma$
redonne le produit de Petersson classique.
Cette formule se trouve dans l'article de Eichler \cite{eichler}. Elle
a été reprise par Haberland \cite{haberland} et aussi par Zagier, Pasol-Popa, Cohen, \dots
(\cite{zagier}, \cite{pasol}, \cite{cohen})
en passant par les sous-groupes de congruence de petit niveau
comme $\Gamma_0(2)$.
Ici, comme le font Eichler et Shimura, nous donnons la formule directement
pour les symboles modulaires associés à $\Gamma$ (ou symboles modulaires généralisés
pour tenir compte des symboles d'Eisenstein) en utilisant
les notions de symbole de Farey et de polygone fondamental
associé à un sous-groupe d'indice fini de $\Psl_2(\ZZ)$ telles qu'elles
sont rappelées dans \cite{farey}.
Nous avons pris le parti de redémontrer des résultats déjà connus
mais pouvant être formulés différemment et pour lesquels il n'est pas
toujours facile de trouver une référence claire, en citant
en même temps les articles qui nous ont inspirés.

Donnons un aperçu sans préciser certaines notations qui seront reprises
dans le texte.
Dans la première partie, nous reprenons une extension des symboles modulaires
dans le langage de Stevens \cite{stevenscupcap}.
Soit $k$ un entier $\geq 2$. Considérons le $\QQ$-espace vectoriel
$V_k$ des polynômes en $x$ et $y$ homogènes de degré $k-2$,
muni de la forme bilinéaire $\langle \cdot, \cdot\rangle_{V_k}$
vérifiant
$$\langle (\tau x + y)^{k-2}, (\tau' x + y)^{k-2} \rangle_{V_k} =(\tau-\tau')^{k-2},$$
et $M_k$ l'espace des formes modulaires de poids $k$.
Si $F$ appartient à $M_k$,
l'\textsl{intégrale de Eichler} $W(F)$ est définie pour $\tau$ appartenant
au demi-plan de Poincaré par
$$W(F)(\tau)=\int_{i\infty}^{\tau} \left(F(t) -a_0(F)\right)(t x + y)^{k-2} dt
+ a_0(F) \int_{0}^{\tau} (t x + y)^{k-2} dt \in V_k(\CC)\ .$$
Les deux théorèmes suivants sont des formulations nouvelles de résultats
classiques.
\begin{thme}
\begin{enumerate}
\item L'intégrale définissant $W(F)$ converge absolument
et permet de définir une fonction holomorphe
$W(F)$ sur $\cH$ à valeurs dans $V_k(\CC)$ vérifiant
$$\frac{\partial W(F)}{\partial \tau} (\tau)= F(\tau)(\tau x + y)^{k-2}\ .$$
\item Si $\gamma \in \Gl_2^+(\QQ)$,
$\cocycleF{F}{\gamma}=W(F)- W(F|_k\gamma^{-1})|\gamma$ ne dépend pas de $\tau$.
L'application $$\gamma \mapsto \left (F \mapsto \cocycleF{F}{\gamma}\right)$$
définit un cocycle sur $\Gl_2^+(\QQ)$ à valeurs dans
$\Hom(M_k, V_k(\CC))$.
\end{enumerate}
\end{thme}
Soit $\KKK$ le groupe des diviseurs sur l'ensemble des symboles $\pi_r(s)$
pour $r$ et $s$ $\in \PP^1(\QQ)$, $r\neq s$ et
$\KKK_0$ le sous-groupe des diviseurs de degré 0, munis de l'action naturelle de
$\Gl_2^+(\QQ)$.
\begin{thme}
Il existe un unique $\Gl_2^+(\QQ)$-homomorphisme $\Per$ de $M_k$ dans
$\Hom(\KKK_0, V_k(\CC))$ (homomorphisme de périodes)
tel que pour tout $F \in M_k$ et tout $\gamma \in \Gl_2^+(\QQ)$,
$$ \Per(F)([\pi_\infty(0),\gamma^{-1}\pi_\infty(0)])=
\cocycleF{F}{\gamma}\ .
$$
Si $F$ est une forme modulaire de poids $k$ pour un sous-groupe de congruence
$\Gamma$, $\Per(F)$ appartient à $\Hom_\Gamma(\KKK_0,V_k(\CC))$.
Si $F$ est parabolique, son image dans $H^1(\Gamma, V_k(\CC))$
est un élément de la cohomologie
parabolique $H^1_{par}(\Gamma, V_k(\CC))$.
\end{thme}

Nous construisons ensuite le symbole de Eisenstein-Dedekind
associé aux fonctions de $(\ZZ/N\ZZ)^2$ comme le fait Stevens.
Notons $\cF_N$ le $\QQ$-espace vectoriel des
fonctions sur $(\ZZ/N\ZZ)^2$ à valeurs dans $\QQ$
et $\Fonct{N}{\CC}=\CC \otimes \cF_N$.
Pour $f \in \Fonct{N}{\CC}$ vérifiant de plus $f(0)=0$ si $k=2$
(on notera ce sous-espace $\hypk{\Fonct{N}{\CC}}$),
$$\Eis_{k}(f)(z)=
N^{k-1}\frac{(k-1)!}{(-2i\pi)^k}{\sum_{(c,d)\in \ZZ^2}^\sprime}\widehat{f}(c,d) (cz+d)^{-k}$$
est une forme modulaire de poids $k$ pour le sous-groupe de congruence
$\Gamma(N)$ (la manière dont la sommation est faite pour $k=2$ sera donnée au paragraphe
\ref{sub:eisenstein}). Soit $\cE_{k}(\Gamma)$ l'image par $\Eis_{k}$ des éléments de
${\cF_{N,k}'}$ invariants par $\Gamma$.

\begin{thme}[Stevens]
L'application $\Psi_k=\Per\circ \Eis_k$
définit un homomorphisme de $\Sl_2(\ZZ)$-modules
$\hypk{\Fonct{N}{\CC}}\to \Hom(\KKK_0,V_k(\CC))$.
L'image de $\cE_{k}(\Gamma)$ par $\Per$ dans $H^1(\Gamma,V_k(\CC))$
est en fait incluse dans $H^1(\Gamma,V_k)$.
\end{thme}

Dans la seconde partie, nous donnons la définition de l'accouplement
étudié dans le cadre de ces symboles modulaires et nous faisons le lien avec
le produit de Petersson classique en utilisant les isomorphismes de Eichler-Shimura.
Si $\cD$ est un domaine fondamental de $\Gamma$ dans $\cH^*=\cH \cup \PP^1(\QQ)$ et
si $F$ et $G$ sont deux formes pour $\Gamma$ dont l'une est parabolique,
le produit de Petersson de $F$ et $G$ est défini par
$$\langle F,G\rangle_{\Gamma} =
\int_{\cD} F(\tau) \overline{G(\tau)} y^k \frac{dx dy}{y^2}
=-\frac{1}{2i}
\int_{\cD} F(\tau) \overline{G(\tau)} \im(\tau)^{k-2} d\tau\wedge\overline{d\tau}
\ . $$

\begin{thme}
Soit un symbole de Farey étendu $\cF=(\cV,*, \mu_{ell})$ associé à $\Gamma$.
La forme bilinéaire
$$\Hom_\Gamma(\KKK_0,V_k) \times \Hom_\Gamma(\Delta_0, V_k) \to \QQ$$
définie par \begin{equation*}
\begin{split}
\left\lbrace \Phi_1, \Phi_2\right\rbrace_{\Gamma}&=
\frac{1}{2}\sum_{a\in \Vmod} \langle
  \Phi_1([\pi_\infty(0),\gamma_a\pi_\infty(0)]),\Phi_2 (a)\rangle_{V_k}
\end{split}
\end{equation*}
se factorise en une forme bilinéaire
$$H^1(\Gamma,V_k) \times \Hom_\Gamma(\Delta_0, V_k) \to \QQ$$
et vérifie les propriétés suivantes :
\begin{enumerate}
\item
$\left\lbrace \Phi_1, \Phi_2\right\rbrace_{\Gamma}$ ne dépend pas du choix
du symbole de Farey étendu $\cF$.
\item La forme bilinéaire $\left\lbrace \cdot , \cdot \right\rbrace_{\Gamma}$
induite sur $\Hom_\Gamma(\Delta_0,V_k)\times \Hom_\Gamma(\Delta_0,V_k)$
est antisymétrique (resp. symétrique) si $k$ est pair (resp. impair).
\item
Soient $\Phi_1 \in \Hom_{\Gamma_1}(\Delta_0, V_k)$
et $\Phi_2\in \Hom_{\Gamma_2}(\Delta_0,V_k)$. Si $\alpha \in M_2(\ZZ)^+$, on a
$$
\left\lbrace \Phi_1|[\Gamma_1 \alpha \Gamma_2], \Phi_2 \right\rbrace_{\Gamma_2}
=
\left\lbrace \Phi_1,
\Phi_2|[\Gamma_2 \alpha^* \Gamma_1] \right\rbrace_{\Gamma_1} \ .
$$
\item
Soient $F$ une forme modulaire et $G$ une forme parabolique
de poids $k$ pour $\Gamma$. Alors
\begin{equation*}
\begin{split}
\left\lbrace \Per(F), \overline{\Per(G)}\right\rbrace_{\Gamma}
=-(2i)^{k-1}\langle F,G\rangle_{\Gamma}
\ , \quad
\left\lbrace Per(F), \Per(G)\right\rbrace_{\Gamma}=0 \ .
\end{split}
\end{equation*}
\end{enumerate}
\end{thme}

La définition étant algébrique, nous avons pris le parti d'essayer de démontrer
les propriétés de cette forme bilinéaire
sans utiliser le lien avec le produit de Petersson classique.
Ainsi, pour nous persuader du bon comportement de l'accouplement
construit par les opérateurs de Hecke,
nous avons eu besoin de redonner l'algorithme de construction
d'un polygone fondamental d'un sous-groupe d'indice fini de $\Gamma$ à partir
d'un polygone fondamental de $\Gamma$ sans supposer que $\Gamma= \Sl_2(\ZZ)$.
Par contre, nous ne sommes pas arrivés à démontrer l'indépendance de l'accouplement
par rapport au symbole de Farey étendu de manière algébrique et avons dû passer
par le produit de Petersson classique.

L'accouplement tel qu'il est décrit ici est implémenté dans Pari/GP
\cite{pari} et permet de calculer un sous-espace de l'espace des symboles
modulaires isomorphe à l'espace des formes paraboliques par orthogonalité
avec l'espace des symboles d'Eisenstein.
C'est ce qui nous a amené à reprendre les constructions de Katz et de Stevens
du $\QQ$-espace vectoriel $\cE_{k}(\Gamma)$.
Nous avons alors cherché un système minimal de fonctions
de $(\ZZ/N\ZZ)^2$ dans $\QQ$ permettant d'engendrer $\cE_{k}(\Gamma)$
en partant de constructions de Kubert. Dans le cas où $\Gamma=\Gamma_0(N)$,
nous en indiquons un plus efficace d'un point de vue algorithmique.
Merci à Karim Belabas pour son implémentation dans Pari/GP
des objets contenus dans cet article
(comme les symboles modulaires, les symboles d'Eisenstein,
le produit de Petersson sur ces symboles,
\dots).

Dans l'appendice, nous avons regroupé des calculs classiques
sur les séries d'Eisenstein de niveau $N$ associées aux fonctions
de $(\ZZ/N\ZZ)^2$ dans $\CC$ dans l'esprit de \cite{katz}.

\tableofcontents
\section{Espaces de symboles modulaires}
\subsection{Symboles modulaires et symboles infinitésimaux}
Cette sous-section et les deux suivantes sont inspirées de \cite{stevenscupcap}.

Soient $\Delta= \ZZ[\PP^1(\QQ)]$ le groupe des diviseurs sur $\PP^1(\QQ)$
et $\Delta_0$ le sous-groupe des diviseurs de degré 0.
L'action naturelle de $\Gl_2(\QQ)$ sur $\PP^1(\QQ)$ induit une action
sur $\Delta$.
Si $c \in \PP^1(\QQ)$, on note $\{c\}$ le diviseur associé à $c$ dans $\ZZ[\PP^1(\QQ)]$.
Si $c_1$ et $c_2$ sont dans $\PP^1(\QQ)$, on note $\{c_1,c_2\}$ le diviseur
$\{c_2\} - \{c_1\}$ associé dans $\Delta_0$.

Soit $\mathcal{P}(\QQ)$ l'ensemble des symboles $\pi_r(s)$
pour $r$ et $s$ $\in \PP^1(\QQ)$, $r\neq s$. Il est muni d'une action à gauche
naturelle de $\Gl_2(\QQ)$. Nous parlons de \textsl{pointes infinitésimales} pour
désigner un élément de $\mathcal{P}(\QQ)$.
Soient $\KKK$ le groupe des diviseurs sur $\mathcal{P}(\QQ)$ et
$\KKK_0$ le sous-groupe des diviseurs de degré 0.
Il est engendré par les diviseurs de la forme
$[c_1,c_2] = \{c_2\} - \{c_1\}$ avec $c_1$, $c_2 \in \mathcal{P}(\QQ)$.
On parlera de \textsl{symboles modulaires} pour les éléments de type
$\{r,s\}=[\pi_{r}(s),\pi_s(r)]$ et de \textsl{symboles infinitésimaux} basés
en $r$ pour les éléments
de type $[s,t]_r=[\pi_{r}(s),\pi_r(t)]$ pour $s$ et $t$ différents de $r$.
\begin{figure}[h]
\begin{center}
\begin{tikzpicture}[scale=0.7]
\scriptsize
\draw [->] (0,0) arc (180:0:3);
\draw [->,very thick] (6,0) arc (180:0:3) node [midway, above] {$\{r,t\}$};
\draw (12,0) arc(180:0:3);

\fill [white] (6,1) circle (1); \draw (6,1) circle (1);
\fill [white] (12,1) circle (1); \draw (12,1) circle (1);

\draw (12,2) node [above]  {$[r,u]_t$};
\draw (4.8,1.9) node {$\pi_r(s)$};
\draw (7.2,1.9) node {$\pi_r(t)$};
\draw (6,2) node [above] {$[s,t]_r$};
\draw [dashed] (9,3) arc (90:0:3) arc (180:90:3);

\draw (10.8,1.9) node {$\pi_t(r)$};
\draw (13.2,1.9) node {$\pi_t(u)$};
\draw [dashed] (3,3) arc (90:0:3) arc (180:90:3);
\draw (-1,0) -- (0,0) node [below] {$s$} -- (6,0) node [below] {$r$} -- (12,0) node [below] {$t$} -- (18,0) node [below] {$u$} -- (19,0);

\draw [very thick] (6,2) arc (90:53:1);
\draw [very thick] (6,2) arc (90:127:1);
\draw [very thick] (12,2) arc (90:53:1);
\draw [very thick] (12,2) arc (90:127:1);
\end{tikzpicture}
\end{center}
\caption{$[\pi_r(s),\pi_t(u)]= [s,t]_r+\{r,t\} + [r,u]_t$}
\end{figure}
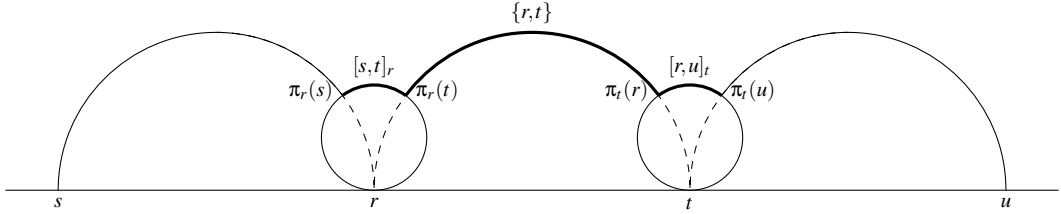
\begin{lem}
Le groupe $\KKK_0$ est engendré par les symboles modulaires
$\{r,s\}$ pour $r\neq s$ et par les symboles infinitésimaux
$[s,t]_r$ pour $s$ et $t$ différents de $r$.
\end{lem}

\begin{proof}
L'élément $[\pi_{r}(s),\pi_t(u)]$ de $\KKK_0$
avec $r\neq s$ et $t\neq u$ s'écrit
\begin{equation*}
[\pi_{r}(s),\pi_t(u)]=
\begin{cases}
[s,u]_r &\text{si $r = t$}\\
[s,t]_r +\{r,t\} +[r,u]_t&\text{si $r\neq t$}
\end{cases}
\end{equation*}
\end{proof}
Soit $\partial : \KKK \to \Delta$ l'application induite par
$\{\pi_r(s)\} \mapsto \{r\}$.
On a donc
\begin{equation*}
\partial ( [\pi_{r_1}(s_1), \pi_{r_2}(s_2)])=
\{r_1, r_2\}
\end{equation*}
\begin{lem}
L'image de $\KKK$ (resp. $\KKK_0$) par $\partial$ est $\Delta$ (resp. $\Delta_0$).
Le noyau de $\partial$ est
$$W= \oplus_{r \in \PP^1(\QQ)} \left(\oplus_{s\neq r} \QQ[\pi_r(s)]\right)_0\ . $$
\end{lem}
\begin{proof}
Soit $v = \sum_{r_i\neq s_i} c_i \{\pi_{r_i}(s_i)\}$. Supposons
$\partial v$ nul. On a donc
$$\sum_{r\in \PP^1(\QQ)} (\sum_{r_i=r} c_i )\{r\} = 0 \ . $$
Donc $\sum_{r_i=r} c_i =0$ pour tout $r\in\PP^1(\QQ)$.
\end{proof}
\subsection{Décomposition de Manin-Stevens}
\begin{lem}Soit $r\in \QQ$.
Notons $r=[a_0,\cdots, a_n]=a_0+1/(a_1+...+1/a_n)$ le développement en fraction continue
de $r$ de convergents $\frac{p_j}{q_j}$.
Notons $a_{n+1}=-\frac{q_{n-1}}{q_n}$ et
$\tau_j=\begin{pmatrix}(-1)^{j-1} p_j&p_{j-1}\\(-1)^{j-1} q_j&q_{j-1}\end{pmatrix}$
pour $-1 \leq j \leq n$
avec $p_{-2}=0$, $q_{-2}=1$, $p_{-1}=1$, $q_{-1}=0$.
Alors
$$
[\pi_0(\infty), \pi_r(\infty)] +
\sum_{j=-1}^n \tau_j [\pi_\infty((-1)^{j+1} a_{j+1}), \pi_0(\infty)] = 0
\ . $$
\end{lem}
\begin{proof}
Posons $c_{-2}=r$, $c_{-1}=\infty=\frac{p_{-1}}{q_{-1}}$, $c_0=\frac{p_0}{q_0}$,
$c_1=\frac{p_1}{q_1}$, $\cdots$,
$c_n=\frac{p_n}{q_n}=r$, $c_{n+1}= \infty$.
On a aussi $q_0=1$, $q_1=a_1$ et donc $c_0=p_0=a_0$.
Alors, le chemin fermé
$$\infty=c_{n+1} \to r=c_{n} \to \cdots \to c_{0} \to c_{-1}=\infty$$
donne la relation
$$ \sum_{j=-1}^{n} [\pi_{c_j}(c_{j+1}), \pi_{c_{j-1}}(c_j)] = 0\ . $$
Posons $a_j'=(-1)^{j}a_j$ pour $i=0,\cdots, n+1$
et calculons les images réciproques des $\tau_j$ sur $c_{j-1}$, $c_{j}$ et $c_{j+1}$ :
\begin{equation*}
\begin{split}
\tau_j \infty &= \frac{p_{j}}{q_{j}}=c_j \text{ pour $0\leq j \leq n$}\\
\tau_j 0 &= \frac{p_{j-1}}{q_{j-1}}=c_{j-1} \text{ pour $0\leq j \leq n$}\\
\tau_{j} a_{j+1}'&= \frac{(-1)^{j-1}p_{j} a'_{j+1} + p_{j-1}}{(-1)^{j-1}q_{j} a'_{j+1} + q_{j-1}}=
\frac{p_{j} a_{j+1} + p_{j-1}}{q_{j} a_{j+1} + q_{j-1}}=
\frac{p_{j+1}}{q_{j+1}}=c_{j+1} \text{ pour $0\leq j< n$}
\\
\tau_n a'_{n+1} &= \frac{-p_{n} q_{n-1} + p_{n-1}q_n}{-q_{n} q_{n-1} + q_{n-1}q_n}=\infty=c_{n+1}
\\
\tau_{-1} \infty&=\infty, \quad
\tau_{-1} 0=0, \quad
\tau_{-1} a_0'=a_0=c_0
\end{split}
\end{equation*}
Ainsi, pour tout $0\leq j \leq n$, on a
$$\tau_j(\infty)=c_j, \quad \tau_j(0)=c_{j-1}, \quad \tau_j(a'_{j+1})=c_{j+1}$$
ce qui implique que
$$[\pi_{c_j}(c_{j+1}), \pi_{c_{j-1}}(c_j)] = \tau_j([\pi_{\infty}(a'_{j+1}), \pi_{0}(\infty)])\ . $$
Pour $j=-1$,
$$[\pi_{c_{-1}}(c_{0}), \pi_{c_{-2}}(c_{-1})]=
[\pi_\infty(a_0), \pi_r(\infty)]
=[\pi_\infty(a_{0}), \pi_0(\infty)] + [\pi_0(\infty), \pi_r(\infty)]
\ . $$
D'où le lemme.
\end{proof}

\begin{prop}
\label{Manin}
\begin{enumerate}
\item
Si $\gamma_r\in \Sl_2(\ZZ)$ est tel que $\gamma_r \infty=r$,
pour $s_1$ et $s_2$ différents de $r$, on a
$$[s_1,s_2]_r=
\gamma_r ([0,\gamma_r^{-1} s_2]_\infty-[0,\gamma_r^{-1} s_1]_\infty)
\ . $$
\item
En utilisant les notations du lemme précédent indexées par $r$, on a
\begin{equation}\label{manino}
[\pi_0(\infty), \pi_r(\infty)]=
\sum_{j=-1}^n \tau_{j,r} [\pi_0(\infty),\pi_\infty((-1)^{j+1}a_{j+1,r})]
\end{equation}
\begin{equation}\label{maninoo}
 [\pi_\infty(0), \pi_r(\infty)]=[\pi_\infty(0), \pi_\infty(a_{0,r})]+
\sum_{j=0}^n \tau_{j,r} [\pi_0(\infty),\pi_\infty((-1)^{j+1}a_{j+1,r})]
\end{equation}
\end{enumerate}
\end{prop}
\begin{prop}\label{generateur}
\begin{enumerate}
\item
Le $\Sl_2(\ZZ)$-module $\KKK_0$ est engendré par le symbole
$\{\infty,0\}$
et par les symboles infinitésimaux $[0,t]_\infty$
pour $t \in \QQ$.
\item
Le $\Sl_2(\ZZ)$-module $W$, noyau de $\partial $, est engendré
par les symboles infinitésimaux $[0,t]_\infty$ pour $t \in \QQ$.
\end{enumerate}
\end{prop}
\begin{proof}
Soit $M$ le $\Sl_2(\ZZ)$-module engendré par le
symbole $\{\infty,0\}$
et par les
$[0,t]_\infty$
pour $t\in \QQ$.
Par la proposition \ref{Manin}, tous les symboles infinitésimaux sont dans $M$.
Comme
$$[\pi_0(\infty),\pi_\infty(s)]=[\pi_0(\infty),\pi_\infty(0)]
+[0,s]_\infty\ , $$
$[\pi_0(\infty),\pi_\infty(s)]$ appartient à $M$.
Donc, par la proposition \ref{Manin},
$[\pi_0(\infty),\pi_r(\infty)]$ appartient à $M$.
Comme $$[\pi_0(\infty),\pi_r(s)]=[\pi_0(\infty),\pi_r(\infty)]
+
[\infty,s]_r \ ,
$$
$[\pi_0(\infty),\pi_r(s)]$ appartient à $M$,
ce qui termine la démonstration.
\end{proof}
\begin{rem}
Soit $(r,s)$ un couple unimodulaire de rationnels, c'est-à-dire
l'image de $(0,\infty)$ par un élément de $\Sl_2(\ZZ)$. Le $\Sl_2(\ZZ)$-module $\KKK_0$
est engendré par $\{r,s\}$ et par les $[s,t]_r$ pour $t\in \QQ- \{r\}$.
\end{rem}
\subsection{Cocycle associé à un symbole}
\subsubsection{Cocycle universel}
Soit $Z$ une pointe infinitésimale.
Si $\gamma \in \Gl_2(\QQ)$, on pose $V_Z(\gamma)=[Z, \gamma(Z)]$.
Explicitement, si l'on prend $Z=\pi_{r}(s)$ avec $r$ et $s$ des éléments distincts
de $\PP^1(\QQ)$, on a
$$V_{Z}(\gamma)=
\begin{cases}
[s,\gamma r]_{r} +\{ r, \gamma r\}
+[r,\gamma s]_{\gamma r} &\text{si $\gamma r \neq r$}\\
[s,\gamma s]_{r} &\text{si $\gamma r = r$}\ .
\end{cases}
$$
L'application $\gamma \mapsto V_Z(\gamma)$ est un $1$-cocycle
sur $\Gl_2(\QQ)$ à valeurs dans $\KKK_0$
$$V_Z(\gamma \gamma') =
[Z, \gamma Z]+[\gamma Z, \gamma(\gamma'Z)]= V_Z(\gamma) + \gamma V_Z(\gamma')
$$
et dépend de $Z$ par un cobord :
$$ V_{Z'}(\gamma)- V_{Z}(\gamma)=
(1-\gamma)([Z',Z])\ . $$
\begin{lem}[Stevens] Si $Z_\infty=\pi_{\infty}(0)$ et
  $\gamma=\begin{pmatrix}a&b\\c&d\end{pmatrix}$,
\begin{equation}
\label{cocycle}
[Z_\infty,\gamma^{-1} Z_\infty]=
\begin{cases}
[\pi_{\infty}(0),\pi_{-\frac{d}{c}}(\infty)]-
\gamma^{-1} \left([\pi_{\infty}(0),\pi_{\infty}(\frac{a}{c})]\right ) &\text{ si $c\neq 0$}\\
[\pi_{\infty}(0),\pi_{\infty}(-\frac{b}{a})]
&\text{ si $c=0$.}
\end{cases}
\end{equation}
\end{lem}
\begin{proof}
On a
\begin{equation*}
[Z_\infty,\gamma^{-1} Z_\infty] = [\pi_{\infty}(0),
\pi_{\gamma^{-1} \infty}(\gamma^{-1}0)]
\end{equation*}
Lorsque $\gamma \infty=\infty$ (i.e $c=0$),
on a
\begin{equation*}
[Z_\infty,\gamma^{-1} Z_\infty] = [\pi_{\infty}(0),\pi_{\infty}(-\frac{b}{a})]
\end{equation*}
Lorsque $c$ est non nul, on a
\begin{equation*}
\begin{split}
[Z_\infty,\gamma^{-1} Z_\infty] =&
[\pi_{\infty}(0),\pi_{\gamma^{-1} \infty}(\infty)]+
[\pi_{\gamma^{-1} \infty}(\infty),\pi_{\gamma^{-1} \infty}(\gamma^{-1}0)]
\\
=&
[\pi_{\infty}(0),\pi_{-\frac{d}{c}}(\infty)]
+
\gamma^{-1}[\pi_{\infty}(\gamma \infty),\pi_{\infty}(0)]\\
=&
[\pi_{\infty}(0),\pi_{-\frac{d}{c}}(\infty)]
+
\gamma^{-1}[\pi_{\infty}(\frac{a}{c}),\pi_{\infty}(0)] ,
\end{split}
\end{equation*}
d'où le lemme.
\end{proof}
\subsubsection{Cocycle universel à valeurs dans $V$}

Si un groupe $G$ agit sur un ensemble $X$, on note
$\Stab_G(x)$ le stabilisateur de $x$ dans $G$ pour $x\in X$.
Soit $V$ un $\QQ$-espace vectoriel muni d'une action à droite de $\Gl_2^+(\QQ)$.
On a une application naturelle de $\Gl_2^+(\QQ)$-modules
$\Hom(\Delta_0, V) \to \Hom(\Xi_0, V)$ induite par $\partial$.
\begin{lem}
\label{lem:surjectivite}
Soient $H$ un sous-groupe de $\Gl_2^+(\QQ)$ et
$\cD(H)$ un système de représentants de $H\backslash\mathcal{P}(\QQ)$.
Soit $c$ un cocycle de $H$ à valeurs dans $V$ s'annulant
sur les stabilisateurs de tous les éléments de $\cD(H)$.
Il existe un élément $\Psi$ de
$\Hom_{H}(\Xi_0, V)$ tel que
$\Psi([Z, \gamma^{-1} Z])=c(\gamma)$ pour tout $\gamma \in H$
et pour tout $Z\in \cD(H)$.
Lorsque $H$ opère transitivement sur $\mathcal{P}(\QQ)$, $\Psi$ est unique.
\end{lem}
\begin{proof}
Remarquons d'abord que si $Z_i \in \cD(H)$ et si $\gamma^{-1}Z_i= {\gamma'}^{-1}Z_i$,
on a $\gamma'=\lambda \gamma$ avec $\lambda \in \Stab_H(Z_i)$
et $c(\gamma')=c(\gamma)$ puisque
$c(\lambda)=0$. On peut donc définir un unique élément
$\Psi$ de $\Hom(\Xi_0, V)$
tel que
$\Psi([T_1, T_2])=c(\gamma_2)-c(\gamma_1)$
avec $T_j=\gamma_j^{-1}Z_{i_j}$ pour un unique $Z_{i_j}$ de $\cD(H)$.
Pour $T_1$ et $T_2$ dans $\mathcal{P}(\QQ)$ et $\gamma \in H$, on a
\begin{equation*}
\begin{split}
\Psi([\gamma T_1, \gamma T_2])&=
\Psi([\gamma\gamma_1^{-1}Z_{i_1}, \gamma\gamma_2^{-1}Z_{i_2}])=
c(\gamma_2 \gamma^{-1})-c(\gamma_1 \gamma^{-1})\\
&=
c(\gamma_2)|\gamma^{-1} + c(\gamma^{-1})
-
c(\gamma_1)|\gamma^{-1} - c(\gamma^{-1})
\\&=
\Psi([T_1, T_2])| \gamma^{-1}
\end{split}
\end{equation*}
ce qui montre que $\Psi$ est un homomorphisme de $H$-modules.
Par définition pour tout $Z_i\in\cD(H)$, on a bien
$\Psi([Z_i, \gamma^{-1} Z_i])=c(\gamma)$.
Si $H$ opère transitivement sur $\mathcal{P}(\QQ)$, $\cD(H)$ est réduit à un élément $Z$
et les $[Z, \gamma^{-1} Z]$ engendrent $\Xi_0$, d'où l'unicité de
$\Psi$.
\end{proof}

Regardons le cas particulier où $H=\Gl_2^+(\QQ)$. L'action de $\Gl_2^+(\QQ)$
sur $\mathcal{P}(\QQ)$ est transitive.
Si $Z\in \mathcal{P}(\QQ)$, soit $U_Z$ l'application
$\Hom_{\Gl_2^+(\QQ)}(\Xi_0, V)\to H^1({\Gl_2^+(\QQ)}, V)$
telle que
$U_Z(\psi)$ est la classe du cocycle $\gamma \mapsto \psi([Z, \gamma^{-1} Z])$.
Si $r\in \PP^1(\QQ)$, on note de même $U_r$ l'application
$$\Hom_{\Gl_2^+(\QQ)}(\Delta_0, V)\to H^1({\Gl_2^+(\QQ)}, V)$$
telle que
$U_r(\psi)$ est la classe du cocycle $\gamma \mapsto \psi(\{r, \gamma^{-1} r\})$.
\begin{prop}Si $Z\in \mathcal{P}(\QQ)$, on a la suite exacte
\begin{equation*}
0\to V^{\Gl_2^+(\QQ)} \to V^{\Stab_{\Gl_2^+(\QQ)} Z}
\to
\Hom_{\Gl_2^+(\QQ)}(\Xi_0, V) \overset{U_Z}{\to} H^1({\Gl_2^+(\QQ)}, V)
\overset{res}{\to} H^1(\Stab_{\Gl_2^+(\QQ)} Z, V)
\end{equation*}
où $V^{\Stab_{\Gl_2^+(\QQ)} Z}\to
\Hom_{\Gl_2^+(\QQ)}(\Xi_0, V)$ est définie
par $v \mapsto ([Z, \gamma^{-1} Z] \mapsto v|(\gamma -1))$.
\end{prop}
\begin{proof}
Soit $c$ un cocycle sur $\Gl_2^+(\QQ)$ à valeurs dans $V$
qui est un cobord par restriction à $\Stab_{\Gl_2^+(\QQ)}(Z)$ :
il existe $v\in V$ tel que
$c(\gamma)= v|(\gamma -1)$ pour $\gamma \in \Stab_{\Gl_2^+(\QQ)}(Z)$.
En modifiant $c$ par le cobord $\gamma \mapsto v|(\gamma-1)$, on
peut supposer que $c$ est nul sur $\Stab_{\Gl_2^+(\QQ)}(Z)$ et appliquer le lemme
\ref{lem:surjectivite}. Les autres points sont immédiats.
\end{proof}
\begin{prop}
\label{prop:suiteexacte}
Soit $\Gamma$ un sous-groupe de congruence de $\Sl_2(\ZZ)$ et $Z \in \mathcal{P}(\QQ)$.
Si $\{\pm 1\}$ agit trivialement sur $V$,
l'homomorphisme
\begin{equation*}
\Hom_{\Gamma}(\Xi_0, V) \overset{U_{Z}}\to H^1(\Gamma, V)
\end{equation*}
est surjectif.
L'image de $\Hom_{\Gamma}(\Delta_0, V)$ par $U_{\partial Z}$ est égale à
la cohomologie parabolique
$H^1_{par}(\Gamma, V)$, noyau de $H^1(\Gamma, V) \to
\prod_{r\in C(\Gamma)} H^1(\Stab_\Gamma(r), V)$
pour $C(\Gamma)$ un système de représentants de $\Gamma \backslash \PP^1(\QQ)$.
\end{prop}
\begin{proof}
Soit $c$ un cocycle de $\Gamma$ à valeurs dans $V$.
Comme $\Stab_\Gamma Z$ est réduit à $\{\pm 1\}$ pour tout $Z\in \mathcal{P}(\QQ)$,
la restriction de $c$ à $\Stab_\Gamma Z$ est nulle
et $H^1(\Stab_\Gamma Z, V)$ est nul.
La condition du lemme \ref{lem:surjectivite} est donc vérifiée par $c$
et on peut appliquer le lemme. En choisissant un système de représentants $\cD(\Gamma)$
de $\Gamma\backslash\mathcal{P}(\QQ)$ contenant $Z$, ceci montre la surjectivité de $U_{Z}$.

Si $\Phi$ provient d'un élément de $\Hom_{\Gamma}(\Delta_0, V)$ et $r\in \PP^1(\QQ)$,
l'image de $\Phi$ dans $H^1(\Gamma,V)$ est définie
comme la classe de $\gamma \mapsto \Phi(\{r, \gamma^{-1} r\})$ et sa restriction à
$H^1(\Stab_{\Gamma}(r),V)$ est donc triviale.
Montrons la surjectivité de $\Hom_{\Gamma}(\Delta_0, V) \to H^1_{par}(\Gamma, V)$
\footnote{Cela peut certainement se déduire de suites exactes générales, mais nous
préférons donner une démonstration terre à terre.}.
Soit $c$ un cocycle représentant un élément de $H^1_{par}(\Gamma, V)$.
Pour tout $r \in C(\Gamma)$, choisissons $v_r \in V$ tel que
$c(\gamma)=v_r|(\gamma -1)$ pour tout $\gamma \in \Stab_\Gamma(r)$.
Pour tout $s \in \PP^1(\QQ)$, il existe un couple
$(t_s,\beta_s)\in C(\Gamma)\times \Gamma$ tel que $s=\beta_s t_s$, et la quantité
$c(\beta_s^{-1}) - v_{t_s} |\beta_s^{-1} $ ne dépend pas du choix de ce couple. On peut donc définir
un élément $\Phi\in \Hom(\Delta, V)$ par
$$
\Phi(\sum_{s\in \PP^1(\QQ)} n_s \{s\})=
\sum_{s\in \PP^1(\QQ)} n_s \left ( c(\beta_s^{-1}) - v_{t_s} |\beta_s^{-1} \right )\ .
$$
où $n_s$ est nul sauf pour un nombre fini de $s$.
Pour $\gamma\in\Gamma$, l'équation $s=\beta t$ donne $\gamma s=(\gamma\beta)t$, et
$$c(\gamma\beta)^{-1} - v_t |(\gamma\beta)^{-1} =(c(\beta^{-1})-v_t|\beta^{-1})|\gamma^{-1}+c(\gamma^{-1}),$$
donc $$\Phi(\gamma\!\!\!\sum_{s\in \PP^1(\QQ)} n_s \{s\})=\Phi(\sum_{s\in \PP^1(\QQ)} n_s \{s\})|\gamma^{-1}+(\sum_{s\in \PP^1(\QQ)} n_s)c(\gamma^{-1}),$$ ce qui montre que la restriction de $\Phi$ à $\Delta_0$ est invariante par $\Gamma$.
Pour $r \in C(\Gamma)$, on a
$$\Phi(\{r, \gamma^{-1} r\})=\Phi(\{\gamma^{-1} r\})-\Phi(\{r\})=c(\gamma)-v_r|\gamma -c(1)+v_r|1=c(\gamma)+v_r|(1-\gamma)\ ,$$
ce qui montre que $c$ est bien un représentant de l'image de $\Phi$ dans $H^1(\Gamma,V)$.
\end{proof}
\subsection{Symboles et espaces de formes modulaires}
Le paragraphe suivant a des liens avec \cite{heumann}.
\label{eichler}
Soit $k$ un entier supérieur ou égal à 2.
Soit $V_k=\QQ[x,y]_{k-2}$ le $\QQ$-espace vectoriel des polynômes homogènes
en $(x,y)$ de degré $k-2$, muni de l'action à droite de $\Gl_2(\QQ)$
$$(P | \gamma)(x,y) = \det(\gamma)^{k-2}P((x,y)\gamma^{-1}) \ . $$
Si $\gamma=\begin{pmatrix}a&b\\c&d\end{pmatrix}$, on a donc
$$(P|\gamma)(x,y)= P(d x- c y,-b x + a y)\ .$$
On note $j_{\gamma}(z)=cz+d$ le facteur d'automorphie.
Soit $M_k$ l'espace des formes modulaires dont le stabilisateur
pour l'action de $\Gl_2^+(\QQ)$
donnée par
$$F|_k\gamma (z)=\det(\gamma)^{k-1}j_{\gamma}(z)^{-k}F(\gamma z)$$
est un sous-groupe de congruence.
Si $\Gamma$ est un sous-groupe de congruence,
on note $M_k(\Gamma)$ (resp. $S_k(\Gamma)$) l'espace des formes modulaires
(respectivement paraboliques) de poids $k$ invariantes par $\Gamma$.
Si $F \in M_k$,
on note $a_0(F)$ le terme constant du développement de Fourier de $F$ en l'infini.
On définit l'\textsl{intégrale de Eichler} $W(F)$, fonction de $\cH$ dans $V_k(\CC)$, par
$$W(F)(\tau)=\int_{i\infty}^{\tau} \left(F(t) -a_0(F)\right)(t x + y)^{k-2} dt
+ a_0(F) \int_{0}^{\tau} (t x + y)^{k-2} dt \in V_k(\CC)\ .$$
Les chemins dans $\cH^*$ utilisés
pour définir cette intégrale (et les suivantes)
sont toujours constitués d'un nombre fini d'arcs de géodésiques.
L'action de $\gamma \in \Gl_2^+(\QQ)$ sur $W(F)$ est l'action naturelle
sur l'espace des fonctions de $\cH$ dans $V_k(\CC)$ :
$$\left(W(F)|\gamma\right)(\tau)= \left(W(F)(\gamma \tau)\right)| \gamma \ .$$
Si $f$ est une fonction sur $\cH$, on note
$\frac{\partial f}{\partial \tau}$ et $\frac{\partial f}{\partial \overline{\tau}}$
les composantes de sa différentielle dans la base $(d \tau, d \overline{\tau})$
$$df = \frac{\partial f}{\partial \tau }d \tau +
  \frac{\partial f}{\partial \overline{\tau}}d\overline{\tau}\ .$$
  On a donc
\begin{equation*}
\frac{\partial f}{\partial \tau }=
\frac{1}{2}\left (
\frac{\partial f}{\partial x } - i\frac{\partial f}{\partial y }
\right)\quad
\frac{\partial f}{\partial \overline{\tau} }=
\frac{1}{2}\left (
\frac{\partial f}{\partial x } + i\frac{\partial f}{\partial y }
\right)\ .
\end{equation*}

\begin{thm}
\label{cocyclehol}
\begin{enumerate}
\item L'intégrale définissant $W(F)$ converge absolument
et permet de définir une fonction holomorphe
$W(F)$ sur $\cH$ à valeurs dans $V_k(\CC)$ vérifiant
$$\frac{\partial W(F) }{\partial \tau} (\tau)= F(\tau)(\tau x + y)^{k-2}\ .$$
\item Si $\gamma \in \Gl_2^+(\QQ)$,
$\cocycleF{F}{\gamma}=W(F)- W(F|_k\gamma^{-1})|\gamma$ ne dépend pas de $\tau$.
L'application $$\gamma \mapsto \left (F \mapsto \cocycleF{F}{\gamma}\right)$$
définit un cocycle sur $\Gl_2^+(\QQ)$ à valeurs dans
$\Hom(M_k, V_k(\CC))$.
Si $\Gamma$ est un sous-groupe de $\Gl_2^+(\QQ)$ tel que $F$ soit invariante par $\Gamma$,
la restriction de $\cocycle(F)$ à $\Gamma$ est un cocycle sur $\Gamma$ à valeurs dans $V_k(\CC)$
et définit un élément de $H^1(\Gamma, V_k(\CC))$.
\item La fonction $s\mapsto \int_0^{i\infty}\left(F(t)-a_0(F)\right )(t x + y)^{k-2} (-it)^{s} dt$
définie pour $\re(s)\geq k$
admet un prolongement analytique à $\CC$.
On a
\begin{equation}
\cocycleF{F}{\gamma}=
\begin{cases}
\int_{i\infty}^0\left(F(t)-a_0(F)\right)(t x + y)^{k-2} (-it)^s dt|_{s=0}
&\text{si $\gamma=\begin{pmatrix}0&-1\\1&0\end{pmatrix}$}\label{eqsigma}\\
a_0(F) \int_0^{\gamma^{-1} 0} (t x + y)^{k-2} dt &\text{si $\gamma$ stabilise $\infty$\ .}
\end{cases}
\end{equation}
Ces équations déterminent entièrement $\cocycle$.
\end{enumerate}
\end{thm}
En particulier, $\cocycleF{F}{\gamma}=0$ si $\gamma$ stabilise $\pi_\infty(0)$.
Si $F \in M_k$, on note $\tilde{F}(t)=F(t)-a_0(F)$ pour simplifier
l'écriture des formules.
\begin{proof}
Notons $N$ un niveau de $F$. On a
$$\widetilde{F}(t) = \sum_{n\geq 1} a_n \exp(\frac{2i\pi n t}N)$$
avec $a_n=\cO(n^{k-1})$.
 L'inégalité $|\exp(\frac{2i\pi n t}N)| \le \exp(\frac{-2\pi \im(t)}N)^n$
 assure la convergence absolue de l'intégrale
en $i\infty$ et justifie la définition de $W(F)$, ainsi que l'équation différentielle.

Vérifions que $W(F) -W(F|\gamma^{-1})|\gamma $ ne dépend pas de $\tau$.
On a
\begin{equation*}
\begin{split}
W(F| \gamma^{-1})|\gamma(\tau)
&= \left(\int_{i\infty}^{\gamma \tau} \widetilde{F| \gamma^{-1}}(t) (t x + y)^{k-2}dt
+a_0(F| \gamma^{-1}) \int_{0}^{\gamma \tau}(t x + y)^{k-2}dt\right )|\gamma \ .
\end{split}
\end{equation*}
En utilisant les égalités suivantes faciles à vérifier
\begin{equation}
\label{formulej}
\begin{cases}
\left.(\gamma \tau x + y)^{k-2}\right| \gamma &=\det(\gamma)^{k-2}j_{\gamma^{-1}}(\gamma \tau)^{k-2} (\tau x + y)^{k-2}\\
\left(F|\gamma^{-1}\right)(\gamma \tau)&= \det(\gamma)^{-(k-1)}j_{\gamma^{-1}}(\gamma \tau)^{-k} F(\tau)\\
j_{\gamma^{-1}}(\gamma \tau) j_{\gamma}(\tau) &=1,
\end{cases}
\end{equation}
on obtient
$$\frac{\partial\left( W(F| \gamma^{-1})| \gamma\right)(\tau)}{\partial \tau}
= \frac{\det(\gamma)}{ j_\gamma(\tau)^{2}} \left(
F| \gamma^{-1}\right)(\gamma \tau)\left.(\gamma\tau x + y)^{k-2}\right|\gamma
=
F(\tau) (\tau x + y)^{k-2}=
\frac{\partial W(F)}{\partial \tau}(\tau)\ .
$$

Le fait que $\gamma \mapsto (F\mapsto \cocycleF{F}{\gamma})$ est un cocycle :
$$ \cocycleF{F}{\gamma_1 \gamma_2}= \cocycleF{F| \gamma_2^{-1}}{\gamma_1}| \gamma_2 + \cocycleF{F}{\gamma_2}$$
se démontre en appliquant la définition.
Enfin, calculons $\cocycleF{F}{\gamma}$ lorsque $\gamma \infty=\infty$,
i.e. lorsque $\gamma$ est de la forme $\smallmat{a&b\\0&d}$.
Pour $\im(\tau)$ tendant vers $\infty$,
l' intégrale de $i\infty$ à $\tau$ et l'intégrale de $i\infty$ à $\gamma\tau$ tendent vers 0.
La valeur de $W(F) - W(F|\gamma^{-1})|\gamma $ est donc
égale à la limite de
\begin{equation*}
\begin{split}
a_0(F)\int_{0}^{\tau}&(t x + y)^{k-2} dt
-a_0(F|\gamma^{-1})\left(\int_{0}^{\gamma \tau} (t x + y)^{k-2} dt\right)|\gamma
\\
&=
a_0(F)\left(
\int_{0}^{\tau}(t x + y)^{k-2} dt - \frac{d}{a}
\int_{0}^{\gamma \tau} (\gamma^{-1}t x + y)^{k-2} dt
\right)\\
&=
a_0(F)\left(\int_{0}^{\tau}(t x + y)^{k-2} dt
-
\int_{\gamma^{-1} 0}^{\tau} (t x + y)^{k-2} dt
\right)\\
&=
a_0(F)
\int_{0}^{\gamma^{-1} 0} (t x + y)^{k-2} dt
\end{split}
\end{equation*}
car
$a_0(F|\gamma^{-1})=
\frac{d}{a^{k-1}} a_0(F)$. Cela démontre la formule pour $\cocycleF{F}{\gamma}$
lorsque $\gamma \infty=\infty$.

Calculons $\cocycleF{F}{\sigma}$ pour $\sigma=\begin{pmatrix}0&-1\\1&0\end{pmatrix}$.
Pour cela, commençons par justifier le prolongement analytique
de $s \mapsto \int_0^{i\infty}\tilde{F}(t)(t x + y)^{k-2} (-it)^s dt$
\footnote{On donne ainsi un sens à
l'intégrale $\int_0^{i\infty} \tilde{F}(t)(t x + y)^{k-2} dt$
qui ne converge pas en $0$.
Il est dans l'esprit des démonstrations des équations fonctionnelles
des fonctions $L$ (voir \cite{heumann}).}.
Remarquons d'abord que $t$ étant dans le demi-plan de Poincaré,
$-it$ appartient au plan complexe privé de la demi-droite
$x\leq0$, ce qui permet de donner un sens à $(-it)^s$ pour $s \in \CC$.
En $0$, $F(t)$ est $\cO(t^{-k})$, donc l'intégrale
est convergente en $0$ pour $\re(s) > k-1$.
Prenons un élément $z_0$ de $\cH$ sur la géodésique
qui va de $0$ à $i\infty$, par exemple,
$z_0= i$.
On écrit
\begin{equation*}
\begin{split}
\int_{i\infty}^0& \tilde{F}(t)(t x + y)^{k-2} (-it)^s dt=
\int_{i\infty}^i \tilde{F}(t)(t x + y)^{k-2} (-it)^s dt
-
\int_0^{i} \tilde{F}(t)(t x + y)^{k-2} (-it)^s dt
\\
&=
\int_{i\infty}^i \tilde{F}(t)(t x + y)^{k-2} (-it)^{s} dt
-\int_0^{i} F(t)(t x + y)^{k-2} (-it)^s dt
+ a_0(F) \int_0^{i}(t x + y)^{k-2} (-it)^s dt
\\
&=
\int_{i\infty}^i \tilde{F}(t)(t x + y)^{k-2} (-it)^{s} dt
-
\left(\int_i^{i\infty} \widetilde{F|\sigma}(t)(tx + y)^{k-2} (-it)^{-s} dt\right)|\sigma
\\&
+ a_0(F) \int_0^{i}(t x + y)^{k-2} (-it)^s dt
- a_0(F|\sigma) \left(\int_0^{i}(t x + y)^{k-2} (-it)^{-s} dt\right)| \sigma
\\
&=W(F)(i)-W(F|_k \sigma ^{-1})|\sigma(i) \ .
\end{split}
\end{equation*}
Les deux premiers termes définissent des fonctions entières,
les deux autres sont des fractions rationnelles en $s$ sans pôle en $0$.
En prenant la valeur en $s=0$, on en déduit la formule
\eqref{eqsigma}.
\end{proof}
\begin{rem}
Comme cela est remarqué dans \cite{heumann},
cette formule peut aussi s'écrire en remplaçant $i$ par $\tau \in \cH$ :
\begin{equation*}
\begin{split}
\int_{i\infty}^0& \tilde{F}(t)(t x + y)^{k-2} (-it)^s dt=
\int_{i\infty}^{\sigma\tau} \widetilde{F}(t)(t x + y)^{k-2} (-it)^{s} dt
-
\left(\int_{i\infty}^\tau \widetilde{F|\sigma}(t)(tx + y)^{k-2} (-it)^{-s} dt\right)|\sigma
\\&
+ a_0(F) \int_0^{{\sigma\tau}}(t x + y)^{k-2} (-it)^s dt
- a_0(F|\sigma) \left(\int_0^{\tau}(t x + y)^{k-2} (-it)^{-s} dt\right)| \sigma
\end{split}
\end{equation*}
ou encore
\begin{equation}
\label{valeur}
\begin{split}
\int_{i\infty}^0& \tilde{F}(t)(t x + y)^{k-2} (-it)^s dt|_{s=0}=
\int_{i\infty}^{\sigma\tau} \widetilde{F}(t)(t x + y)^{k-2} dt
+ a_0(F) \int_0^{\sigma\tau}(t x + y)^{k-2} dt
\\&
-\left(\int_{i\infty}^\tau \widetilde{F|\sigma}(t)(tx + y)^{k-2} dt
+ a_0(F|\sigma) \int_0^{\tau}(t x + y)^{k-2}dt \right)| \sigma\\
&=
\cocycleF{F}{\sigma} \ .
\end{split}
\end{equation}
La formule étant vraie pour tout $\tau \in \cH$, on peut
aussi écrire en remplaçant $\tau$ par $\sigma^{-1} \tau$
\begin{equation*}
\begin{split}
\int_{i\infty}^0& \tilde{F}(t)(t x + y)^{k-2} (-it)^s dt|_{s=0}=
\int_{i\infty}^{\tau} \widetilde{F}(t)(t x + y)^{k-2} dt
+ a_0(F) \int_0^{\tau}(t x + y)^{k-2} dt
\\&
-\left(\int_{\sigma^{-1} 0}^{\sigma^{-1}\tau}
\widetilde{F|\sigma}(t)(tx + y)^{k-2} dt
+ a_0(F|\sigma) \int_{\sigma^{-1} \infty}^{\sigma^{-1}\tau}(t x + y)^{k-2}dt \right)| \sigma
\ .
\end{split}
\end{equation*}
\end{rem}

\begin{thm}
\label{thm:fond}
Il existe un unique $\Gl_2^+(\QQ)$-homomorphisme $\Per$ de $M_k$ dans
$\Hom(\Xi_0, V_k(\CC))$
tel que pour tout $F \in M_k$ et tout $\gamma \in \Gl_2^+(\QQ)$,
$$ \Per(F)([\pi_\infty(0),\gamma^{-1}\pi_\infty(0)])=
\cocycleF{F}{\gamma}
$$
ou encore
\begin{equation}
\label{eq:cocycle}
\begin{cases} \Per(F)(\{\infty,0\})=&
\int_{i\infty}^0 \left(F(t)-a_0(F)\right)(t x + y)^{k-2}(-it)^s dt |_{s=0}
\\
\Per(F)([0,r]_\infty)=& a_0(F) \int_0^r (t x + y)^{k-2} dt \ .
\end{cases}
\end{equation}
\end{thm}
\begin{proof}
L'existence et l'unicité de $\Per$ se déduisent
du lemme \ref{lem:surjectivite} avec $Z=\pi_\infty(0)$, $G=\Gl_2^+(\QQ)$
et
$V=\Hom(M_k,V_k)$ en remarquant que
$\Hom(M_k,\Hom(\Xi_0,V_k))= \Hom(\Xi_0,\Hom(M_k,V_k))$
et que $\cC$ est nul sur le stabilisateur de $\pi_\infty(0)$.
Comme $\Gl_2^+(\QQ)$ est engendré par $\sigma$ et par les matrices triangulaires
supérieures de déterminant strictement positif,
les égalités \eqref{eq:cocycle} suffisent à caractériser l'application
$\Per$ en utilisant le théorème \ref{cocyclehol}.
\end{proof}
\begin{rem}
Attention, $\Per(F)$ dépend du choix du cocycle $\gamma \mapsto \cocycleF{F}{\gamma}$
parmi les cocycles de sa classe qui sont
nuls sur le stabilisateur de $\pi_\infty(0)$ dans $\Gl_2^+(\QQ)$.
\end{rem}
\begin{prop} Soit $F$ une forme modulaire de poids $k$ de groupe $\Gamma$.
L'image de $\Per(F)$ dans $H^1(\Gamma, V_k(\CC))$
est la classe du cocycle $\cocycle(F)$ :
$$\Per(F)([\pi_\infty(0),\gamma^{-1} \pi_\infty(0)])=\cocycleF{F}{\gamma}$$
pour $\gamma \in \Gamma$.
Soient $t\in \PP^1(\QQ)$ et $\Gamma_t$ son stabilisateur dans $\Gamma$.
L'image de $\Per(F)$ dans
$H^1(\Gamma_t, V_k(\CC))$ est la classe du cocycle
$\gamma \in \Gamma_t \mapsto \Per(F)([s, \gamma^{-1}s]_t)$
pour n'importe quel $s\in \PP^1(\QQ)$ différent de $t$.
En particulier, lorsque $F$ est parabolique,
la classe de $\cocycle(F)$ est un élément de la cohomologie
parabolique $H^1_{par}(\Gamma, V_k(\CC))$.
\end{prop}

\begin{proof}Démontrons l'assertion sur la pointe $t$.
Calculons l'image de la classe de $\cocycle(F)$ dans
$H^1(\Gamma_t, V_k(\CC))$.
On a pour $s\neq t$ et $\gamma \in \Gamma_t$
\begin{equation*}
\begin{split}
\cocycleF{F}{\gamma}&=
\Per(F)([\pi_\infty(0),\pi_t(s)])
+
\Per(F)([\pi_t(s), \pi_t(\gamma^{-1}s)])
+
\Per(F)([\gamma^{-1}\pi_t(s),\gamma^{-1} \pi_\infty(0)])\\
&=
\Per(F)([\pi_\infty(0),\pi_t(s)])| (1 - \gamma)
+
\Per(F)([\pi_t(s), \pi_t(\gamma^{-1}s)])\ .
\end{split}
\end{equation*}
L'image de $\cocycle(F)$ dans $H^1(\Gamma_t, V_k(\CC))$ est donc la classe du cocycle
$\gamma \in \Gamma_t \mapsto \Per(F)([\pi_t(s), \pi_t(\gamma^{-1}s)])$.

\end{proof}
Notons $\Per_\RR(F)$ (resp. $\cocycle_\RR(F)$) la partie réelle
de $\Per(F)$ (resp. $\cocycle(F)$). Il serait possible de les définir
en prenant l'intégrale de la partie réelle de la forme différentielle
dans les intégrales.
\begin{thm}[Eichler-Shimura]
L'application
$\cocycle_\RR: M_k(\Gamma) \to H^1(\Gamma,V_k(\RR))$
est un isomorphisme de $\RR$-espaces vectoriels.
L'image de $S_k(\Gamma)$ est la cohomologie parabolique
$H^1_{par}(\Gamma,V_k(\RR))$.
\end{thm}
Par transport de structure, on en déduit une structure de $\CC$-espace vectoriel sur
$H^1(\Gamma,V_k(\RR))$.

\subsection{Symboles d'Eisenstein}
Cette sous-section a sa source dans \cite{katz}.

\subsubsection{\texorpdfstring{Fonctions sur $(\ZZ/N\ZZ)^2$}{Lg}}
Soit $\Fonct{N}{\CC}$ l'espace vectoriel des fonctions de
$(\ZZ/N\ZZ)^2$ dans $\CC$ muni de l'action suivante de $\Gl_2(\ZZ)$ :
$$f \mid \gamma\ (x,y)= f((x,y)\gamma^{-1})\ .$$
Si $f \in \Fonct{N}{\CC}$,
on note $\Gamma_f$ le stabilisateur de $f$ dans $\Sl_2(\ZZ)$.
Pour $\Gamma$ contenant $\Gamma(N)$,
on note $\Fonct{\Gamma}{\CC}$ le sous-espace des éléments de $\Fonct{N}{\CC}$
invariants par $\Gamma$. On a donc l'abus de notation
$\Fonct{\Gamma(N)}{\CC}=\Fonct{N}{\CC}$.
Si $X$ est un sous-ensemble de $(\ZZ/N\ZZ)^2$, on note $\charun_X$ la fonction indicatrice
de $X$.
En particulier, la fonction indicatrice d'une orbite sous $\Gamma$
appartient à $\Fonct{\Gamma}{\ZZ}$.
\begin{defn}
Si $f$ est une fonction sur $\ZZ/N\ZZ$, on définit la
\textsl{transformée de Fourier} de $f$
par
$$\widehat{f}(n)=\sum_{a \bmod N} f(a) \exp(-\frac{2 \pi i an }{N} )\ .$$
\end{defn}
L'application réciproque est donnée par
$$f(n)=\frac{1}{N}\sum_{a \bmod N} \widehat{f}(a) \exp(\frac{2 \pi i an }{N})$$
et on a $ \widehat{\widehat{f\, }} = Nf^-$
où $f^-$ est définie par $f^-(x)=f(-x)$.
En particulier,
\begin{equation*}
\begin{split}
\widehat{f\ }(0)=&\sum_{a \bmod N} f(a) , \quad
f(0)= N^{-1}\sum_{a \bmod N} \widehat{f\ }(a) \ .
\end{split}
\end{equation*}

\begin{defn}\cite{katz}
Si $f$ est une fonction sur $(\ZZ/N\ZZ)^2$, on définit la \textsl{transformée de Fourier} de $f$
par
$$\widehat{f\ }(n,m) =
\frac{1}{N} \sum_{(a,b) \bmod N} f(a,b)
\exp(\frac{2 \pi i }{N}
\det \begin{pmatrix}a&b\\n&m
\end{pmatrix})
$$
et ses transformées de Fourier partielles par
$$P_1(f)(n,m) = \sum_{a \bmod N} f(a,m) \exp(-\frac{2 \pi i a n}{N})$$
$$P_2(f)(n,m) = \sum_{b \bmod N} f(n,b) \exp(-\frac{2 \pi i b m}{N})\ .$$
\end{defn}
Pour $\gamma \in \Sl_2(\ZZ)$, on a
$ \widehat{f\ }|\gamma = \widehat{f|\gamma}$.
On a aussi $\widehat{\widehat{f\, }} = f$,
$P_i(f)^-=P_i(f^-)$ et $ \widehat{f\ }^-=\widehat{f^-}$.
Si $f$ est à variables séparées: $f(x_1,x_2)=f_1(x_1)f_2(x_2)$, i.e.
$f=f_1\otimes f_2$,
on a $$
\widehat{f_1\otimes f_2}=\frac{1}{N}\widehat{f_2}\otimes \widehat{f_1}^- $$
et
\begin{equation*}
\begin{split}
P_1(f_1\otimes f_2)=& \widehat{f_1}\otimes f_2,
\quad P_2(f_1\otimes f_2)=f_1\otimes \widehat{f_2},\quad
P_1(\widehat{f_1\otimes f_2\ })=f_2^-\otimes \widehat{f_1}^-,
\quad
P_2(\widehat{f_1\otimes f_2\ })=\widehat{f_2}\otimes \widehat{f_1}\ .
\end{split}
\end{equation*}
Les définitions suivantes d'opérateurs de Hecke seront justifiées
une fois fait le lien avec les
séries d'Eisenstein (voir \cite{stevens} pour le cas de poids 2).

\begin{defn} Soit $f \in \Fonct{\Gamma}{\CC}$. Pour $k$ entier $\geq2$,
on pose
\begin{equation*}
T_\ell^{(k)}(f)(x,y)= \sum_{\ell s=y} f(x,s) +
\ell^{k-1} f(\ell x, y)
-\ell^{k-2}\begin{cases}
0 &\text{ si $\ell \nmid N$}
\\
\sum_{\ell t=\ell y} f(\ell x, t) &\text{ si $\ell \mid N$}\ .
\end{cases}
\end{equation*}
\end{defn}
Pour $\ell \nmid N$, la formule se simplifie en
\begin{equation*}
T_\ell^{(k)}(f)(x,y)= f(x,\ell^{-1}y) + \ell^{k-1} f(\ell x, y)\ .
\end{equation*}

\subsubsection{Distributions de Bernoulli}
Soit $B_h$ les polynômes de Bernoulli (voir \cite{lang} par exemple).
Si $a$ est un entier, on pose
$\langle\frac{a}{N}\rangle= \frac{a}{N} - \left\lfloor\frac{a}{N}\right\rfloor$
et on note pour $r \in \ZZ/N\ZZ$
\begin{equation*}
\begin{split}
\beta_h(r)= \begin{cases}
-N^{h-1}\frac{B_h(\langle\frac{r}{N}\rangle)}{h} &\text{si $h>1$}\\
-B_1(\langle\frac{r}{N}\rangle)-\frac{1}{2}\delta_{0}(\langle\frac{r}{N}\rangle) &\text{si $h=1$}\\
N^{-1} &\text{si $h=0$}
\end{cases}
\end{split}
\end{equation*}
où $\delta_0$ vaut 1 en 0 et 0 ailleurs.
La distribution de Dirac sur $\ZZ/N\ZZ$ est définie par
$$\int_{\ZZ/N\ZZ} f(x) d\delta_0(x)= f(0)\ .$$
Les \textsl{distributions de Bernoulli} $\beta_h$ sur $\ZZ/N\ZZ$ sont définies pour $h\geq 0$
par
$$\int_{\ZZ/N\ZZ} f(x) d\beta_h(x)= \sum_{a\in \ZZ/N\ZZ} f(a) \beta_h(a)\ .$$
On a en particulier
$$\int_{\ZZ/N\ZZ} f(x) d\beta_0(x)=N^{-1}\sum_{a\in \ZZ/N\ZZ} f(a)=N^{-1}\widehat{f}(0)\ .$$
Cette définition ne dépend pas de $N$ dans le sens où si $M=NP$ et si $f$
est une fonction sur $\ZZ/M\ZZ$ qui se factorise par $\ZZ/P\ZZ$,
on a
$$\int_{\ZZ/M\ZZ} f d \beta_h= \int_{\ZZ/P\ZZ} f d \beta_h$$
car
\begin{equation*}
\begin{split}
 M^{h-1}\sum_{a\in \ZZ/M\ZZ} f(a) \frac{B_h(\langle\frac{a}{M}\rangle)}{h}
 &=
M^{h-1}\sum_{\substack{a \bmod P\\\ b \bmod N}} f(a + P b) \frac{B_h(\langle\frac{a+ P b}{M}\rangle)}{h}
\\&=
M^{h-1}\sum_{a \bmod P} f(a) \sum_{b \bmod N} \frac{B_h(\langle\frac{a}{NP}+ \frac{b}{N}\rangle)}{h}
\\&=
P^{h-1}\sum_{a \in \ZZ/P\ZZ} f(a) \frac{B_h(\langle\frac{a}{P}\rangle)}{h}\ .
\end{split}
\end{equation*}
On note $\int f(x) d\beta_h(x) =\int_{\ZZ/N\ZZ} f(x) d\beta_h(x)$
s'il n'y a pas d'ambiguité.
Pour $r\in \ZZ/N\ZZ$, $$\beta_h(r)= \int \charun_{r\bmod N\ZZ} d \beta_h$$
est la valeur de la distribution $\beta_h$ sur la fonction indicatrice
$\charun_{r\bmod N\ZZ}$ de $r$ dans $\ZZ/N\ZZ$.
Les relations $B_h(1-r)=(-1)^{h} B_h(r)$ pour $h\geq 0$ impliquent que
$\beta_h(1-r)=(-1)^{h} \beta_h(r)$ et donc que
\begin{equation}\label{eq:beta}
\int g^- d\beta_h = (-1)^h\int g d\beta_h
\ .\end{equation}
Cela est clair pour $h\neq 1$ car $B_h(1)=B_h(0)$.
Pour $h= 1$,
\begin{equation*}
\begin{split}
\beta_1(1-r)&= -B_1(1-r)= B_1(r) = -\beta_1(r) \text{ si $0< r <1$}\\
\beta_1(1)&= 1/2-1/2=0 = -\beta_1(0)\ .
\end{split}
\end{equation*}
Si $g$ est une fonction sur $\NN$, on pose
$$ L(s, g) = \sum_{n>0} \frac{g(n)}{n^s}\ .$$
Cette définition s'étend à une fonction
sur $\ZZ/N\ZZ$ par relèvement de manière naturelle.
Dans ce cas, $L(g,s)$ converge pour $\re(s) > 1$ et
se prolonge en une fonction méromorphe sur le plan complexe ayant
au plus un pôle simple en $s=1$ de résidu $\frac{\widehat g(0)}N$.

\begin{prop}\label{valeurzeta}
Si $g$ est une fonction sur $\ZZ/N\ZZ$,
\begin{equation*}
\begin{split}
\int_{\ZZ/N\ZZ} g d\beta_h &= L(1-h,g) \text{ pour $h> 1$}
\\
\int_{\ZZ/N\ZZ} g d\beta_1 &= L(0,g) + \frac{1}{2} g(0)\ .
\end{split}
\end{equation*}
\end{prop}
La proposition est classique dans le cas où $h > 1$,
voir \cite{wittaker}, p. 271 pour le cas $h=1$.

Nous aurons besoin plus tard de la distribution $\beta'_{h-1}$
définie par
\begin{equation*}
\begin{split}
\int_{\ZZ/N\ZZ} g d\beta'_{h-1} &=
\frac{1}{-2i\pi}L'(2-h,\widehat{g}^- + (-1)^h \widehat{g})
  \ .
\end{split}
\end{equation*}
Elle vérifie
\begin{equation}
\int g^- d\beta'_{h-1} =
(-1)^h\int g d\beta'_{h-1}
\ .\end{equation}

\subsubsection{Séries d'Eisenstein}
\label{sub:eisenstein}
Soit $k$ un entier $\geq 2$. Si $f$ appartient à $\Fonct{N}{\CC}$,
la \textsl{série d'Eisenstein associée à $f$} est définie par
$$\eisen{k}{f}(z)= \sum_{(c,d)\in \ZZ^2 }^{\sprime} f(c,d) (cz+d)^{-k}$$
où le $'$ signifie que l'on somme pour $k>2$ sur les couples $(c,d)$ non nuls
et pour $k=2$ par la sommation d'Eisenstein
\footnote{Pour $k=2$, on aurait pu la définir
comme la valeur en $s=0$ du prolongement analytique de
$\sum_{(c,d)\in \ZZ^2-\{0,0\} } f(c,d) (cz+d)^{-k}| cz+d|^{-2s}$,
ce que fait Kronecker, voir \cite{weil}.}
$$\eisen{k}{f}(z)= \sum_{c\in \ZZ} \sum_{d\in \ZZ'_c}f(c,d) (cz+d)^{-k}$$
où $\ZZ'_c$ est égal à $\ZZ$ si $c$ est non nul et à $\ZZ-\{0\}$ si $c=0$.
Elle est holomorphe.
La proposition suivante est bien connue (pour $k=2$, on peut facilement
adapter l'exercice 1.2.8 de \cite{diamond}).
\begin{prop}
Pour $\gamma \in \Sl_2(\ZZ)$, on a
\begin{enumerate}
\item pour $k >2$,
$$ \eisen{k}{f} |_k \gamma = \eisen{k}{f|\gamma}$$
\item pour $k=2$,
$$ \left(\eisen{k}{f} + \frac{\widehat{f}(0)}{N^2}\frac{\pi}{\im(.)}\right)|_2 \gamma =
\eisen{k}{f|\gamma} + \frac{\widehat{f}(0)}{N^2}\frac{\pi}{\im(.)}\ .$$
\end{enumerate}
\end{prop}
Pour $f \in \Fonct{N}{\CC}$ et $\widehat{f}(0)=0$ dans le cas $k=2$, $\eisen{k}{f}$
est une forme modulaire de poids $k$ pour $\Gamma_f$.
Si $F$ est un sous-espace de $\Fonct{N}{\CC}$,
notons $\hypk{F}$ le sous-espace de $F$
vérifiant l'hypothèse supplémentaire que $f(0)=0$ lorsque $k=2$.
Soit
 $\Eis_k$ l'application $\hypk{\Fonct{N}{\CC}} \to M_k$ donnée
 par $$\Eis_k(f)=\frac{(k-1)!}{(-2i\pi)^k} N^{k-1}E_{k,\widehat{f}}\ .$$
Notons $\Eisk{k}{\CC}$ son image et $\Eisk{k,\Gamma}{\CC}$ l'image de
$\hypk{\Fonct{\Gamma}{\CC}}$ dans $M_k(\Gamma)$.

Nous avons fixé le niveau $N$ et ne l'avons pas fait intervenir dans la notation.
Cela est justifié par le lemme suivant.
\begin{lem}
\label{independanceN} Soient $M$ un multiple de $N$ : $M=NP$
et $f\in \hypk{\Fonct{N}{\CC}}$. Si $g$ est l'élément
de $\hypk{\Fonct{M}{\CC}}$ qui s'en déduit par composition avec la projection
$(\ZZ/M\ZZ)^2 \to (\ZZ/N\ZZ)^2$, on a l'égalité
$$\Eis_{k}(g)= \Eis_{k}(f)\ .$$
\end{lem}
\begin{proof}
On a $g(x+Px',y+Py')=g(x,y)$.
On vérifie que
\begin{equation*}
\widehat{g}(x,y)=
\begin{cases}
P \widehat{f}(\frac{x}{P},\frac{y}{P}) &\text{si $P$ divise $x$ et $y$}\\
0 &\text{sinon.}
\end{cases}
\end{equation*}
On en déduit que
\begin{equation*}
\Eis_k(g)=\frac{(k-1)!}{(-2i\pi)^k} M^{k-1}E_{k,\widehat{g}}
=\frac{(k-1)!}{(-2i\pi)^k}  M^{k-1}
{\sum_{c,d}}'\frac{P\widehat{f}(c,d)}{P^{k}(cz+d)^{k}}
=\Eis_k(f)\ .
\end{equation*}
\end{proof}

\begin{prop}
Soit $f \in \hypk{\Fonct{\Gamma}{\CC}}$.
Si $T_\ell$ est l'opérateur de Hecke
$\Gamma \begin{pmatrix}1&0\\0&\ell\end{pmatrix}\Gamma$, on a
$$ T_\ell(\Eis_{k}(f)) = \Eis_k(T_\ell^{(k)}(f))\ .$$
\end{prop}

\begin{proof}
Ce fait est bien connu et démontré dans \cite[Theorem 1.3.2
ou Proposition 2.4.7]{stevens}
au moins lorsque $k=2$. Nous en donnons une démonstration directe
pour $k \geq 3$ en partant de la définition de $\eisen{k}{f}$.
Calculons $\tilde{T}_\ell^{(k)}(f)$ tel que
$T_\ell E_{k,f}= E_{k,\tilde{T}_\ell^{(k)}(f)}$.

Supposons d'abord que $\ell \nmid N$. Choisissons $\ell' \in \ZZ$
tel que $\ell \ell'\equiv 1 \bmod N^2$.
Les matrices
$\epsilon_i=\begin{pmatrix}1&N i\\0&\ell\end{pmatrix}
=\begin{pmatrix}1&0\\0&\ell\end{pmatrix}\begin{pmatrix}1&N i\\0&1\end{pmatrix}$
pour $i=0,\cdots, \ell-1$
et
$$\epsilon_{l}=
\begin{pmatrix}1&0\\0&\ell\end{pmatrix}
\begin{pmatrix}\ell \ell' &N\\\frac{\ell \ell'-1}{N}&1\end{pmatrix}
= \begin{pmatrix}\ell\ell' &N\\\ell\frac{\ell \ell'-1}{N}&\ell\end{pmatrix}
$$
sont de déterminant $\ell$ et vérifient
$\Gamma \begin{pmatrix}1 &0\\0&\ell\end{pmatrix} \Gamma=\sqcup \Gamma \epsilon_i $.
On calcule l'opérateur de Hecke
$\Gamma \begin{pmatrix}1&0\\0&\ell\end{pmatrix}\Gamma$
à l'aide de ces matrices :
\begin{small}\begin{equation*}
\begin{split}
T_\ell(\eisen{k}{f})(z)
&=\ell^{k-1}\sum_{(c',d')\in \ZZ^2}^{\sprime}\left(\sum_{i=0}^{\ell-1}
\frac{f(c',d')}{(c'z+ \ell d' + N i c')^{k}}
+ \frac{f(c',d')}{(\ell (\ell' c' + \frac{\ell\ell'-1}{N} d')z + N c' + \ell d')^{k}}\right)\ .
\end{split}
\end{equation*}
\end{small}
On considère les deux systèmes
\begin{equation} \begin{cases} c=&c'\\
d=&\ell d'+ N i c' \\
i \in& \{0\cdots \ell-1\}\label{syst1}
\end{cases} \text{ donc } \begin{cases} c\equiv &c'\bmod N\\
d\equiv &\ell d'\bmod N
\end{cases}
\end{equation}
et
 \begin{equation} \begin{cases}
c&=\ell (\ell' c' + \frac{\ell\ell'-1}{N} d')\\
d&=N c' + \ell d' \label{syst2}
\end{cases} \text{ donc } \begin{cases}
c&\equiv c'\bmod N\\
d&\equiv \ell d' \bmod N \ .
\end{cases}
\end{equation}
Lorsque $\ell \nmid c$ (resp. $\ell \mid c$), le système \eqref{syst1}
(resp. le système \eqref{syst2}) a une solution unique.
En rassemblant ces contributions, on trouve donc
$$\ell^{k-1} \sum\ ^{'} f(c, \ell^{-1}d) (cz+d)^{-k}\ .$$
Lorsque $\ell \nmid c$, le système \eqref{syst2} n'a pas de solution.
Lorsque $\ell \mid c$, le système \eqref{syst1}
a exactement $\ell$ solutions pour $\ell \mid d$,
ce qui donne la contribution
$$\ell^{k-1} \sum\ ^{'} \ell f(\ell c, \ell \frac{d}{\ell}) (\ell cz+\ell d)^{-k}
=\sum\ ^{'} f(\ell c, d) (cz+d)^{-k}\ .
$$
On obtient donc que pour $\ell \nmid N$, $T_\ell E_{k,f}= E_{k,\tilde{T}_\ell^{(k)}(f)}$ avec
\begin{equation}
\label{hecketilde1}
\tilde{T}_\ell^{(k)}(f)(c,d)=f(\ell c,d) + \ell^{k-1}f(c,\ell^{-1}d)\ .
\end{equation}
Supposons maintenant que $\ell$ divise $N$. L'opérateur de Hecke
se calcule à l'aide des matrices
$\epsilon_i$ pour $i=0,\cdots, \ell-1$:
\begin{equation*}
\begin{split}
T_\ell(\eisen{k}{f})(z)=\ell^{k-1}\sum_{i= 0}^{\ell-1} \ell^{-k}\eisen{k}{f}
(\frac{z+Ni}{\ell})
=\ell^{k-1}\sum_{(c',d')}\ ^{'}\sum_{i=0}^{\ell-1} f(c',d') (c'z+ \ell d' + Ni c')^{-k}
\end{split}
\end{equation*}
Considérons le système
\begin{equation}\label{syst3} \begin{cases}
c&=c'\\
d&=\ell (d'+ \frac{N}{\ell} i c')\\
i&\in \{0,\cdots, \ell-1\}\ .
\end{cases}
\end{equation}
Nécessairement $\ell$ divise $d$.
Le système \eqref{syst3} s'inverse alors en
\begin{equation*} \begin{cases}
c'&=c\\
d'&=\frac{d}{\ell}-\frac{N}{l}i c'\\
i&\in \{0,\cdots, \ell-1\}\ .
\end{cases}
\end{equation*}
On a donc
\begin{equation*}
\sum_{\substack{(c',d') \in (\ZZ/N\ZZ)^2 \\
\text{solution de \eqref{syst3}}}}
f(c',d')=\begin{cases}
0 &\text{si $\ell \nmid d$}\\
\ell f(c, d/\ell) &\text{si $\ell \mid c, \ell \mid d$}
\\
\sum_{x\in \ZZ/N\ZZ, \ell x =d} f(c,x)
&\text{si $\ell \nmid c$}\ .
\end{cases}
\end{equation*}
En conclusion, pour $\ell \mid N$,
\begin{equation}
\label{hecketilde2}
\begin{split}
\tilde{T}_\ell^{(k)}(f)(c,d)&=
 f(\ell c,d) +\ell^{k-1}
\begin{cases}
\sum_{i=0}^{\ell-1}f(c,\frac{d}{\ell}-\frac{N}{\ell} i c)
 &
  \text{si $\ell\mid d$ et $\ell \nmid c$}
  \\
0 &
  \text{sinon}
\end{cases}
\\
&=
f(\ell c,d) +\ell^{k-1}
\begin{cases}
\sum_{\ell x =d}f(c,x) &\text{si $\ell \nmid c$}
  \\
0 & \text{sinon.}
\end{cases}
\end{split}
\end{equation}
Calculons maintenant $T_\ell^{(k)}(f)=
\widehat{\tilde{T}^{(k)}_\ell(\widehat{f})}$.
La transformée de Fourier de $g$ définie par
$g(x,y)=\widehat{f}(\ell x,y)$ est
$$\widehat{g}(x,y)= \sum_{\ell s =y} f(x, s)\ .$$
En effet,
\begin{equation*}
\begin{split}
\widehat{g}(x,y)&=
\frac{1}{N} \sum_{u,v} \widehat{f}(\ell u, v) \exp(2i\pi \frac{uy-vx}{N})
\\
&=
\frac{1}{N^2}\sum_{u,v,r,s} f(r,s) \exp(2i\pi \frac{rv-\ell s u}{N})
\exp(2i\pi \frac{uy-vx}{N})\\
&=
\frac{1}{N}\sum_{u,v,s} f(x,s) \exp(2i\pi \frac{(y-\ell s) u}{N})
\\
&=
\sum_{\ell s=y} f(x,s)\ .
\end{split}
\end{equation*}
Lorsque $\ell \nmid N$, on a donc
$\widehat{g}(x,y)=f(x,\ell^{-1}y)$
et la transformée de Fourier de
$(x,y) \to \widehat{f}(x,\ell^{-1}y)$
est $(x,y) \to f(\ell x,y)$.
On déduit alors de la formule \eqref{hecketilde1}
que pour $\ell \nmid N$
$$  T_\ell^{(k)}(f)(x,y)=f(x,\ell^{-1}y)+ \ell^{k-1} f(\ell x,y)\ .$$
Si $\ell \mid N$, la transformée de Fourier de la fonction $h$ définie par
\begin{equation*}
h(x,y)=\begin{cases}
 0 &\text{ si $\ell \mid x$}\\
 \sum_{\ell s=y} \widehat{f}(x, s) & \text{ si $\ell \nmid x$}
\end{cases}
\end{equation*}
est
$$
\widehat{h}(x,y)=f(\ell x,y)
- \frac{1}{\ell}\sum_{\ell t=\ell y} f(\ell x,t)
\ .$$
En effet,
\begin{equation*}
\begin{split}
\widehat{h}(x,y)
&=
\frac{1}{N}
\sum_{\substack{u,s\\ \ell\nmid u}}
\widehat{f}(u, s)\exp(2i\pi \frac{uy-\ell sx}{N})
\\
&=
\frac{1}{N^2}
\sum_{\substack{u,s\\ \ell\nmid u}}
\sum_{r,t}
f(r,t) \exp(2i\pi\frac{rs-tu}{N})
\exp(2i\pi \frac{uy-\ell sx}{N})
\\
&=
f(\ell x,y) -
\frac{1}{N^2}
\sum_{r,s,t,u}f(r,t)
 \exp(2i\pi\frac{(r-\ell x)s-(\ell t -\ell y)u}{N})
\\
&=f(\ell x,y)
-
\frac{1}{\ell}
\sum_{\substack{
\ell t=\ell y}}f(\ell x,t) \ .
\end{split}
\end{equation*}
On déduit alors de \eqref{hecketilde2} que pour $\ell \mid N$
$$
T_\ell^{(k)}(f)(x,y)= \sum_{\ell s=y} f(x,s) +
\ell^{k-1} f(\ell x, y) -\ell^{k-2}
\sum_{\ell t=\ell y} f(\ell x, t) \ .
$$
D'où la proposition.
\end{proof}
\subsubsection{Générateurs de $\Eisk{k,\Gamma}{\CC}$}
Suivant l'inspiration de  \cite{kubert}, nous allons maintenant
donner un ensemble de fonctions de
$\Fonct{\Gamma}{\QQ}$ engendrant l'espace des séries d'Eisenstein $\Eisk{k,\Gamma}{\CC}$.
Les questions de rationalité seront étudiées au paragraphe \ref{symbole}.
La proposition suivante est bien connue.
\begin{prop}[Relation de distribution]
\label{relationdist}
Soit $M$ un diviseur de $N$. Alors si $a \in M(\ZZ/N\ZZ)^2$,
on a
\begin{equation}
\begin{split}
M^{k-2}\sum_{Mb=a} \Eis_{k}(\charun_{b}) &= \Eis_{k}(\charun_{a})
\ .
\end{split}
\end{equation}
\end{prop}

\begin{proof}
Il s'agit de montrer que
\begin{equation}
\begin{split}
M^{k-2}\sum_{Mb=a} E_{k,\widehat{\charun_{b}\ }} &= E_{k,\widehat{\charun_{a}\ }}
\ .
\end{split}
\end{equation}
L'équation
$Mb=a$ a une solution dans $(\ZZ/N\ZZ)^2$ si et seulement si $a$
appartient à $M(\ZZ/N\ZZ)^2$. Posons $a=Mb_0$ avec $b_0\in (\ZZ/N\ZZ)^2$.
Les solutions de l'équation $Mb=a$
sont de la forme $b_0 + \frac{N}{M} h$ avec $h \bmod M\ZZ^2$.
On a alors en écrivant $b_0=(b_1,b_2)$, $h=(h_1,h_2)$
\begin{equation*}
\begin{split}
\sum_{Mb=a} E_{k,\widehat{\charun_{b}\ }}(z)
&=
\frac{1}{N}\sum_{(c,d)\in \ZZ^2}\frac{\exp(2i\pi\frac{b_2 d - b_1 c}{N})}{(cz+d)^k}
\sum_{(h_1,h_2) \bmod M\ZZ^2} \exp(2i\pi\frac{h_2 d - h_1 c}{M})\\
&=\frac{M^2}{N}\sum_{(c,d)\in M\ZZ^2}\frac{\exp(2i\pi\frac{b_2 d - b_1 c}{N})}{(cz+d)^k}
=\frac{M^{2-k}}{N}\sum_{(c,d)\in \ZZ^2}\frac{\exp(2i\pi\frac{Mb_2 d - Mb_1 c}{N})}{(cz+d)^k}
\\
&=\frac{M^{2-k}}{N}\sum_{(c,d)\in \ZZ^2}\frac{\exp(2i\pi\frac{a_2 d - a_1 c}{N})}{(cz+d)^k}
=M^{2-k}E_{k,\widehat{\charun_{a}\ }}(z) \ .
\end{split}
\end{equation*}
\end{proof}
Modulo l'identification de $\QQ[(\ZZ/N\ZZ)^2]$ avec
$\Fonct{N}{\CC}$, l'application $\Eis_k$ est la restriction à $(\ZZ/N\ZZ)^2$ d'une distribution universelle de poids $k-2$
au sens de D. Kubert (\cite{kubert}).
\begin{prop}Soit $(\ZZ/N\ZZ)^2_{prim}$ le sous-ensemble de $(\ZZ/N\ZZ)^2$ formé des
éléments d'ordre $N$.
Pour $k\neq 2$, l'image par $\Eis_k$ des fonctions
à support dans $(\ZZ/N\ZZ)^2_{prim}$ engendre le $\CC$-espace vectoriel $\Eisk{k,\Gamma(N)}{\CC}$.
\end{prop}
La proposition est une conséquence de l'appendice de \cite{kubert}.
Pour $k=2$, le poids de la distribution universelle est 0 et Kubert construit
explicitement dans ce cas aussi un système de générateurs
(théorème 1.8 de \cite{kubert}). C'est en programmant
l'espace $\EiskQ{k,\Gamma_0(N)}$ que nous nous sommes rendus compte que
pour $k=2$, l'image par $\Eis_k$ des fonctions à support dans $(\ZZ/N\ZZ)^2_{prim}$
n'engendrait pas $\Eisk{k,\Gamma_0(N)}{\CC}$.
Dans la suite de ce paragraphe, nous allons construire un système de générateurs
valable pour tout poids dans le cas où $\Gamma=\Gamma_0(N)$.

Commençons par donner une description des orbites sous $\Gamma_0(N)$.
À tout couple $(x,y)\in (\ZZ/N\ZZ)^2$, on associe le triplet suivant
$[q_1,q_2,u]$ avec
$$q_1=\pgcd(x,N), \quad q_2=\pgcd(y,q_1), \quad
u= \frac{x}{q_1}\frac{y}{q_2} \bmod \pgcd(\frac{N}{q_1},\frac{q_1}{q_2})\ .$$
Comme $\pgcd(\frac{x}{q_1},\frac{N}{q_1})=\pgcd(\frac{y}{q_2},\frac{q_1}{q_2})=1$,
$u$ appartient à $(\ZZ/\pgcd(\frac{N}{q_1},\frac{q_1}{q_2})\ZZ)^*$.
Notons $\Orb{N}$ l'ensemble des triplets $[q_1,q_2,u]$ avec
\begin{equation}
q_2 \mid q_1\mid N, \quad u \in (\ZZ/\pgcd(\frac{N}{q_1},\frac{q_1}{q_2})\ZZ)^*\ .
\end{equation}

\begin{lem}
\begin{enumerate}
\item L'ensemble des orbites de $(\ZZ/N\ZZ)^2$ sous l'action de $\Gamma_0(N)$ est
en bijection avec $\Orb{N}$.
\item Un représentant dans $(\ZZ/N\ZZ)^2$ de l'orbite $[q_1,q_2,u]$
est $(q_1,q_2v)$ où $v$ est un représentant de $u$
dans $\ZZ/\frac{N}{q_2}\ZZ$ premier à $\frac{q_1}{q_2}$.
\item
Le cardinal de l'orbite associée à $[q_1,q_2,u]$ est
$\frac{\frac{N}{q_1}\varphi(\frac{N}{q_1}) \varphi(\frac{q_1}{q_2})}{
\varphi(\pgcd(\frac{N}{q_1},\frac{q_1}{q_2}))}$.
\end{enumerate}
\end{lem}
\begin{proof}
On voit facilement que le triplet $[q_1,q_2,u]$ ne dépend que de l'orbite de $(x,y)$ sous l'action de $\Gamma_0(N)$. Réciproquement, montrons que si $(x,y)$ et $(x',y')$ ont même image $[q_1,q_2,u]$, il existe
$\gamma \in \Gamma_0(N)$ tel que $(x,y)=(x',y')\gamma$.
On a $x=q_1 v$, $y=q_2 w$, $x'=q_1 v'$, $y'=q_2  w'$ avec
$v w \equiv v'w' \bmod \pgcd(\frac{N}{q_1},\frac{q_1}{q_2})$
et $v'$ et $w'$ inversibles modulo $\pgcd(\frac{N}{q_1},\frac{q_1}{q_2})$.
Donc
$v/v' \equiv w/w' \bmod \pgcd(\frac{N}{q_1},\frac{q_1}{q_2})$.
Par le théorème chinois, il existe $a \in (\ZZ/N\ZZ)^*$ tel que
\begin{equation*}
a \equiv\begin{cases} v/v' & \bmod \frac{N}{q_1}\\
 w/w'& \bmod \frac{q_1}{q_2} \ .
 \end{cases}
\end{equation*}
Soit $a'$ un inverse de $a \bmod N$.
Comme $y'a'-y$ est divisible par $q_1$ et que $v'$ est inversible
modulo $N/q_1$, il existe $b\in \ZZ$ tel que
$b\equiv\frac{y'a' -y}{q_1 v'^{-1}}\bmod N$.
On a alors $(x,y)=(x',y') \begin{pmatrix}a&b\\0&a^{-1}\end{pmatrix}\bmod N$.
Ce qui démontre l'assertion (1).

L'assertion (2) se déduit de la construction.
Calculons le cardinal de l'orbite associée à
$[q_1,q_2,u]$. Les éléments de l'orbite sont de la forme $(q_1 v, q_2 w)$.
Comme $v$ est défini modulo $\frac{N}{q_1}$
et est inversible modulo $\frac{N}{q_1}$, on a $\varphi(\frac{N}{q_1})$
choix pour $v$. Comme $w$ est défini modulo $\frac{N}{q_2}$, que $w \bmod \frac{q_1}{q_2}$ est
premier à $\frac{q_1}{q_2}$ et qu'on a la relation $vw=u$, on a
$\frac{N}{q_1}\frac{\varphi(\frac{q_1}{q_2})}{\varphi(\pgcd(\frac{N}{q_1},\frac{q_1}{q_2}))}$
choix pour $w$. Le cardinal de l'orbite associée à $[q_1,q_2,u]$ est donc bien
$\frac{\frac{N}{q_1}\varphi(\frac{N}{q_1}) \varphi(\frac{q_1}{q_2})}{
\varphi(\pgcd(\frac{N}{q_1},\frac{q_1}{q_2}))}$.
Pour vérification,
$$\sum_{q_2\mid q_1\mid N, u\in (\ZZ/\pgcd(\frac{N}{q_1},\frac{q_1}{q_2})\ZZ)^*}
\frac{\frac{N}{q_1}\varphi(\frac{N}{q_1}) \varphi(\frac{q_1}{q_2})}{
\varphi(\pgcd(\frac{N}{q_1},\frac{q_1}{q_2}))}
= N^2\ .$$
\end{proof}
On identifie désormais $[q_1,q_2,u]$ à l'orbite correspondante.

\begin{prop}\label{gen}
Soit $\VN$ le sous-ensemble de $\Fonct{\Gamma_0(N)}{(\ZZ/N\ZZ)^2}$ formé des
fonctions indicatrices des orbites
$[\frac{N}{\pgcd(d,\frac{N}{d})},\frac{\frac{N}{d}}{\pgcd(d,\frac{N}{d})},u]$
pour $d$ diviseur de $N$ et $u\in(\ZZ/\pgcd(d,\frac{N}{d})\ZZ)^*$.
Alors l'ensemble des $\Eis_k(f)$ pour $f \in \VN$ pour $k>2$
(resp. $f \in \VN-\{\charun_{(N,N)}\}$ pour $k=2$)
forme une base de $\Eisk{k,\Gamma_0(N)}{\CC}$.
\end{prop}
\begin{rem}
\begin{enumerate}
\item Le cardinal de $\VN$ est égal au nombre $\sum_{d \mid N} \varphi(\pgcd(d,\frac{N}{d}))$
de pointes de $\Gamma_0(N)$. Explicitement, on a une bijection
de l'ensemble des pointes dans $\VN$ définie de la manière suivante :
si $r=\frac{x}{y}$ est un rationnel donné de manière irréductible,
on pose $d = \pgcd(y, N)$, $y = dv$ et l'élément
correspondant dans $\VN$ est la fonction indicatrice de
$[\frac{N}{\pgcd(d, N/d)}, \frac{N}{d\pgcd(d, N/d)}, xv \bmod \pgcd(d, N/d)]$.
Il ne dépend que de la classe de $r$ modulo $\Gamma_0(N)$.
\item Le cardinal de l'orbite
$[\frac{N}{\pgcd(d,\frac{N}{d})},\frac{\frac{N}{d}}{\pgcd(d,\frac{N}{d})}, u]$
est
$$\frac{ \pgcd(d,\frac{N}{d})\varphi(\pgcd(d,\frac{N}{d})) \varphi(d)}{\varphi(\pgcd(d,\frac{N}{d}))}=
\pgcd(d,\frac{N}{d})\varphi(d) \leq \pgcd(d,\frac{N}{d})d\leq N
\ . $$
\item Lorsque le poids $k$ est égal à 2, il faut enlever l'orbite $[N,N,1]$
car sa fonction indicatrice $\charun_{(N,N)}=\charun_{(0,0)}$
n'est pas nulle en $(0,0)$. Dans tous les cas, le cardinal de $\VN$ pour $k>2$
(resp. de $\VN-\{\charun_{(N,N)}\}$ pour $k=2$)
est égal à la dimension de $\Eisk{k,\Gamma_0(N)}{\CC}$.
\end{enumerate}
\end{rem}
Pour simplifier les notations, nous travaillerons dans la suite
dans le $\QQ$-espace vectoriel $\QQ[\Orb{N}]$ de base $\Orb{N}$ modulo les relations
plutôt que dans l'espace $\Fonct{\Gamma_0(N)}{(\ZZ/N\ZZ)^2}$ et noterons
$\vN$ l'ensemble des orbites
$[\frac{N}{\pgcd(d,\frac{N}{d})},\frac{\frac{N}{d}}{\pgcd(d,\frac{N}{d})},u]$.

Donnons deux cas particuliers.
\begin{enumerate}
\item
Pour $N$ sans facteurs carrés, $\vN$ est formé des
orbites $[N,d,1]$ avec $d$ divisant $N$, qui ont $\varphi(\frac Nd)$ éléments.

\item Pour $N=p^n$ puissance d'un nombre premier $p$,
$\vN$ est formé des orbites
$[p^s,p^{2s-n},u]$ et $[p^s,1,u]$ pour $\frac n2\leq s\leq n$ avec $u \in (\ZZ/p^{n-s}\ZZ)^*$
(voir lemme \ref{techV}).
L'orbite $[p^s,p^{2s-n},u]$ a $\varphi(p^{2n-2s})$ éléments et l'orbite $[p^s,1,u]$ a $\varphi(p^n)$ éléments.

Visualisons $\vN$ pour $n=5$. Dans le tableau suivant,
sur chaque colonne, $\frac{q_1}{q_2}$ est constant et sur chaque ligne,
$q_1$ (et donc $\frac{N}{q_1}$) est constant.

\begin{scriptsize}
$$\begin{matrix}\color{red}{[p^5,p^5,1]}& [p^5,p^4,1]&[p^5,p^3,1]&[p^5,p^2,1]&[p^5,p,1]&\color{red}{[p^5,1,1]}\\
[p^4,p^4,1]& \color{red}{[p^4,p^3,u\in(\ZZ/p\ZZ)^\ast]}&[p^4,p^2,u\in(\ZZ/p\ZZ)^\ast]&[p^4,p,u\in(\ZZ/p\ZZ)^\ast]&\color{red}{[p^4,1,u\in(\ZZ/p\ZZ)^\ast]}\\
[p^3,p^3,1]& [p^3,p^2,u\in(\ZZ/p\ZZ)^\ast]&\color{red}{[p^3,p,u\in(\ZZ/p^2\ZZ)^\ast]}&\color{red}{[p^3,1,u\in(\ZZ/p^2\ZZ)^\ast]}\\
[p^2,p^2,1]& [p^2,p,u\in(\ZZ/p\ZZ)^\ast]&[p^2,1,u\in(\ZZ/p^2\ZZ)^\ast]\\
[p,p,1]&[p,1,u\in(\ZZ/p\ZZ)^\ast]\\
[1,1,1]
\end{matrix}
$$
\end{scriptsize}
\end{enumerate}

Avant de démontrer la proposition \ref{gen}, nous allons démontrer quelques
lemmes. Donnons une description locale de $\vN$.
\begin{lem}\label{techV}
Une orbite $[q_1,q_2,u]$ pour $q_2|q_1|N$ et
$u \in (\ZZ/\pgcd(\frac{N}{q_1},\frac{q_1}{q_2})\ZZ)^*$ appartient à
$\vN$ si et seulement si pour tout nombre premier $p \mid N$,
\begin{equation*}
\begin{cases}
\ord_p(\frac{N}{q_1})=\ord_p(\frac{q_1}{q_2})
\\
\text{ou}\\
\ord_p(\frac{N}{q_1}) \leq \ord_p(\frac{q_1}{q_2}) \text{ et } \ord_p(q_2)=0 \ .
\end{cases}
\end{equation*}
\end{lem}
Remarquons qu'en particulier, $q_1$ est nécessairement divisible par le radical de $N$.
\begin{proof}
Une orbite $[q_1,q_2,u]$ appartient à $\vN$ si et seulement s'il existe $d | N$ tel que
si $a=\pgcd(\frac{N}{d},d)$, $q_1=\frac{N}{a}$, $q_2=\frac{N}{da}$.
Soit $[q_1,q_2,u]$ appartenant à $\vN$. Pour tout $p | N$,
on a
\begin{equation*}
\begin{split}
\ord_p(a)&=\min( \ord_p(N)-\ord_p(d),\ord_p(d))\\
\ord_p(q_1)&= \max( \ord_p(N)-\ord_p(d),\ord_p(d)) \ .
\end{split}
\end{equation*}
Donc, si $\ord_p(q_2) > 0$,
$\ord_p(d)< \ord_p(q_1)$.
Par la deuxième équation, $\ord_p(q_1)=\ord_p(N)-\ord_p(d)$.
Donc
$$2 \ord_p(q_1)-\ord_p(q_2)=\ord_p(N)\ .$$
Si $\ord_p(q_2)=0$, on a $\ord_p(N)=\ord_p(da)\leq 2 \ord_p(d)
=2 \ord_p(q_1)$.

Réciproquement, soit $[q_1,q_2,u]$ une orbite et
supposons que pour tout $p \mid N$,
$\ord_p(N)= 2 \ord_p(q_1)-\ord_p(q_2)$
si $\ord_p(q_2)>0$ et
que $\ord_p(N)\leq 2 \ord_p(q_1)$ si $\ord_p(q_2)=0$.
Soit $d=\frac{q_1}{q_2}$ et $a=\pgcd(\frac{N}{d},d)$.
Pour $p \mid N$ tel que $\ord_p(q_2)>0$,
\begin{equation*}
\begin{split}
\ord_p(q_1 a)&= \ord_p(q_1) +
\min (\ord_p(N) - \ord_p(q_1) +\ord_p(q_2), \ord_p(q_1) -\ord_p(q_2))
\\
&=\min (\ord_p(N)+ \ord_p(q_2), 2\ord_p(q_1) -\ord_p(q_2))= \ord_p(N) \ .
\end{split}
\end{equation*}
Pour $p\mid N$ tel que $\ord_p(q_2)=0$,
\begin{equation*}
\begin{split}
\ord_p(q_1 a)
&=\min (\ord_p(N) , 2\ord_p(q_1))=\ord_p(N) \ .
\end{split}
\end{equation*}
Donc $N=q_1 a$ et $[q_1,q_2,u]=[\frac{N}{a},\frac{N}{da}, u]$ appartient à $\vN$.
\end{proof}

Traduisons maintenant les relations \ref{relationdist} relatives à un nombre premier
en des distributions dans $\QQ[\Orb{N}]$.
On écrira $a \sim b$ dans le quotient de $\QQ[\Orb{N}]$
par les relations obtenues de poids $W=k-2$.
\begin{lem} \label{distrib1}
Soit $p$ un nombre premier divisant $N$ et
$[q_1,q_2,u]$ une orbite telle que $p \mid \frac{N}{q_1}$ et $p\mid q_2$.
Alors
$$
[q_1,q_2,u]\sim p^{W} \sum_{\pi(v)= u} [\frac{q_1}{p},\frac{q_2}{p},v]
$$
où $v \in (\ZZ/\pgcd(\frac{N p}{q_1},\frac{q_1}{q_2})\ZZ)^*$
et
$$\pi: (\ZZ/\pgcd(\frac{N p}{q_1},\frac{q_1}{q_2})\ZZ)^*
\to (\ZZ/\pgcd(\frac{N}{q_1},\frac{q_1}{q_2})\ZZ)^*$$
est la projection naturelle.
En particulier, si $\ord_p(\frac{q_1}{q_2})\leq \ord_p(\frac{N}{q_1})$,
$$
[q_1,q_2,u]\sim p^{W} [\frac{q_1}{p},\frac{q_2}{p},u]
\ . $$
\end{lem}

\begin{proof}Le nombre premier $p$ divise $q_1$ et $q_2$.
L'orbite de $p(\frac{q_1}{p},\frac{q_2}{p} v)$ pour
$v\in (\ZZ/\pgcd(\frac{N p}{q_1},\frac{q_1}{q_2})\ZZ)^*$ est donnée par
$[q_1,q_2, u]$ avec
$u \equiv v \bmod (\ZZ/\pgcd(\frac{N}{q_1},\frac{q_1}{q_2})\ZZ)^*$.
Ce qui démontre la première partie.
Lorsque $\ord_p(\frac{q_1}{q_2}) \leq \ord_p(\frac{N}{q_1})$,
$\pgcd(\frac{N p}{q_1},\frac{q_1}{q_2})=\pgcd(\frac{N}{q_1},\frac{q_1}{q_2})$
et il n'y a qu'une solution en $v$ ce qui implique le cas particulier.
\end{proof}
\begin{rem}Si $p\mid \frac{N}{q_1}$, $p\nmid \frac{q_1}{q_2}$ et $p\mid q_1$,
on a
$\ord_p(\frac{q_1}{q_2}) \leq \ord_p(\frac{N}{q_1})$
et par le lemme
$$
[q_1,q_2,u]\sim p^{W} [\frac{q_1}{p},\frac{q_2}{p},u]\ .
$$
\end{rem}

\begin{lem} \label{distrib2}
Soit $p$ un nombre premier divisant $N$ et
$[q_1,q_2,u]$ une orbite telle que $p \nmid \frac{N}{q_1}$,
$p \mid \frac{q_1}{q_2}$ et $p\mid q_2$.
Alors,
$$
 [q_1,q_2,u]\sim p^{W} \sum_{\pi(v)= u} [\frac{q_1}{p},\frac{q_2}{p},v]
+ p^{W}[q_1,\frac{q_2}{p},p^{-1}u]
$$
où $\pi: (\ZZ/\pgcd(\frac{N p}{q_1},\frac{q_1}{q_2})\ZZ)^*
\to (\ZZ/\pgcd(\frac{N}{q_1},\frac{q_1}{q_2})\ZZ)^* $
est la projection naturelle
et où $p^{-1}u$ appartient à $(\ZZ/\pgcd(\frac{N}{q_1},\frac{q_1}{q_2})\ZZ)^*$.
\end{lem}

\begin{proof}
Le premier terme s'obtient comme précédemment.
On déduit le deuxième terme de ce que
$\pgcd(\frac{N}{q_1},p\frac{q_1}{q_2})
=
\pgcd(\frac{N}{q_1},\frac{q_1}{q_2})
$
est premier à $p$ et de ce que $\pgcd(pq_1,N)=\pgcd(q_1,N)$.
\end{proof}

\begin{lem} \label{distrib3}
Soit $p$ un nombre premier divisant $N$ et
$[q_1,q_2,u]$ une orbite telle que $p \nmid \frac{N}{q_2}$.
Alors,
$$
[q_1,q_2,u]\sim p^{W} \left([\frac{q_1}{p},\frac{q_2}{p},u] +
[q_1,\frac{q_2}{p}, p^{-1}u] +[q_1,q_2,p^{-2}u]
\right)
$$
avec $u \in (\ZZ/\pgcd(\frac{N}{q_1},\frac{q_1}{q_2})\ZZ)^*$.
De manière équivalente, si
$p || \frac{N}{q_1}$ et $p\nmid \frac{q_1}{q_2}$,
on a $$
p^{W} [q_1,q_2,u] \sim [pq_1,pq_2,u] -
p^{W}[pq_1,pq_2,p^{-2}u] - [pq_1,q_2, p^{-1}u])\ .
$$
\end{lem}

\begin{proof}
Si $p \nmid \frac{N}{q_2}$, $\ord_p(\frac{q_1}{q_2})=\ord_p(\frac{N}{q_1})=0$
et $\pgcd(\frac{N}{q_1},\frac{q_1}{q_2})$ est premier à $p$.
L'étude de la multiplication par $p$ donne facilement le résultat
comme précédemment.
\end{proof}

\begin{proof}[Démonstration de la proposition \ref{gen}]

Nous allons montrer que modulo les relations de distribution,
toute orbite appartient à $\QQ[\vN]$.
Choisissons un élément $[q_1,q_2,u]$ de $\Orb{N}$ n'appartenant pas
à $\QQ[\vN]/\sim$ et minimisant l'entier
\[M(q_1,q_2)=q_2\left(\frac{\frac{N}{q_1}}{\pgcd(\frac{N}{q_1},\frac{q_1}{q_2})}\right)^2\ .\]
Soit $p$ un nombre premier divisant $N$.
\begin{enumerate}
\item Si $p \mid \frac{N}{q_1}$ et $p\nmid\frac{q_1}{q_2}$:
 \begin{enumerate}
 \item si $p^2 \mid \frac{N}{q_1}$:
d'après le lemme \ref{distrib1} appliqué à $[pq_1,pq_2,u]$,
on a
\[[q_1,q_2,u]\sim p^{-W}[pq_1,pq_2,u]\] avec $M(pq_1,pq_2)=M(q_1,q_2)/p$.
\item Si $p\ddivides \frac{N}{q_1}$:
d'après le lemme \ref{distrib3} appliqué à $[pq_1,pq_2,u]$, on a
\[[pq_1,pq_2,u]\sim p^{W}\left([q_1,q_2,u]+ [pq_1,q_2,p^{-1}u]+[pq_1,pq_2,p^{-2}u]\right)\]
avec $M(pq_1,pq_2)=M(q_1,q_2)/p$ et $M(pq_1,q_2)=M(q_1,q_2)/p^2$.
\end{enumerate}
\item Si $p\mid \frac{N}{q_1}$,
$\ord_p(\frac{N}{q_1})<\ord_p(\frac{q_1}{q_2})$ et $p\mid q_2$:
d'après le lemme \ref{distrib1}, on a
\[[q_1,q_2,u]\sim p^{W}\sum_{\pi(v)=u} [\frac{q_1}{p}, \frac{q_2}{p},v]\]
avec $M(\frac{q_1}{p},\frac{q_2}{p})=M(q_1,q_2)/p$.
\item Si $p\mid \frac{N}{q_1}$,
$\ord_p(\frac{N}{q_1})>\ord_p(\frac{q_1}{q_2})>0$:
d'après le lemme \ref{distrib1} on a
\[[q_1,q_2,u] \sim p^{-W}[pq_1,pq_2,u]\]
avec $M(pq_1,pq_2)=M(q_1,q_2)/p$.
\item Si $p\nmid\frac{N}{q_1}$, $\ord_p(\frac{q_1}{q_2})>0$ et $p\mid q_2$:
d'après le lemme \ref{distrib2}, on a
\[[q_1,q_2,u]\sim p^{W}\sum_{\pi(v)=u}[\frac{q_1}{p},\frac{q_2}{p},v] +
  p^{W} [q_1,\frac{q_2}{p},p^{-1}u]\]
avec $M(q_1,\frac{q_2}{p})=M(q_1,q_2)/p$ et $M(\frac{q_1}{p},
\frac{q_2}{p})=M(q_1,q_2)/p$.

\end{enumerate}

Dans tous les cas, on a une contradiction avec la minimalité de $M(q_1,q_2)$.
Pour chaque $p$, on est donc dans une des situations suivantes:
\begin{enumerate}
\item si $p\nmid \frac{N}{q_1}$, $p\nmid \frac{q_1}{q_2}$ ou $p\nmid q_2$,
\item si $p\mid \frac{N}{q_1}$, on a $p \mid \frac{q_1}{q_2}$ d'après le cas (1),
puis $\ord_p(\frac{N}{q_1})\leq \ord_p(\frac{q_1}{q_2})$ d'après le cas (3),
et enfin,
($\ord_p(\frac{N}{q_1})< \ord_p(\frac{q_1}{q_2})$ et $p\nmid q_2$)
ou
$\ord_p(\frac{N}{q_1})= \ord_p(\frac{q_1}{q_2})$
par le cas (2).
\end{enumerate}
Autrement dit, $[q_1,q_2,u]$ vérifie une des conditions suivantes
\begin{itemize}
\item $p\nmid \frac{N}{q_1}$ et $p\nmid \frac{q_1}{q_2}$,
\item $p\nmid \frac{N}{q_1}$ et $p\nmid q_2$,
\item $p\mid \frac{N}{q_1}$, $p\mid \frac{q_1}{q_2}$,
$\ord_p(\frac{N}{q_1})\leq\ord_p(\frac{q_1}{q_2})$ et $p\nmid q_2$,
\item $p\mid \frac{N}{q_1}$, $p\mid \frac{q_1}{q_2}$ et
$\ord_p(\frac{N}{q_1})=\ord_p(\frac{q_1}{q_2})$
\end{itemize}
ce qui se résume en
\begin{itemize}
\item $\ord_p(\frac{N}{q_1})=\ord_p(\frac{q_1}{q_2})$
\item $\ord_p(\frac{N}{q_1})\leq\ord_p(\frac{q_1}{q_2})$ et
 $p\nmid q_2$.
\end{itemize}
Donc $[q_1,q_2,u]\in \vN$ par le lemme \ref{techV}, ce qui est contradictoire
avec l'hypothèse du départ. Ceci termine la démonstration de la
proposition \ref{gen}.
\end{proof}

\subsubsection{Symbole de Eisenstein-Dedekind}
\label{symbole}
\begin{thm}[Stevens]
L'application $\Psi_k=\Per\circ \Eis_k$
définit un homomorphisme de $\Sl_2(\ZZ)$-modules
$\hypk{\Fonct{N}{\CC}}\to \Hom(\KKK_0,V_k(\CC))$.
Il est déterminé par les formules
\begin{equation*}
\begin{split}
\Psi_k(f)(\{\infty,0\})
&=
\sum_{j=0}^{k-2} (-1)^j \binom{k-2}{j}
\left(\int_{(\ZZ/N\ZZ)^2} f^- \dd{\beta_{k-1-j}}{\beta_{j+1}}\right)
 x^j y^{k-2-j}\\
 &\quad +\left(\int_{(\ZZ/N\ZZ)^2} \widehat{f}
\dd{\beta'_{k-1}}{\beta_0} \right)x^{k-2}
-
\left(\int_{(\ZZ/N\ZZ)^2}
\widehat{f} \dd{\beta_0}{\beta'_{k-1}}\right)y^{k-2}
\\
\Psi_k(f)([0,r]_\infty)
&=\frac{1}{k-1}\left(\int_{(\ZZ/N\ZZ)^2} f^- \dd{\beta_{k}}{\beta_{0}}\right)
\frac{(rx + y)^{k-1} - y^{k-1}}{x} \text{pour $r\in\QQ$}\ .
\end{split}
\end{equation*}
\end{thm}
\begin{proof} C'est une conséquence du théorème
\ref{thm:fond} et de la proposition \ref{psi}.
\end{proof}
On déduit du lemme \ref{independanceN} que $\Psi_k$ est compatible
aux applications naturelles
$\hypk{\Fonct{M}{\CC}}\to \hypk{\Fonct{N}{\CC}}$ pour $N\mid M$.

\begin{rem} Dans le cas où $N=1$, on retrouve les formules d'Haberland
(\cite[Satz 3]{haberland}). Le cocycle d'Haberland $\rho_{Hab,k}$ est
donné par
\begin{equation*}
\begin{split}
\rho_{Hab,k}(\sigma^{-1})&=
-\frac{1}{2(k-1)!}\Psi_k(\charun)([\pi_{\infty}(0),\sigma\pi_{\infty}(0)])(-1,x)=
-\frac{1}{2(k-1)!}\Psi_k(\charun)(\{\infty,0\})(-1,x)\\
\rho_{Hab,k}(T^{-1})&=
-\frac{1}{2(k-1)!}\Psi_k(\charun)([\pi_{\infty}(0),T\pi_{\infty}(0)])(-1,x)=
-\frac{1}{2(k-1)!}\Psi_k(\charun)([0,1]_\infty)(-1,x)
\end{split}
\end{equation*}
avec $T=\smallmat{1&1\\0&1}$
(rappelons que $\Sl_2(\ZZ)$ est engendré par $\sigma$ et $T$).
\end{rem}
\begin{rem}
L'invariance par $\Sl_2(\ZZ)$ signifie que
\begin{equation*}
\begin{split}
\Psi_k(f|\gamma)= \Psi_k(f)|\gamma
\end{split}
\end{equation*}
autrement dit que pour $\delta \in \Xi_0$,
\begin{equation}
\label{invariance}
\Psi_k(f)(\gamma \delta)=\Psi_k(f|\gamma)(\delta)| \gamma^{-1}
\end{equation}
puisque $\left(\Psi_k(f)| \gamma\right)(\delta)= \Psi_k(f)(\gamma\delta)|\gamma$,
Pour $\epsilon=\smallmat{-1&0\\0&1}$, on a
\begin{equation}
\label{epsilon}
\Psi_k(f|\epsilon)= (-1)^{k-1}\Psi_k(f)|\epsilon\ .
\end{equation}
En effet, on déduit des formules explicites de $\Psi_k(f)$ que
\begin{equation*}
\begin{split}
\Psi_k(f|\epsilon)(\{\infty,0\})&=
  (-1)^{k-1} \Psi_k(f)(\{\infty,0\})| \epsilon=
 (-1)^{k-1}\Psi_k(f)|\epsilon(\{\infty,0\})
\\
\Psi_k(f|\epsilon)([0,r]_\infty)&=
  \Psi_k(f)([0,r]_\infty)=
 (-1)^{k-1}\Psi_k(f)(\epsilon[0,r]_\infty)|\epsilon
=(-1)^{k-1}\Psi_k(f)|\epsilon ([0,r]_\infty)
\end{split}
\end{equation*}
et donc que pour tout $\delta\in \KKK_0$,
\begin{equation*}
\Psi_k(f|\epsilon)(\delta)=(-1)^{k-1}\Psi_k(f)|\epsilon (\delta)\ .
\end{equation*}
En particulier, on a
$$\Psi_k(f|\epsilon)([\pi_\infty(0),\gamma \pi_\infty(0)])=
(-1)^{k-1}\Psi_k(f)(\epsilon[\pi_\infty(0),\gamma \pi_\infty(0)])|\epsilon
=
(-1)^{k-1}\Psi_k(f)([\pi_\infty(0),\epsilon \gamma \epsilon^{-1}\pi_\infty(0)])|\epsilon\ .
$$
\end{rem}
\bigskip

On observe que $\Psi_k(f)(\{\infty,0\})$ et donc $\Psi_k(f)$ est la somme
de deux termes de nature différente. Nous allons les étudier séparément.
Posons $v_\infty=x^{k-2}$, c'est une base de la droite de $V_k$
invariante par le stabilisateur de $\infty$.
\begin{prop}\label{invariance0}
Supposons $k > 2$.
Il existe un unique homomorphisme de $\Sl_2(\ZZ)$-modules
$$\Theta_{k}: \Fonct{N}{\CC} \to \Hom_{\Gamma(N)}(\Delta, V_k(\CC))$$
tel que pour
$r \in \PP^1(\QQ)$ et pour $f\in \Fonct{N}{\CC}$,
\begin{equation}
\label{deftheta}
\Theta_{k}(f)(\{r\})=
\left(\int_{(\ZZ/N\ZZ)^2} \widehat{f\ }|\gamma_r
  \dd{\beta'_{k-1}}{\beta_{0}}\right) v_\infty|\gamma_r^{-1}
\end{equation}
pour $\gamma_r\in \Sl_2(\ZZ)$ tel que $r=\gamma_r \infty$.
Il vérifie
$$\Theta_{k}(f)(\{\infty\}-\{0\}) =
\left(\int_{(\ZZ/N\ZZ)^2} \widehat{f}
\dd{\beta'_{k-1}}{\beta_0} \right)x^{k-2}
-
\left(\int_{(\ZZ/N\ZZ)^2}
\widehat{f} \dd{\beta_0}{\beta'_{k-1}}\right)y^{k-2}\ .$$
\end{prop}
\begin{rem}\label{cobord}
Si $f$ appartient à $\Fonct{\Gamma}{\CC}$ pour $\Gamma$ un sous-groupe de congruence
de niveau $N$, $\Theta_{k}(f)$ appartient à
$\Hom_\Gamma(\Delta, V_k(\CC))$ et son image dans $H^1(\Gamma,V_k(\CC))$ est donc nulle:
on a en effet la suite exacte
$$\Hom_\Gamma(\Delta,V_k(\CC)) \to \Hom_\Gamma(\Delta_0,V_k(\CC)) \to H^1(\Gamma, V_k(\CC))\ .$$
\end{rem}
\begin{proof}[Démonstration de la proposition]
Pour $r=\infty$, on pose
$$\Theta_{k}(f)(\{\infty\})=
\left(\int_{(\ZZ/N\ZZ)^2} \widehat{f\ }
  \dd{\beta'_{k-1}}{\beta_{0}}\right) v_\infty
\ .$$
L'invariance par $\Sl_2(\ZZ)$ implique que l'on doit avoir pour $z\in \Delta$
et $\gamma \in \Sl_2(\ZZ)$
$$\Theta_{k}(f|\gamma)(z)=\Theta_{k}(f)|\gamma (z)=\Theta_{k}(f)(\gamma z)|\gamma\ ,$$
autrement dit en l'appliquant à $f|\gamma^{-1}$ et $z=\{r\}$,
$$\Theta_{k}(f)(\{r\})=\Theta_{k}(f|\gamma^{-1})(\{\gamma r\})|\gamma\ .$$
En prenant $\gamma=\gamma_r^{-1}$, on obtient que $\Theta_{k}(f)(\{r\})$
doit être égal à
$\Theta_{k}(f|\gamma_r)(\{\infty\})|\gamma_r^{-1}$.
Vérifions que cette dernière quantité ne dépend pas du choix de $\gamma_r$,
ce qui permettra de l'utiliser comme définition de $\Theta_{k}(f)$.
Il suffit de le vérifier
pour $-\id_2$ et pour $\begin{pmatrix}1&1\\0&1\end{pmatrix}$.
Comme $\dd{\beta'_{k-1}}{\beta_{0}}$ est de même parité que $k$
et que $v_\infty | -\id_2 = (-1)^k v_\infty$,
on a pour $\gamma'_r=-\gamma_r$
$$\Theta_{k}(f|\gamma'_r)(\{\infty\})|{\gamma'_r}^{-1}
=\Theta_{k}(f|\gamma_r)(\{\infty\})|\gamma_r^{-1}\ .$$
Pour $\gamma_r'=\gamma_r \begin{pmatrix}1&1\\0&1\end{pmatrix}$,
\begin{equation*}
\begin{split}
\int_{(\ZZ/N\ZZ)^2} f|\gamma_r'\dd{\beta'_{k-1}}{\beta_{0}}
&=
\int_{(\ZZ/N\ZZ)^2} f|\gamma_r(x_1,x_2-x_1) \dd{\beta'_{k-1}}{\beta_{0}}
\\
&=
\int_{(\ZZ/N\ZZ)^2} f|\gamma_r(x_1,x_2) \dd{\beta'_{k-1}}{\beta_{0}}
\end{split}
\end{equation*}
en utilisant le fait que la distribution $\beta_0$
est invariante par translation.
On a prouvé l'existence et l'unicité de $\Theta_{k}(f)$.
En appliquant la formule \eqref{deftheta} à
$\sigma=\begin{pmatrix} 0&-1\\1&0\end{pmatrix}$, on a
\begin{equation*}
\begin{split}
  \Theta_{k}(f)(\{0\})&=
  \left(\int \widehat{f}(-x_2,x_1) d\beta'_{k-1}(x_1) d\beta_0(x_2) \right)y^{k-2}
  =
  \left(\int \widehat{f}(x_2,x_1) d\beta'_{k-1}(x_1) d\beta_0(x_2) \right)y^{k-2}\\
  &=
  \left(\int \widehat{f} \dd{\beta_0}{\beta'_{k-1}}\right)y^{k-2}
\end{split}
\end{equation*}
en utilisant le fait que $\beta_0$ est pair.
\end{proof}
Soit $\Gamma$ un sous-groupe de congruence de niveau $N$.
Lorsque $k$ est impair et $-1 \in \Gamma$,
$\Hom_\Gamma(\Delta,V_k)$ est nul.
Supposons donc que $k$ est pair ou que
$k$ est impair et $-1 \notin \Gamma$.
Pour chaque pointe  $P\in \Gamma\backslash\PP^1(\QQ)$, fixons
un élément $\gamma_P$ de $SL_2(\ZZ)$ tel que $\gamma_P \infty$
appartienne à $P$. Notons
$\Eiscc{\gamma_P}$ l'élément de $\Hom_{\Gamma}(\Delta,V_k)$ défini par
$$
\Eiscc{\gamma_P}(\{s\}) = \begin{cases}
  v_\infty | \gamma_P^{-1}\gamma^{-1} &\text{ si $s=\gamma \gamma_P \infty$ avec $\gamma\in \Gamma$}\\
0 &\text{ sinon.}
\end{cases}
$$
Lorsque $k$ est pair, $\Eiscc{r,\gamma_P}$ ne dépend que de la classe
$P$ de $\gamma_P \infty$ et on peut écrire
$$
\Eisc{r}{\Gamma}(\{s\}) = \begin{cases}
  v_\infty | \gamma_s^{-1} &\text{ si $s\in \Gamma r$ (avec $\gamma_s\infty=s$)}\\
0 &\text{ sinon.}
\end{cases}
$$
Si $f\in \Fonct{N}{\CC}$, notons $s_{\gamma_P}^{\pm}(f)$ la fonction
sur $\ZZ/N\ZZ$ définie par
$$s_{\gamma_P}^{\pm}(f)=f((0,x) \gamma_P^{-1}) \pm f((0,x) \gamma_P^{-1})\ .$$

\begin{cor}
Soit $k$ un entier strictement supérieur à $2$
tel que $k$ soit impair si $-1 \notin \Gamma$,
Si $f\in \Fonct{\Gamma}{\CC}$, on a
\begin{equation}\label{formule}
\Theta_{k}(f)=
\sum_{P \in \Gamma\backslash \PP^1(\QQ)}
\left ( \frac{1}{-2i \pi}L'(2-k,s_{\gamma_P}^{(-1)^k}(f))
\right)
Eis_{\gamma_P}\ .
\end{equation}
\end{cor}
\begin{proof}
On a
\begin{equation*}
\begin{split}
\int_{(\ZZ/N\ZZ)^2} \widehat{f\ } \dd{\beta'_{k-1}}{\beta_{0}}
=&\frac{1}{N^2}\sum_{x_1,x_2} \sum_{a,b} f(a,b)
\exp(\frac{2i\pi}{N}(ax_2 -b x_1)) \beta'_{k-1}(x_1)\\
=
&\frac{1}{N^2}\sum_{a,b} f(a,b)\left(\sum_{x_2} \exp(\frac{2i\pi ax_2}{N})\right)
\left(\sum_{x_1}
\exp(-\frac{2i\pi b x_1}{N}) \beta'_{k-1}(x_1)\right)\\
=&\frac{1}{N}\sum_{x_1} \sum_{b} f(0,b) \exp(-\frac{2i\pi b x_1}{N})
\beta'_{k-1}(x_1)\\
=&\frac{1}{-2i \pi}L'(2-k,f(0,\cdot)+(-1)^k f^-(0,\cdot))\ .
\end{split}
\end{equation*}
En remarquant que
$f|\gamma_P(0,x)= f((0,x)|\gamma_P^{-1})$,
on obtient donc la formule
$$\int_{(\ZZ/N\ZZ)^2} \widehat{f\ }|\gamma_P \dd{\beta'_{k-1}}{\beta_{0}}
=\frac{1}{-2i \pi}L'(2-k,s_{\gamma_P}^{(-1)^k}(f))\ .
$$
Supposons que $f \in \Fonct{\Gamma}{\CC}$ et décomposons $\Theta_{k}(f)$
dans la base des $\Eiscc{\gamma_P}$. Pour tout $s\in \PP^1(\QQ)$
et $Q=\Gamma s$,
on a $s = \gamma \gamma_{Q}\infty$ avec $\gamma \in \Gamma$
pour le choix de $\gamma_Q$ fait. D'où,
\begin{equation*}
\begin{split}
\sum_{P \in \Gamma\backslash \PP^1(\QQ)}
\left(\int \widehat{f\ } | \gamma_P \dd{\beta'_{k-1}}{\beta_0} \right )
\Eiscc{\gamma_P}(\{s\})
&=
\left(\int \widehat{f\ } | \gamma_{Q} \dd{\beta'_{k-1}}{\beta_0} \right )
\Eiscc{\gamma_{Q}}(\{s\})\\
&=
\left(\int \widehat{f\ } | \gamma_{Q} \dd{\beta'_{k-1}}{\beta_0} \right )
v_\infty | \gamma_{Q}^{-1}\gamma^{-1}\\
&=\Theta_{k}(f)(\gamma_Q \infty)| \gamma^{-1}=\Theta_{k}(f|\gamma)(s)
=\Theta_{k}(f)(s)
\end{split}
\end{equation*}
car $f$ est invariant par $\Gamma$.
On en déduit la formule \eqref{formule}.
\end{proof}
\begin{rem}
Lorsque $\Gamma_f$ contient $\Gamma_0(N)$ ($k$ est alors pair pour que
$\Hom_\Gamma(\Delta,V_k)$ ne soit pas nul),
$s_{\gamma_r}^{+}(f)(x)$ est égal à $2f((0,x)|\gamma_r^{-1})$
pour $r$ équivalent à $\infty$ ou à $0$ et $x\in (\ZZ/N\ZZ)^*$,
c'est-à-dire à $2f(0,1)$ pour $r$ équivalent à $\infty$
(resp. $2f(1,0)$ pour $r$ équivalent à $0$).
C'est une fonction constante. On a donc
$$\Theta_{k}(f)(\{r\})
=
\Theta_{k}(f|\gamma_r)(\{\infty\})|\gamma_r^{-1}
=\frac{1}{-i\pi} \zeta_N'(2-k)f|\gamma_r(0,1) v_\infty | \gamma_r^{-1})$$
et
$$\Theta_{k}(f)=\frac{1}{-i\pi} \zeta_N'(2-k)Eis_f$$
où $Eis_f$ est l'élément de
$\Hom_\Gamma(\Delta,V_k)$ (en particulier rationnel) défini par
$$Eis_f(\{r\})=f|\gamma_r(0,1) v_\infty | \gamma_r^{-1}\ .$$
\end{rem}
\begin{cor}\label{cor:rationalite}Si $f \in \Fonct{\Gamma}{\QQ}$,
l'image $\cocycle_k(f)$ de $\Psi_k(f)$ dans $H^1(\Gamma, V_k(\CC))$
appartient à $H^1(\Gamma, V_k)$.
\end{cor}
\begin{proof}
Le cocycle associé à un élément de $\Hom_\Gamma(\Delta,V_k)$
étant un cobord, l'image de $\Theta_{k}(f)$ dans $H^1(\Gamma,V_k)$
est nulle et $\cocycle_k(f)$
est aussi l'image du cocycle associé à $\Psi_k(f) - \Theta_{k}(f)$
qui est rationnel.
\end{proof}
Soit $\wEisk{k,\Gamma}{\CC}$ l'image de $\hypk{\Fonct{\Gamma}{\CC}}$
dans $H^1(\Gamma, V_k(\CC))$ par $\cocycle_k$.
Si $\wEiskQ{k,\Gamma}$ est de même l'image de
$\hypk{\Fonct{\Gamma}{\QQ}}$
dans $H^1(\Gamma, V_k)$, on a
$\wEisk{k,\Gamma}{\CC}= \CC\otimes \wEiskQ{k,\Gamma}$.

\begin{prop}
\label{prop:ortheisenstein}
\
\begin{enumerate}
\item
L'orthogonal de $\Eisk{k,\Gamma}{\CC}$ dans $M_k(\Gamma)$
pour le produit de Petersson
est $S_k(\Gamma)$.
\item Si $f\in \hypk{\Fonct{\Gamma}{\CC}}$ et si la restriction de $\Psi_k(f)$
aux symboles infinitésimaux est nulle, $\Eis_k(f)$ est nul.

\item Le $\QQ$-espace vectoriel $H^1(\Gamma, V_k)$ est engendré par l'image
de $\Hom_{\Gamma}(\Delta_0,V_k)$
et par le module d'Eisenstein $\wEiskQ{k,\Gamma}$.
\end{enumerate}
\end{prop}
\begin{proof}
On peut trouver une démonstration de l'orthogonalité n'utilisant pas les opérateurs de Hecke
dans \cite{diamond}. Si $k>2$, pour toute pointe $r$, il existe un élément
de $\Eisk{k,\Gamma}{\CC}$ nul en toutes les pointes
de $\Gamma\backslash \PP^1(\QQ)$ sauf $r$. Pour $k=2$,
pour tout couple de pointes distinctes $r$ et $s$, il existe un élément
de $\hypk{\Eisk{k,\Gamma}{\CC}}$ nul en toutes les pointes sauf $r$ et $s$.
Un argument de dimension permet d'en déduire que l'orthogonal de $\Eisk{k,\Gamma}{\CC}$
dans $M_k(\Gamma)$ est exactement $S_k(\Gamma)$.

La restriction de $\Psi_k(f)$ aux symboles infinitésimaux
est nulle si et seulement si le terme constant du
$q$-développement de $\Eis_k(f)|_k\gamma$
est nul pour tout $\gamma \in \Sl_2(\ZZ)$. Dans ce cas,
la série d'Eisenstein $\Eis_k(f)$ appartient à $S_k(\Gamma)$
et est nulle.

La dernière assertion se déduit alors du corollaire \ref{cor:rationalite}.
\end{proof}

\section{Produit sur l'espace de symboles modulaires}

\subsection{Définition algébrique}
Soient $V$ un espace vectoriel (de dimension finie) muni d'une
action à droite de $\Gl_2(\QQ)$ et $\langle \cdot, \cdot \rangle_{V}$
une forme bilinéaire sur $V$ invariante par $\Sl_2(\QQ)$.
Si $R$ est un $\QQ$-espace vectoriel, on note $V(R)=R\otimes_\QQ V$.

On se donne un symbole de Farey étendu $\cF=(\cV,*, \mu_{ell})$ associé à $\Gamma$
et ses données de recollement $\gamma_{a}$ pour $a \in \cV$
(on renvoie à \cite{farey} pour la définition).
On note
$\cV_{ell,2}=\mu_{ell}^{-1}(2)$, $\cV_{ell,3}=\mu_{ell}^{-1}(3)$,
$\cV_{ell}=\cV_{ell,2} \cup \cV_{ell,3}$.
On associe au symbole de Farey un domaine fondamental de $\Gamma$. Sa frontière est formée des arcs de $\cV- \cV_{ell,3}$,
ainsi que d'un couple d'arcs géodésiques pour chaque élément de $ \cV_{ell,3}$.
Donnons une extension naturelle de $\Hom_\Gamma(\Delta_0,V)$ à ces derniers arcs.

Soient $a=(r_a,s_a)$ un chemin elliptique d'ordre 3 et
$z_a$ le point fixe de $\gamma_a$.
On appelle triangle associé à $a$ le triangle hyperbolique de sommets
$r_a$, $s_a=\gamma_a^{-1} r_a$, $t_a=\gamma_a r_a$ et on note
$u_a$ et $v_a$ les chemins $u_{a}=(r_a, z_a)$ et $v_{a}=(z_a,s_a)$.
Si $\Phi\in \Hom_{\Gamma}(\Delta_0, V)$, on définit
\begin{equation*}
\begin{split}
\Phi(u_a)&= \frac{1}{3} \left(\Phi((r_a,\gamma_a r_a)+ \Phi((r_a,\gamma_a^2 r_a))\right)
=\frac{1}{3} \left(\Phi((r_a,t_a))+ \Phi((r_a,s_a))\right)\\
\Phi(v_a)&= \frac{1}{3} \left(\Phi((\gamma_a s_a,s_a))+ \Phi((\gamma_a^2 s_a,s_a))\right)
=\frac{1}{3} \left(\Phi((r_a,s_a))+ \Phi((t_a,s_a))\right)
\ .
\end{split}
\end{equation*}

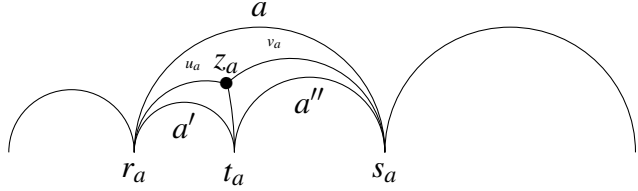
\begin{center}
\begin{tikzpicture}[scale=20]
\draw (0.25,0)
node (start) {} arc (180:0:0.0417)
node [below]{$r_a$}
node (start) {} arc (180:0:0.0833)
node [midway, above] {$a$} node [below]{$s_a$};
 \draw (start) arc (180:73.1736:0.0476) node [near end,above] {\tiny{$u_a$}} node {$\bullet$} node [above] {$z_a$}
arc (133.1736:0:0.0625) node [near start,above] {\tiny{$v_a$}} node (start) {} arc (180:0:0.0833);
\draw (1/3,0) arc (180:0:1/30) node [midway,below] {$a'$} node [below] {$t_a$}
 arc (180:0:1/20) node [midway,below] {$a''$} ;
 \draw (2/5,0) arc (0:11.927:21/95);
\end{tikzpicture}
\end{center}
On a alors $u_a= \gamma_a v_a^-$. Pour $\Phi\in \Hom_{\Gamma}(\Delta_0, V)$,
\begin{equation}
\begin{split}
\Phi(u_a)+ \Phi(v_a)
&=\frac{1}{3} \Phi((r_a,s_a)+ (r_a,t_a)+(r_a,s_a)+ (t_a,s_a))
=\Phi(a)\ .
\end{split}
\end{equation}
On pose
\begin{equation}
\begin{split}
\Phi(\gamma u_a)=\Phi(u_a)|\gamma \ ,\quad
\Phi(\gamma v_a)=\Phi(v_a)|\gamma
\end{split}
\end{equation}
pour $\gamma \in \Gamma$, ce qui est compatible avec les relations
déduite de l'équation $u_a= \gamma_a v_a^-$.

On fait de même pour un arc elliptique $a=(r_a,s_a)$ d'ordre 2
en posant $u_a=(r_a,z_a)$, $v_a=(z_a,s_a)$
pour $z_a$ le milieu de $a$, point fixe de $\gamma_a$.
On a $u_a=\gamma_a v_a^{-}$.
On pose $\Phi(u_a)=\frac{1}{2} \Phi((r_a,\gamma_a r_a)=\frac{1}{2}\Phi(a)$,
$\Phi(v_a)=\frac{1}{2} \Phi(\gamma_a s_a,s_a)=\frac{1}{2}\Phi(a)$.
On a encore $\Phi(u_a)+\Phi(v_a)=\Phi(a)$.

\begin{defn}
Si $\cF=(\cV,*,\mu)$ est un symbole de Farey étendu, on note
$\Vmod$ la suite d'arcs obtenue à partir de $\cV$ en remplaçant
les arcs elliptiques $a$ par les deux arcs $u_a$ et $v_a$.
\end{defn}

Comme $u_a= \gamma_a v_a^{-}$, l'application $*$ se prolonge donc
de manière naturelle à $\Vmod$ par $u_a^*=v_a$, $v_a^*=u_a$ et la donnée
de recollement $\gamma_{u_a}$ de $u_a$ est $\gamma_a$.
Si $\Phi \in \Hom_{\Gamma}(\Delta_0, V)$, $\Phi(a)$ est bien définie pour
$a\in \Vmod$.

Rappelons qu'il y a une application naturelle
$$\Hom_\Gamma(\Xi_0, V) \to H^1(\Gamma,V)$$
qui associe à $\Phi$ la classe du cocycle $\widetilde{\Phi}_{1,Z}:
\gamma \mapsto \Phi([Z,\gamma^{-1}Z])$ dans
$H^1(\Gamma,V)$ avec $Z \in \mathcal{P}(\QQ)$. Nous noterons
de la même manière l'application qui s'en déduit sur
$\Hom_\Gamma(\Delta_0, V)$ et
qui associe à $\Phi$ la classe du cocycle $\widetilde{\Phi}_{1,Z}=\widetilde{\Phi}_{1,r}:
\gamma \mapsto \Phi((r,\gamma^{-1}r))$ dans la cohomologie
$H^1_{par}(\Gamma,V)$ où $r$ est l'image de $Z$ dans $\PP^1(\QQ)$
(proposition \ref{prop:suiteexacte}).

\begin{thm} Soit $\cF=(\cV,*, \mu_{ell})$ un symbole de Farey étendu associé à $\Gamma$.
\label{def:biendefini}
Pour $\widetilde{\Phi}_1$ un 1-cocyle de $\Gamma$ à valeurs dans $V$
et $\Phi_2$ appartenant à $\Hom_\Gamma(\Delta_0, V)$, on pose
\begin{equation}
\label{def:petersson}
\left\lbrace \widetilde{\Phi}_1, \Phi_2\right\rbrace_{\Gamma,\cF}=
\frac{1}{2}\sum_{a\in \Vmod} \langle
  \widetilde{\Phi}_1(\gamma_a^{-1}),\Phi_2 (a)\rangle_{V}\ .
\end{equation}

\begin{enumerate}
\item Cette expression ne dépend que de la classe $\Phi_1$ de
$\widetilde{\Phi}_1$ dans $H^1(\Gamma, V)$ et on pose
\begin{equation*}
\begin{split}
\left\lbrace \Phi_1, \Phi_2\right\rbrace_{\Gamma,\cF}=
&\left\lbrace \widetilde{\Phi}_1, \Phi_2\right\rbrace_{\Gamma,\cF} \ .
\end{split}
\end{equation*}
\item Si $\Phi_1 \in \Hom_\Gamma(\Delta_0, V)$
$$\left\lbrace \Phi_1, \Phi_2\right\rbrace_{\Gamma,\cF}=\left\lbrace
\widetilde{\Phi}_{1,r}, \Phi_2\right\rbrace_{\Gamma,\cF}$$
ne dépend pas de $r\in \PP^1(\QQ)$.
\item Pour $V=V_k$ muni de la forme bilinéaire
$\langle \cdot , \cdot \rangle_{V_k}$ définie au paragraphe
\ref{defbilk}, $\left\lbrace \Phi_1, \Phi_2\right\rbrace_{\Gamma,\cF}$
ne dépend pas de $\cF$. On l'appelle le
\textsl{produit de Petersson algébrique} de $\Phi_1$ et $\Phi_2$
et on le note aussi
$\left\lbrace \Phi_1, \Phi_2\right\rbrace_{\Gamma}$.
\end{enumerate}
\end{thm}
Pour $\Phi_1\in \Hom_\Gamma(\KKK_0,V)$ et $\Phi_2\in \Hom_\Gamma(\Delta_0,V)$,
nous écrirons
$\left\lbrace \Phi_1, \Phi_2\right\rbrace_{\Gamma,\cF}$
pour
$\left\lbrace \Phi'_1, \Phi_2\right\rbrace_{\Gamma,\cF}$
où $\Phi'_1$ est l'image de $\Phi_1$ dans $H^1(\Gamma,V)$.

Remarquons que l'involution $*$ n'a pas de point fixe
sur $\Vmod$ contrairement à ce qui se passe sur $\cV$.
On pourrait faire la somme sur $\Vmod/*$ grâce au lemme \ref{symetrie}
qui suit,
ce qui permet dans le cas de calculs explicites de ne tenir compte
que de la moitié des arcs. On rappelle que $a =\gamma_a {a^*}^-$ avec $\gamma_a \in \Gamma$.
\begin{lem}
\label{symetrie}
Si $a \in \Vmod$,
$$
\langle \widetilde{\Phi}_1(\gamma_a^{-1}), \Phi_2(a)\rangle_V =
\langle \widetilde{\Phi}_1(\gamma_a), \Phi_2(a^*)\rangle_V
=
\langle \widetilde{\Phi}_1(\gamma_{a^*}^{-1}), \Phi_2(a^*)\rangle_V
\ .
$$
\end{lem}
En termes de $\cV$ plutôt que de $\Vmod$, le produit de Petersson algébrique s'écrit:
\begin{equation}
\label{def:petersson1}
\left\lbrace \Phi_1, \Phi_2\right\rbrace_{\Gamma,\cF}=
\frac{1}{2}\sum_{a\in \cV- \cV_{ell,3}} \langle
  \widetilde{\Phi}_1(\gamma_a^{-1}),\Phi_2 (a)\rangle_{V}
+ \frac{1}{3}\sum_{a\in \cV_{ell,3}}
\langle \widetilde{\Phi}_1(\gamma_a^{-1})+\widetilde{\Phi}_1(\gamma_a^{-2}),
\Phi_2 (a)\rangle_{V}
\end{equation}
et pour $\Phi_1$ et $\Phi_2$ dans $\Hom_{\Gamma}(\Delta_0, V)$ et $r \in \PP^1(\QQ)$,
\begin{equation*}
\begin{split}
\left\lbrace\Phi_1, \Phi_2\right\rbrace_{\Gamma,\cF} &=
\frac{1}{2} \sum_{a\in \Vmod}
\langle \Phi_{1}((r,\gamma_a r)),\Phi_2 (a)\rangle_{V}
\\&=
\frac{1}{2} \sum_{a\in \cV - \cV_{ell,3}}
\langle \Phi_{1}((r,\gamma_a r)),\Phi_2 (a)\rangle_{V}
+
\frac{1}{3}\sum_{a \in \cV_{ell,3}}
\langle \Phi_{1}((r,\gamma_a r))+
\Phi_{1}((r,\gamma_a^{2} r)),\Phi_2 (a)\rangle_{V}
\ .
\end{split}
\end{equation*}
L'indépendance par rapport à $\cF$ sera démontrée
après l'établissement du lien avec le produit de Petersson complexe
(corollaire \ref{independance}) pour $V_k$ et la forme bilinéaire
$\langle \cdot,\cdot\rangle_{V_k}$.
Nous garderons donc $\cF$ en indice dans le cas général.\footnote{Nous ne sommes pas
arrivés à la démontrer directement.}

\begin{proof}[Démonstration de l'indépendance par rapport au cocycle]
Notons $H(\widetilde{\Phi}_1, \Phi_2)$ le second membre de
\eqref{def:petersson}.
Si $\widetilde{\Phi}_1$ et $\widetilde{\Phi}_1'$ sont deux cocycles représentant $\Phi_1$,
on a $$\widetilde{\Phi}_1(\gamma)-\widetilde{\Phi}_1'(\gamma)= C|(1-\gamma)$$
avec $C \in V$.
La contribution de $a\in \Vmod$
à $H(\widetilde{\Phi}_1, \Phi_2)-H(\widetilde{\Phi}'_1 , \Phi_2)$ est
\begin{equation*}
\begin{split}
\frac{1}{2}\langle\widetilde{\Phi}_1(\gamma_a^{-1})-\widetilde{\Phi}_1'(\gamma_a^{-1}),
\Phi_2(a)\rangle_V
&=
\frac{1}{2}\langle C |(1-\gamma_a^{-1}), \Phi_2(a)
\rangle_V\\
&=
\frac{1}{2}\langle C,\Phi_2(a) | (1-\gamma_a)\rangle_V=
\langle C,\Phi_2((1-\gamma_a^{-1})a)\rangle_V\\
&=
\frac{1}{2}\langle C,\Phi_2(a+a^*)\rangle_V\ .
\end{split}
\end{equation*}
Finalement,
$$
H(\widetilde{\Phi}_1, \Phi_2) -H(\widetilde{\Phi}_1' , \Phi_2)=
\frac{1}{2}\langle C, \Phi_2\left(\sum_a a + \sum_a a^*\right)\rangle_V = 0
\ . $$
\end{proof}
\begin{prop}
\label{lem:nul}
Si $\Phi_2$ appartient à l'image de $\Hom_\Gamma(\Delta, V)$
et si $\Phi_1$ appartient à $\Hom_\Gamma(\Delta_0, V)$, on a
$$\left\lbrace \Phi_2, \Phi_1\right\rbrace_{\Gamma,\cF} = \left\lbrace \Phi_1, \Phi_2\right\rbrace_{\Gamma,\cF}=0\ . $$
\end{prop}
\begin{proof}
Soit
$\Phi'_2 \in \Hom_\Gamma(\Delta,V)$ tel que
$\Phi_2((c,d))= \Phi'_2(\{d\}) - \Phi'_2(\{c\})$.
Prenons comme cocycle représentant l'image de $\Phi_2$ (resp. $\Phi_1$) dans
$H^1(\Gamma,V)$ le cocycle $\widetilde{\Phi}_2=\widetilde{\Phi}_{2,r}$
(resp. $\widetilde{\Phi}_1=\widetilde{\Phi}_{1,r}$).
Montrons d'abord la nullité de $\left\lbrace \Phi_2, \Phi_1\right\rbrace_{\Gamma,\cF}$. On a
\begin{equation*}
\begin{split}
\left\lbrace \Phi_2, \Phi_1\right\rbrace_{\Gamma,\cF}&=
 \frac{1}{2}\sum_{a\in \Vmod}
\langle
 \widetilde{\Phi}_{2}(\gamma_a^{-1}),\Phi_1(a)\rangle_V
\\
 &=\frac{1}{2}\sum_{a\in\Vmod}\langle
 \Phi'_2(\gamma_a\{r\}) - \Phi'_2(\{r\}),\Phi_1(a)\rangle_V
\\
 &=\frac{1}{2}\sum_{a\in \Vmod }\langle
 \Phi'_2(\{r\}) | (\gamma_a^{-1} -1),\Phi_1(a)\rangle_V
\\
 &=-\frac{1}{2}\sum_{a \in \Vmod}\langle
 \Phi'_2(\{r\}),\Phi_1(a + a^*)\rangle_V=0
\end{split}
\end{equation*}
comme dans la démonstration de la définition \ref{def:biendefini}.

Montrons maintenant la nullité de $\left\lbrace \Phi_1, \Phi_2\right\rbrace_{\Gamma,\cF}$.
La démonstration s'inspire de l'article de Shimura \cite{shimura}.
Si $a=(P,Q)$ n'est pas un chemin elliptique d'ordre 3 et
si $a^*=(R,S)$, on a par définition $Q=\gamma_a R$ et
$P=\gamma_a S$.
La contribution de $a$ à $\left\lbrace \Phi_1, \Phi_2\right\rbrace_{\Gamma,\cF} $ est
\begin{equation*}
\begin{split}
\frac{1}{2}\langle \widetilde{\Phi}_1(\gamma_{a}^{-1}),\Phi_2' (\{Q\}-\{P\})\rangle_V
&=
\frac{1}{2}\langle \widetilde{\Phi}_1(\gamma_{a}^{-1}),\Phi_2' (\{\gamma_{a} R\})\rangle_V
- \frac{1}{2}
\langle \widetilde{\Phi}_1(\gamma_{a}^{-1}),\Phi_2' (\{P\})\rangle_V
\\&
=\frac{1}{2}\langle \widetilde{\Phi}_1(\gamma_{a}^{-1})|_{\gamma_a},\Phi_2' (\{R)\}\rangle_V
-\frac{1}{2}
\langle \widetilde{\Phi}_1(\gamma_{a}^{-1}),\Phi_2' (\{P\})\rangle_V
\\&
=-\frac{1}{2}\langle \widetilde{\Phi}_1(\gamma_{a^*}^{-1}),\Phi_2' (\{R\})\rangle_V
-
\frac{1}{2}\langle \widetilde{\Phi}_1(\gamma_{a}^{-1}),\Phi_2' (\{P\})\rangle_V\end{split}
\end{equation*}
($\widetilde{\Phi}_1$ est nul sur l'identité).
Si $\gamma_{a}$ est d'ordre 3, la contribution du chemin
$a=(P,\gamma_{a}^{-1} P)$ est
\begin{equation*}
\begin{split}
\frac{1}{3} (\langle \widetilde{\Phi}_1(\gamma_{a}^{-1})
&+ \widetilde{\Phi}_1(\gamma_{a}^{-2}),
  \Phi_2'(\{\gamma_{a}^{-1} P\})\rangle_V
  -
\langle \widetilde{\Phi}_1(\gamma_{a}^{-1})+\widetilde{\Phi}_1(\gamma_{a}^{-2}),
  \Phi_2'(\{P\})\rangle_V)\\
&=\frac{1}{3}\langle \widetilde{\Phi}_1(\gamma_{a}^{-1})|\gamma_{a}^{-1}
+\widetilde{\Phi}_1(\gamma_{a}^{-2})|\gamma_{a}^{-1}
-\widetilde{\Phi}_1(\gamma_{a}^{-1})-\widetilde{\Phi}_1(\gamma_{a}^{-2}),
  \Phi_2'(\{P\})\rangle_V
\\
&=\frac{1}{3}\langle
\widetilde{\Phi}_1(\gamma_{a}^{-2})
-\widetilde{\Phi}_1(\gamma_{a}^{-1})
-\widetilde{\Phi}_1(\gamma_{a}^{-1})
-\widetilde{\Phi}_1(\gamma_{a}^{-1})
-\widetilde{\Phi}_1(\gamma_{a}^{-2}),
  \Phi_2'(\{P\})\rangle_V
\\
&=-\langle \widetilde{\Phi}_1(\gamma_{a}^{-1}),\Phi_2'(\{P\})\rangle_V
\ .
\end{split}
\end{equation*}
Autrement dit,
\begin{equation}
\label{contribution}
\left\lbrace \Phi_1,\Phi_2\right\rbrace_{\Gamma,\cF} =-\sum_{a\in \cV}
\langle \widetilde{\Phi}_1(\gamma_{a}^{-1}),
\Phi_2'(\{P_a\})\rangle_V
\end{equation}
où $P_a$ est l'origine du chemin $a$. Si $P$ qui est une extrémité
d'un élément de $\cV$, on étudie ensuite son orbite au sens suivant :
le successeur de $P$ est $\gamma^{-1} P$ si $\gamma$ est la donnée
de recollement associée à l'élément de $\cV$ d'origine $P$.
Notons $\alpha$ une orbite nécessairement finie de longueur $l(\alpha)$ :
$\alpha=(P_0,\cdots, P_{l(\alpha)-1})$
avec $a_{j}=(P_j,Q_j)$,
$P_{j+1}=\gamma_{a_{j}}^{-1}P_j$,
$P_{0}=\gamma_{a_{l(\alpha)-1}}^{-1}P_{l(\alpha)-1}$.
En posant $\tau_{0}=\id$,
$\tau_{j+1}= \gamma_{a_{j}}^{-1} \tau_{j}$,
on a donc
$P_{j}=\tau_j P_0$ et
la contribution de l'orbite $\alpha$ à $-\left\lbrace \Phi_1, \Phi_2\right\rbrace_{\Gamma,\cF} $ est la somme
pour $j=0$ à $l(\alpha)-1$ des
$\langle \widetilde{\Phi}_1(\gamma_{a_{j}}^{-1}),
  \Phi_2'(\{P_{j}\})\rangle_V$.
  On a
\begin{equation*}
\begin{split}
\langle \widetilde{\Phi}_1(\gamma_{a_{j}}^{-1}),
  \Phi_2'(\{P_{j}\})\rangle_V
  &=
\langle \widetilde{\Phi}_1(\gamma_{a_{j}}^{-1}),
  \Phi_2'(\{\tau_{j} P_{0}\})\rangle_V
\\
  &=
\langle \widetilde{\Phi}_1(\gamma_{a_{j}}^{-1})|\tau_{j},
  \Phi_2'(\{P_{0}\})\rangle_V\\
   &=\langle
  \widetilde{\Phi}_1(\gamma_{a_{j}}^{-1}\tau_{j})
  -
  \widetilde{\Phi}_1({\tau_{j}}),
  \Phi_2'(\{P_{0}\})\rangle_V
  \text{ par la propriété de cocycle}\\
  &=\langle
  \widetilde{\Phi}_1({\tau_{j+1}})
  -
  \widetilde{\Phi}_1({\tau_{j}}),
  \Phi_2'(\{P_{0}\})\rangle_V\ .
\end{split}
\end{equation*}
La somme des contributions des points de l'orbite $\alpha$ est égale à
\begin{equation}
\label{contribution2}
\langle
  \widetilde{\Phi}_1({\tau_{l(\alpha)}})
  -
  \widetilde{\Phi}_1({\tau_{0}}),
  \Phi_2'(\{P_{0}\})\rangle_V
  =
\langle
  \widetilde{\Phi}_1({\tau_{l(\alpha)}}),
  \Phi_2'(\{P_{0}\})\rangle_V
\end{equation}
où $\tau_{l(\alpha)}$ appartient au stabilisateur de
$P_{0}$.
Comme la restriction du cocycle $\widetilde{\Phi}_1$
au stabilisateur de $P_{0}$ est un cobord (voir proposition \ref{prop:suiteexacte}),
il existe $u\in V$ tel que
\begin{equation*}
\begin{split}
\langle
  \widetilde{\Phi}_1({\tau_{l(\alpha)}}),
  \Phi_2'(\{P_{0}\})\rangle_V=
\langle u|(\tau_{l(\alpha)}-1),
  \Phi_2'(\{P_{0}\})\rangle_V
=\langle u, \Phi_2'(\{(\tau_{l(\alpha)}-1)P_{0}\})\rangle_V=0\ .
\end{split}
\end{equation*}
On en déduit que $\left\lbrace \Phi_1,\Phi_2\right\rbrace_{\Gamma,\cF} =0$.
\end{proof}

Avec les notations de la preuve précédente,
$\tau_{l(\alpha)}$ est le générateur du stabilisateur de la pointe
$P_0$ conjugué à $\smallmat{1&w(P_0)\\0&1}$
où $w(P_0)>0$ est la largeur de la pointe $P_0$, générateur dit positif,
et il y a une orbite par pointe. Dans le cas où le polygone fondamental
est sous forme canonique,
il y a une orbite de longueur 1 pour chaque pointe sauf une et une orbite
de longueur $>1$ pour la pointe restante.

Un corollaire de la démonstration et en particulier des formules
\eqref{contribution} et \eqref{contribution2} est le suivant.
\begin{prop}
\label{noncusp}
Soit $C(\Gamma)$ un système de représentants de $\Gamma\backslash \PP^1(\QQ)$.
Si $\Phi_1\in H^1(\Gamma,V)$ et
 si $\Phi_2 \in\Hom_{\Gamma}(\Delta_0,V)$ est l'image
d'un élément $\Phi_2'$ de $\Hom_{\Gamma}(\Delta,V)$,
on a
$$\left\lbrace \Phi_1,\Phi_2\right\rbrace_{\Gamma,\cF} =-
\sum_{s\in C(\Gamma)}
\langle \widetilde{\Phi_{1}^{(s)}}(\tau_s), \Phi_2'(\{s\}) \rangle_V
$$
où $\tau_s$ est le générateur positif du stabilisateur $\Gamma_s$ de $s$
et où $\widetilde{\Phi_1^{(s)}}$ est un cocycle de $\Gamma_s$ représentant la
restriction de $\Phi_1$ à $\Gamma_s$.
\end{prop}
\begin{prop}
Si $\langle\cdot , \cdot \rangle_V $ est symétrique
(resp. antisymétrique),
la forme bilinéaire $\left\lbrace \cdot , \cdot \right\rbrace_{\Gamma,\cF}$
est antisymétrique (resp. symétrique)
sur $\Hom_\Gamma(\Delta_0, V)$.
\end{prop}
\begin{proof}
Notons $E_3$ l'ensemble des points elliptiques d'ordre 3 de $\Gamma$.
Si $t\in E_3$, il existe $(r_a,z_a)\in \Vmod$ et $\gamma \in \Gamma$
tels que $\gamma z_a=t$ et on pose $\Phi((s,t))=
\Phi((s,\gamma r_a)) + \Phi((\gamma r_a,t))$
pour tout $s\in \PP^1(\QQ)$.
Fixons $r\in \PP^1(\QQ)$.
Si $\Phi\in \Hom_\Gamma(\Delta_0, V)$,
on introduit la fonction $\widehat{\Phi}_r$ sur $\PP^1(\QQ)\cup E_3$
définie par
$$
\widehat{\Phi}_r(t) =\Phi((r,t))
$$
C'est un moyen d'étendre $\Phi$ en une fonction de
$\PP^1(\QQ)\cup E_3$ dans $V$, comme on le fait dans le cas complexe.
Elle n'est pas invariante par $\Gamma$ : on a pour tout $t\in \PP^1(\QQ)\cup E_3$
$$\widehat{\Phi}_r(t)- \widehat{\Phi}_r | \gamma(t)= \widetilde{\Phi}_r(\gamma)$$
En effet,
\begin{equation*}
\begin{split}
\widehat{\Phi}_r(\gamma t)&= \Phi((r,\gamma r)) +\Phi(\gamma(r,t))
=
\widetilde{\Phi}_r(\gamma^{-1}) + \Phi((r,t)) | \gamma^{-1}\\
&=\widetilde{\Phi}_r(\gamma^{-1}) + \widehat{\Phi}_r(t) | \gamma^{-1}
=- \widetilde{\Phi}_r(\gamma) | \gamma^{-1}+ \widehat{\Phi}_r(t) | \gamma^{-1}
\end{split}
\end{equation*}
car
$$\widetilde{\Phi}_r(\gamma^{-1}) =- \widetilde{\Phi}_r(\gamma) | \gamma^{-1}$$
D'où
$$
\widehat{\Phi}_r(t)-\widehat{\Phi}_r(\gamma t)| \gamma= \widetilde{\Phi}_r(\gamma)
$$
Calculons à l'aide des fonctions $\widehat{\Phi}_{1,r}$ et
$\widehat{\Phi}_{2,r}$
la contribution d'un arc
$a=(\partial_1 a, \partial_2 a)$ de $\Vmod$ à $2\left\lbrace \Phi_1, \Phi_2\right\rbrace_{\Gamma,\cF} $
Pour $t \in \PP^1(\QQ)\cup E_3$ quelconque, on a
\begin{equation*}
\begin{split}
\langle\widetilde{\Phi}_{1,r}(\gamma_a^{-1}),\Phi_2(a)\rangle_V&=
\langle
\widehat{\Phi}_{1,r}( t) - \widehat{\Phi}_{1,r}(\gamma_a^{-1} t)|\gamma_a^{-1},
\Phi_2(a)
\rangle_V
\\
&=
\langle \widehat{\Phi}_{1,r}( t), \Phi_2(a)\rangle_V
-
\langle \widehat{\Phi}_{1,r}(\gamma_a^{-1} t),
\Phi_2(a)|\gamma_a\rangle_V
\\
&=
\langle \widehat{\Phi}_{1,r}( t), \Phi_2(a)\rangle_V
+
\langle \widehat{\Phi}_{1,r}(\gamma_a^{-1} t),
\Phi_2(a^*)\rangle_V
\\
&=
\langle\widehat{\Phi}_{1,r}(t),\widehat{\Phi}_{2,r}(\partial_2 a)\rangle_V
- \langle\widehat{\Phi}_{1,r}( t),\widehat{\Phi}_{2,r}(\partial_1 a)\rangle_V
\\ &\quad\quad + \langle\widehat{\Phi}_{1,r}(\gamma_a^{-1} t),
\widehat{\Phi}_{2,r}(\partial_2 a^*)\rangle_V
- \langle\widehat{\Phi}_{1,r}(\gamma_a^{-1} t),
\widehat{\Phi}_{2,r}(\partial_1 a^*)\rangle_V\ .
\end{split}
\end{equation*}
En prenant $t=\partial_2 a$, on obtient
\begin{equation*}
\begin{split}
\langle\widetilde{\Phi}_{1,r}(\gamma_a^{-1}),\Phi_2(a)\rangle_V&=
\langle\widehat{\Phi}_{1,r}(\partial_2 a),\widehat{\Phi}_{2,r}(\partial_2 a)\rangle_V
- \langle\widehat{\Phi}_{1,r}(\partial_2 a),\widehat{\Phi}_{2,r}(\partial_1 a)\rangle_V
\\
&\quad\quad + \langle\widehat{\Phi}_{1,r}(\partial_1 a^*),
\widehat{\Phi}_{2,r}(\partial_2 a^*)\rangle_V
- \langle\widehat{\Phi}_{1,r}(\partial_1 a^*),
\widehat{\Phi}_{2,r}(\partial_1 a^*)\rangle_V\\
&=\langle\widehat{\Phi}_{1,r}(\partial_2 a),\widehat{\Phi}_{2,r}(\partial_2 a)\rangle_V
-\langle\widehat{\Phi}_{1,r}(\partial_1 a^*), \widehat{\Phi}_{2,r}(\partial_1 a^*)\rangle_V
\\&
\quad \quad -
\langle\widehat{\Phi}_{1,r}(\partial_2 a),\widehat{\Phi}_{2,r}(\partial_1 a)\rangle_V
+ \langle\widehat{\Phi}_{1,r}(\partial_1 a^*),
\widehat{\Phi}_{2,r}(\partial_2 a^*),
\rangle_V\ .
\end{split}
\end{equation*}
En faisant la somme sur tous les $a$ et en utilisant le fait que
$a \to a^*$ est une bijection, le terme
$$
\sum_{a\in \Vmod}
\langle\widehat{\Phi}_{1,r}(\partial_2 a),\widehat{\Phi}_{2,r}(\partial_2 a)\rangle_V
-\langle\widehat{\Phi}_{1,r}(\partial_1 a^*), \widehat{\Phi}_{2,r}(\partial_1 a^*)\rangle_V
$$
est nul. On a alors
\begin{equation}
\label{autredef}
\left\lbrace \Phi_1,\Phi_2\right \rbrace_{\Gamma,\cF}
=\frac{1}{2}\sum_{a\in\Vmod} \left(\langle\widehat{\Phi}_{1,r}(\partial_1 a^*),
\widehat{\Phi}_{2,r}(\partial_2 a^*)\rangle_V
-
\langle\widehat{\Phi}_{1,r}(\partial_2 a),\widehat{\Phi}_{2,r}(\partial_1 a)\rangle_V
\right)\ .
\end{equation}
Cette expression est antisymétrique (resp. symétrique) si $\langle \cdot, \cdot \rangle_{V}$ est symétrique (resp. antisymétrique).
\end{proof}
\begin{rem}
L'équation \eqref{autredef} donne une autre définition de l'accouplement
$\left\lbrace \cdot,\cdot \right\rbrace_\Gamma$ restreint à $\Hom_{\Gamma}(\Delta_0,V)$.
\end{rem}
\subsection{Lien avec le produit de Petersson complexe}
On prend maintenant pour $V$ le $\QQ$-espace vectoriel
$V_k=\QQ[x,y]_{k-2}$.
\label{defbilk}
La forme bilinéaire $\langle \cdot, \cdot\rangle_{V_k}$
définie sur $V_k$ par
$$ \langle \sum_{i=0}^{k-2} a_i x^i y^{k-2-i} ,\sum_{i=0}^{k-2} b_i x^i y^{k-2-i} \rangle_{V_k}
=
(-1)^{k-2}\sum_{i=0}^{k-2} (-1)^i \frac{a_i b_{k-2-i}}{\binom{k-2}{i}} $$
est symétrique ou antisymétrique selon que $k$ est pair ou impair et vérifie
$$\langle (\tau x + y)^{k-2}, (\tau' x + y)^{k-2} \rangle_{V_k} =(\tau-\tau')^{k-2}$$
et
$$\langle P|\gamma, Q \rangle_{V_k}
=\langle P, Q|\gamma^*\rangle_{V_k}$$
pour $\gamma \in \Gl_2(\QQ)$ et $\gamma^*=\det(\gamma) \gamma^{-1}$.
Elle est en particulier invariante par l'action de $\Sl_2(\QQ)$.
On l'étend à $V_k(\CC)$ par linéarité.

Si $\cD$ est un domaine fondamental de $\Gamma$ dans $\cH^*=\cH \cup \PP^1(\QQ)$ et
si $F$ et $G$ sont deux formes pour $\Gamma$ dont l'une est parabolique,
le produit de Petersson de $F$ et $G$ est défini par
$$\langle F,G\rangle_{\Gamma} =
\int_{\cD} F(\tau) \overline{G(\tau)} y^k \frac{dx dy}{y^2}
=-\frac{1}{2i}
\int_{\cD} F(\tau) \overline{G(\tau)} \im(\tau)^{k-2} d\tau \overline{d\tau}
\ . $$
\begin{thm}
\label{thm:petersson1}Soit $\cF$ un symbole de Farey pour $\Gamma$.
Soient $F$ et $G$ deux formes modulaires pour $\Gamma$ de poids $k$. On suppose
que $G$ est parabolique. Pour l'accouplement
$\left\lbrace, \right\rbrace_{\Gamma,\cF}$ associée à $V_k$ et à
$\langle \cdot,\cdot\rangle_{V_k}$ et prolongée à $\CC$ par $\CC$-linéarité,
on a
\begin{equation*}
\begin{split}
\left\lbrace \Per(F), \overline{\Per(G)}\right\rbrace_{\Gamma,\cF}
=-(2i)^{k-1}\langle F,G\rangle_{\Gamma}
\end{split}
\end{equation*}
et
\begin{equation*}
\left\lbrace \Per(F), \Per(G)\right\rbrace_{\Gamma,\cF}=0\ .
\end{equation*}
\end{thm}
Rappelons que
\begin{equation*}
\left\lbrace \Per(F), \Per(G)\right\rbrace_{\Gamma,\cF}
=\left\lbrace \cocycle(F), \Per(G)\right\rbrace_{\Gamma,\cF}
\end{equation*}
par définition.
On trouve une formule de ce type dans les articles
de Eichler \cite{eichler} et Shimura \cite{shimura} de la fin des années cinquante
lorsque $F$ et $G$ sont paraboliques.
Nous donnons une esquisse de la démonstration (voir aussi
\cite{haberland}, \cite{pasol}, \cite{cohen} pour $\Gamma=\Sl_2(\ZZ)$).

Commençons par des lemmes qui sont classiques dans le cas parabolique
et qui sont des variantes de l'intégrale de Eichler (voir paragraphe \ref{eichler}).
Soit $F$ une forme modulaire de poids $k$ pour un sous-groupe de congruence.
On pose pour $\tau \in \cH$
$$\cW(F)(\tau)= \int_{i\infty}^\tau \widetilde{F}(t) (t-\bar{\tau})^{k-2} dt
+ a_0(F)\int_0^\tau (t-\bar{\tau})^{k-2} dt$$
où les chemins utilisés sont formés d'un nombre fini d'arcs géodésiques.
\begin{lem} \label{lem:diff}
\begin{enumerate}
\item On a
$$\frac{\partial \cW(F)}{\partial \tau} (\tau)=
F(\tau) (\tau-\overline{\tau})^{k-2}\ . $$
\item Supposons $F$ et $G$ invariantes par $\Gamma$.
Si $a$ est un arc géodésique entre deux éléments de $\PP^1(\QQ)$,
pour $\gamma \in \Gamma$,
$$
\int_{a} \cW(F)(\tau) \overline{G(\tau) d\tau}
-
\int_{\gamma^{-1} a} \cW(F)(\tau) \overline{G(\tau) d\tau}
$$ ne dépend pas de $\tau$ :
\begin{equation*}
\begin{split}
\int_{a} \cW(F)(\tau) \overline{G(\tau) d\tau}
-
\int_{\gamma^{-1} a} \cW(F)(\tau) \overline{G(\tau) d\tau}
&=
\langle
 \cocycleF{F}{\gamma^{-1}}
  ,
  \overline{\Per(G)(a)}\rangle_{V_k} \ .
\end{split}
\end{equation*}
\end{enumerate}
\end{lem}
\begin{proof}
En utilisant l'identité
$$\langle (\tau x + y)^{k-2}, (\tau' x + y)^{k-2} \rangle_{V_k} =(\tau-\tau')^{k-2}$$
et en faisant le changement de variable $\tau \mapsto \gamma \tau$, on trouve
\footnote{utiliser par exemple les formules du type de celles de \eqref{formulej}.}
que
pour tout $\gamma$ dans $\Gl_2^+(\QQ)$,
\begin{equation*}
\begin{split}
\int_{\gamma^{-1} a} \cW(F)(\tau) \overline{G(\tau) d\tau}
&=
\int_{a}
\langle W(F)(\gamma^{-1}\tau)
  ,\overline{\left(
  (\tau x + y)^{k-2}G|\gamma^{-1}(\tau)d\tau \right )}| \gamma
  \rangle_{V_k}
\\
&=
\int_{a}
\langle
 W(F)(\gamma^{-1}\tau)| \gamma^{-1}
  ,
  \overline{(\tau x + y)^{k-2}G|\gamma^{-1}(\tau)d \tau} \rangle_{V_k}
\\
&=
\int_{a} \langle W(F)|\gamma^{-1} (\tau)
  ,
  \overline{(\tau x + y)^{k-2}G|\gamma^{-1}(\tau)d\tau} \rangle_{V_k} \ .
\end{split}
\end{equation*}
Si de plus $\gamma$ appartient à $\Gamma$ et que
$F$ et $G$ sont invariants par $\Gamma$, on a
\begin{equation*}
\begin{split}
\int_{a} \cW(F)(\tau) \overline{G(\tau) d\tau}
&-
\int_{\gamma^{-1} a} \cW(F)(\tau) \overline{G(\tau) d\tau}
\\&=
\int_{a}\langle
 W(F)(\tau) - W(F)|\gamma^{-1} (\tau)
  ,
  \overline{(\tau x + y)^{k-2}G(\tau)d\tau} \rangle_{V_k}
\\&=
\langle
  \cocycleF{F}{\gamma^{-1}}
  ,
  \int_{a}\overline{(\tau x + y)^{k-2}G(\tau)d\tau} \rangle_{V_k}
\\&=
\langle
 \cocycleF{F}{\gamma^{-1}}
  ,
  \overline{\Per(G)(a)}\rangle_{V_k} \ .
\end{split}
\end{equation*}
\end{proof}
\begin{proof}[Démonstration du théorème \ref{thm:petersson1}]
Soit $\cD$ le domaine fondamental associé à $\cF$. Son bord $\partial \cD$
est exactement formé
des chemins de $\Vmod$.
On a
$$\langle F,G\rangle_{\Gamma}=\int_{\cD} F(\tau) \overline{G(\tau)} y^k \frac{dx dy}{y^2}=
-\frac{1}{(2i)^{k-1}}\int_{\cD} F(\tau)\overline{G(\tau)} (\tau-\overline{\tau})^{k-2} d\tau d\overline{\tau}
\ . $$
$$\frac{\partial \left(\cW(F)(\tau)\overline{G(\tau)}\right)}{\partial \tau}=
F(\tau) (\tau-\overline{\tau})^{k-2} \overline{G(\tau)}\ . $$
En appliquant la formule de Stokes
$$ \int_{\cD} \left(\frac{\partial Q}{\partial \tau}
- \frac{\partial P}{\partial \overline{\tau}} \right)d\tau d\overline{\tau}
=
\int_{\partial \cD} (P d\tau + Q d\overline{\tau})
$$
avec $P=0$ et $Q(\tau)=\cW(F)(\tau)\overline{G(\tau)}$,
on obtient
$$\int_{\cD} F(\tau)\overline{G(\tau)} (\tau-\overline{\tau})^{k-2} d\tau d\overline{\tau}
=\int_{\partial \cD} \cW(F)(\tau)
\overline{G(\tau) d\tau}\ . $$
Le contour $\partial \cD$ est formé des arcs géodésiques $a \in \Vmod$.
Comme
$a^*=\gamma_{a}^{-1}a^-$ avec $\gamma_{a} \in \Gamma$, la contribution
de $a$ et $a^*$ pour $a\in \Vmod$ est par le lemme \ref{lem:diff}
\begin{equation*}
\begin{split}
\int_{a} \cW(F)(\tau)
\overline{G(\tau) d\tau}
 + \int_{a^*} \cW(F)(\tau)
\overline{G(\tau) d\tau}
&=
\int_a \cW(F)(\tau)\overline{G(\tau) d\tau}
-\int_{\gamma_{a}^{-1} a}
\cW(F)(\tau)\overline{G(\tau) d\tau}
\\&=
\langle
\cocycleF{F}{\gamma_a^{-1}},
\overline{\Per(G)(a)}\rangle_{V_k}
\ .
\end{split}
\end{equation*}
On a finalement
\begin{equation*}
\begin{split}
\int_{\partial \cD} \cW(F)(\tau)\overline{G(\tau) d\tau}
&=
\frac{1}{2}\sum_{a\in \Vmod}\langle
\cocycleF{F}{\gamma_a^{-1}}, \overline{\Per(G)(a)}
\rangle_{V_k}
\\
&=
\left\lbrace \cocycle(F),\overline{\Per(G)}\right\rbrace_{\Gamma,\cF}
\ .
\end{split}
\end{equation*}
D'où la première partie du théorème.
Démontrons maintenant que
$\left\lbrace \cocycle(F),\Per(G)\right\rbrace_{\Gamma,\cF}=0$.
Introduisons pour cela la fonction holomorphe $\cW_1(F)$ sur $\cH$
donnée par
$$\cW_1(F)(\tau)= \int_{\infty}^\tau \widetilde{F}(t) (t-\tau)^{k-2} dt
+ a_0(F)\int_0^\tau (t-\tau)^{k-2} dt\ .$$
Le même argument que précédemment donne que
\begin{equation*}
\begin{split}
\left\langle \cocycle(F),\Per(G)(a)\right \rangle_{V_k}
&=
\int_a \cW_1(F)(\tau) G(\tau) d\tau -
\int_{\gamma_a^{-1}a}\cW_1(F)(\tau) G(\tau) d\tau
\end{split}
\end{equation*}
pour $a \in \Vmod$
et, en utilisant $
\langle(tx +y)^{k-2},(\tau x + y)^{k-2}\rangle_{V_k}=(t-\tau)^{k-2}$, que
\begin{equation*}
\begin{split}
\left\lbrace \cocycle(F),\Per(G)\right\rbrace_{\Gamma,\cF}=
\int_{\partial \cD} \cW_1(F)(\tau) G(\tau) d\tau
=0
\end{split}
\end{equation*}
car $\cW_1(F)(\tau) G(\tau)$ est holomorphe sur $\cH$.
\end{proof}
Le produit de Petersson sur $S_k(\Gamma)$ induit une forme bilinéaire hermitienne
positive sur $H^1_{par}(\Gamma,V_k(\RR))$. Nous allons voir
que la forme bilinéaire alternée associée est à un scalaire près
la forme $\left\lbrace\cdot,\cdot\right\rbrace_{\Gamma,\cF}$.
\begin{cor}
Soient
$\Phi_1 \in H^1(\Gamma,V_k(\RR))$
et $\Phi_2\in \Hom_\Gamma(\Delta_0, V_k(\RR))$. Il existe
$F_1 \in M_k(\Gamma)$ (resp. $F_2 \in S_k(\Gamma)$) tel que
$\Phi_1=\cocycle_\RR(F_1)$
(resp. $\Phi_2 - \Per_{\RR}(F_2) \in \Hom_\Gamma(\Delta, V_k(\RR))$).
Alors,
$$
\left\lbrace\Phi_1, \Phi_2\right\rbrace_{\Gamma,\cF}
= (2i)^{k-2}\im(\langle F_1,F_2\rangle_{\Gamma})\ .
$$
\end{cor}
\begin{proof}
Les formules ont bien un sens car
$\Hom_\Gamma(\Delta, V_k(\RR))$
est contenu dans radical de $\left\lbrace\cdot , \cdot\right\rbrace_{\Gamma,\cF}$
à droite comme à gauche.
On déduit du théorème \ref{thm:petersson1} que
\begin{small}
\begin{equation*}
\begin{split}
\left\lbrace\Phi_1, \Phi_2\right\rbrace_{\Gamma,\cF}&=
\left\lbrace \cocycle_\RR(F_1),\Per_\RR(F_2)\right\rbrace_{\Gamma,\cF}\\
&=\frac{1}{4}\left (
\left\lbrace \cocycle(F_1),\Per(F_2)\right\rbrace_{\Gamma,\cF}
+
\left\lbrace \cocycle(F_1),\overline{\Per(F_2)}\right\rbrace_{\Gamma,\cF}
+\left\lbrace \overline{\cocycle(F_1)},\Per(F_2)\right\rbrace_{\Gamma,\cF}
+\left\lbrace \overline{\cocycle(F_1)},\overline{\Per(F_2)}\right\rbrace_{\Gamma,\cF}
\right )\\
&=\frac{1}{4}\left (
\left\lbrace \cocycle(F_1),\overline{\Per(F_2)}\right\rbrace_{\Gamma}
+\overline{\left\lbrace \cocycle(F_1),\overline{\Per(F_2)}\right\rbrace_{\Gamma}}
\right )
\\
&=
(2i)^{k-2}\frac{\langle F_1,F_2\rangle_{\Gamma}
-\overline{\langle F_1,F_2\rangle}_{\Gamma}}{2i}=
(2i)^{k-2} \im \left(\langle F_1,F_2\rangle_{\Gamma}\right)\ .
\end{split}
\end{equation*}
\end{small}
\end{proof}
\begin{cor}
\label{independance}
La forme bilinéaire $\left\lbrace \cdot,\cdot \right\rbrace_{\Gamma}=
\left\lbrace \cdot,\cdot \right\rbrace_{\Gamma,\cF}$ définie sur
$H^1(\Gamma,V_k) \times \Hom_{\Gamma}(\Delta_0,V_k)$
est indépendante du symbole de Farey utilisé pour la définir.
\end{cor}
\subsection{Construction d'un symbole de Farey d'un sous-groupe}
Soit $\cF=(\cV,*,\mu)$ un symbole de Farey étendu pour $\Gamma$.
Des algorithmes pour construire un symbole de Farey $\cF'$ pour un sous-groupe
$\Gamma'$ de $\Gamma$ sont bien connus depuis longtemps. Cependant,
ils sont en général écrits pour $\Gamma=\Psl_2(\ZZ)$.
En cherchant à démontrer le comportement de la forme bilinéaire
sous les opérateurs de Hecke, nous avons été amenés à reprendre ces algorithmes
et à observer un phénomène particulier lorsque la courbe modulaire
associée à $\Gamma$ a plusieurs points elliptiques d'ordre 3.
Ce paragraphe reprend donc cette construction.

Avant de passer à l'algorithme de construction de $\cF'$,
donnons quelques notations.
Soit $\cC$ un système de représentants de $\Gamma'\backslash \Gamma$.
Soit $\cW=\cC \times \cV$.
Notons $\cl{\gamma}$ le représentant dans $\cC$ de la classe $\Gamma' \gamma$.
On a $\cl{\cl{\gamma}\gamma'}=\cl{\gamma \gamma'}$.
L'application qui à $(\xi, a)\in \cW$ associe l'arc géodésique
$\xi a$ est injective car un arc de $\cV$ ne peut être le transformé
par un élément de $\Gamma$ d'un autre arc de $\cV$ et le stabilisateur
d'un arc est réduit à $\pm \textrm{Id}$.
Nous identifions dans la suite $(\xi, a)$ à $\xi a$ pour $\xi \in \cC$ et $a\in \cV$.
Pour $\xi a \in \cW$, posons
  $$\gammap_{\xi a}=\xi \gamma_a \cl{\xi \gamma_a}^{-1}\in \Gamma'$$
appelée donnée de recollement de $\cW$ associée à $\xi a$
et $\permm{\gamma_a}$ la permutation de $\cC$ donnée par
 $\perm{\gamma_a}{\xi}=\cl{\xi \gamma_a}$. On a donc
$$
\xi \gamma_a=\gamma_{\xi a} \perm{\gamma_a}{\xi}\ .$$
L'ensemble $\cW$ est muni d'une bijection $Ast$ induite par l'involution $*$
de $\cF$ de la manière suivante :
$$
Ast(\xi a)= \cll{\xi}{\gamma_a} a^* \ . $$
Lorsque $a$ n'est pas un arc elliptique,
$Ast \circ Ast (\xi a)= \xi a$
car
$$Ast(\cl{\xi \gamma_a} a^*)= \cl{\xi \gamma_a\gamma_{a^*}} a
=\cl{\xi} a=\xi a$$
puisque $\gamma_{a^*}=\gamma_a^{-1}$ et $\xi\in \cC$.
On a de plus
$$\xi a = \xi \gamma_a {a^*}^-= \gamma_{\xi a} \cll{\xi}{\gamma_a} {a^*}^-
= \gamma_{\xi a} Ast(\xi a)^-\ .$$
Lorsque $a$ est un arc elliptique d'ordre 2, on a $a=a^*$ et
$Ast\circ Ast(\xi a)=\cl{\cl{\xi \gamma_a}\gamma_a} a= \xi a$.
Lorsque $a$ est un arc elliptique d'ordre 3, $\gamma_a^3$ est $\pm \id$ et
on a
$$
Ast^3(\xi a)=
 Ast(\cl{\cll{\xi}{\gamma_a}\gamma_a} a)=
 Ast(\cll{\xi}{\gamma_a^2} a)=\cll{\xi}{\gamma_a^3} a=\xi a
 \ .
$$
Ainsi, les orbites des éléments de $\cW$ par $Ast$ sont d'ordre 1, 2 ou 3.
Remarquons que
lorsque $Ast(\xi a)=\xi a$, on a $a=a^*$,
$\cll{\xi}{\gamma_a}=\xi$,
$\gammap_{\xi a}=\xi \gamma_a\cl{\gamma_a\xi}^{-1}
=\xi \gamma_a\xi^{-1}$ appartient à $\Gamma'$ et est de même ordre $\mu_{ell}(a)$ que
$\gamma_a$.

\begin{prop} \label{prop:sousgroupe}
On garde les notations précédentes.
Il existe un système de représentants $\cC$ de
$\Gamma'\backslash \Gamma$ tel que $\cF'=(\cV',Ast',\mu')$ décrit comme suit soit
un symbole de Farey: les arcs de $\cV'$ sont
\begin{enumerate}
\item
les éléments $\xi a \in \cC\times (\cV - \cV_{ell,3})$ avec $\gammap_{\xi a}\neq 1$;
on a alors $\mu'(\xi a)=\mu(a)$ et $Ast'(\xi a)=Ast(\xi a)$;
\item les points fixes $\xi a$ dans $\cC\times \cV_{ell,3}$ pour $Ast$,
on a alors $\perm{\gamma_a}{\xi}=\xi$, $\gammap_{\xi a}=\xi \gamma_{a} \xi^{-1}$,
$\mu'(\xi a)=\mu(a)=3$ et $Ast'(\xi a)=Ast(\xi a)=\xi a$;
\item pour chaque orbite $\{\xi_1 a,\xi_2 a,\xi_3 a\}$ d'ordre 3 par $Ast$,
les chemins $\xi_1 a'$, $\xi_1 a''$, $\xi_2 a,\xi_3 a$ où
$\xi_1 a'$ et $\xi_1 a''$
sont définis par
\begin{equation*}
\begin{cases}
\xi_1 a'&=\gammap_{\xi_1a}\xi_2 a^-
\\
\xi_1 a''&=\gammapp_{\xi_3a}^{-1}\xi_3 a^- \ ;
\end{cases}
\end{equation*}
on a alors
\begin{equation*}
\begin{split}
&\gammap_{\xi_1 a}=\xi_1 \gamma_a \xi_2^{-1}
,\quad
\gammap_{\xi_2 a}=\xi_2 \gamma_a \xi_3^{-1}
,\quad
\gammap_{\xi_3 a}=\xi_3 \gamma_a \xi_1^{-1}
,\\
&\perm{\gamma_a}{\xi_1}=\xi_2
,\quad
\perm{\gamma_a}{\xi_2}=\xi_3
,\quad
\perm{\gamma_a}{\xi_3}=\xi_1
,\\
&\gammap_{\xi_1 a'}=\gamma_{\xi_1a}
,\quad \gammap_{\xi_1 a''}=
\gamma_{\xi_3 a}^{-1} ,\quad
\gammap_{\xi_2 a}=\gamma_{\xi_1a}^{-1}
,\quad \gammap_{\xi_3 a}=\gamma_{\xi_3a}
,\\&
Ast'(\xi_1 a')= \xi_2 a ,\quad Ast'(\xi_1 a'')= \xi_3 a
,\quad Ast'(\xi_2 a)=\xi_1 a'
,\quad Ast'(\xi_3 a)=\xi_2 a'
,\\&
\mu'(\xi_1 a')=\mu'(\xi_1 a'')=\mu'(\xi_2 a)=\mu'(\xi_3 a)=1\ .
\end{split}
\end{equation*}
\end{enumerate}
\end{prop}
\begin{proof}
Soit un symbole de Farey $\cF=(\cV,*,\mu_{ell})$.
À chaque étape de l'algorithme,
les éléments de $L$ et $L_3$ sont par construction dans $\cW$
et les éléments de $\cC$ ont des classes distinctes dans $\Gamma'\backslash \Gamma$.

On note $\dot{\cC}$ la réunion des classes $\Gamma' \xi$ pour $\xi \in \cC$
et $\concat$ l'addition d'un élément à la fin d'une liste.

\noindent \rule{\textwidth}{1pt}\vskip1mm\hrule
 \begin{algorithmic}[1]
\REQUIRE un symbole de Farey de groupe $\Gamma$ : $(\cV, *,\mu_{ell})$
et un critère d'appartenance à $\Gamma'$.
\ENSURE Un symbole de Farey $(\cV',Ast',\mu'_{ell})$ de $\Gamma'$
et un système de représentants $\cC$ de $\Gamma'\backslash \Gamma$.
\STATE Calculer les données de recollement $\gamma_a$ pour $a \in \cV$.
\STATE $\cC\leftarrow \{Id\}$; $L\leftarrow \{Id\} \times (\cV -\cV_{ell,3})$;
$L_3\leftarrow \{Id\} \times \cV_{ell,3} $; $\cW \leftarrow \{id\} \times \cV$;
$\cV'\leftarrow \{id\} \times \cV$;
\WHILE{$L \cup L_3 \neq \emptyset$}
  \IF{$L_3 \neq \emptyset$}
    \STATE Prendre le premier élément $\xi a$ de $L_3$ (on a donc $a^*=a$)
    et l'enlever de $L_3$.
    \IF{$\xi\gamma_a \notin \dot{\cC}$}
      \STATE $\cC\leftarrow \cC \cup \{\xi\gamma_a,\xi\gamma_a^2\}$;
      $\cW \leftarrow \cC \times \cV$.
      \FOR{$b \in \cV_{ell,3}$}
        \STATE $L_3\leftarrow L_3 \concat \xi\gamma_a b $.
      \ENDFOR
      \FOR{$b \in \cV_{ell,3}$}
        \STATE $L_3\leftarrow L_3 \concat \xi\gamma_a^2 b$.
      \ENDFOR
      \FOR{$b \in \cV - \cV_{ell,3}$}
        \STATE $L\leftarrow L \concat \xi\gamma_a b$.
      \ENDFOR
      \FOR{$b \in \cV - \cV_{ell,3}$}
        \STATE $L\leftarrow L \concat \xi\gamma_a^2 b$.
      \ENDFOR
      \STATE Insérer dans $\cV'$ à la place de $\xi a$
      la suite des $\xi\gamma_a v$ pour
      $v$ parcourant $\cV-\{a\}$ à partir du successeur de $a$ de manière circulaire
     puis la suite des $\xi\gamma_a^2 v$ pour
      $v$ parcourant $\cV-\{a\}$ à partir du successeur de $a$ de manière circulaire.
    \label{ligne1}
    \ENDIF
  \ELSE
    \STATE Prendre le premier élément $\xi a$ de $L$ et l'enlever de $L$.
    \IF{$\xi\gamma_a \notin \dot{\cC}$}
      \STATE $\cC\leftarrow \cC \cup \{\xi\gamma_a\}$; $\cW \leftarrow \cC \times \cV$.
      \FOR{$b \in \cV_{ell,3}$}
        \STATE $L_3\leftarrow L_3 \concat \xi\gamma_a b$.
      \ENDFOR
      \FOR{$b \in \cV-\cV_{ell,3}$}
        \STATE $L\leftarrow L \concat \xi\gamma_a b$.
      \ENDFOR
      \STATE Insérer dans $\cV'$ à la place de $\xi a$
      la suite des $\xi\gamma_a v$ pour
      $v$ parcourant $\cV-\{a^*\}$ à partir du successeur de $a^*$
      de manière circulaire. \label{ligne2}
    \ENDIF
  \ENDIF
\ENDWHILE
\FOR{$\xi a \in \cV'$ avec $\xi \in \cC$ et $a \in \cV$ }
 \STATE $Ast'(\xi a) \leftarrow \cl{\xi \gamma_a} {a^*}$ ;
 $\gammap_{\xi a} \leftarrow \xi\gamma_a \cl{\xi \gamma_a}^{-1}$
 \IF{$Ast'(\xi a)=\xi a$}
    \STATE $\mu'_{ell}(\xi a) \leftarrow \mu_{ell}(a)$.
  \ELSE
    \STATE $\mu'_{ell}(\xi a) \leftarrow 1$.
  \ENDIF
 \ENDFOR
\FOR{ chaque orbite d'ordre 3 pour $Ast'$} \label{ligneend}
  \STATE Choisir un élément $A=\xi a$ de l'orbite ;
  $B\leftarrow Ast'(A)$, $C \leftarrow Ast'(B)$.
  \STATE Dans $\cV'$, remplacer $A$ par $A'=\gamma_a A^-$ suivi de $A''=\gamma_a^2 A^-$.
  \STATE $Ast'(A') \leftarrow B$; $Ast'(A'') \leftarrow C$;
  $Ast'(B) \leftarrow A'$; $Ast'(C) \leftarrow A''$.
  \STATE $\mu'_{ell}(A') \leftarrow 1$, $\mu'_{ell}(A'') \leftarrow 1$.
\ENDFOR
\RETURN $\left (\cV',Ast',\mu'_{ell}\right)$.
\end{algorithmic}
\hrule\bigskip

L'algorithme termine si et seulement si $\Gamma'$ est d'indice fini dans $\Gamma$.
En effet, une fois que $\cC$ est de cardinal l'indice de $\Gamma'$ dans
$\Gamma$, $\cC$ est un système de représentants de $\Gamma'\backslash \Gamma$
et on va directement à la ligne \ref{ligneend}.

A la ligne \ref{ligne1}, $\xi$ est dans $\cC$ et $\xi \gamma_a$ et $\xi \gamma_a^2$
viennent d'être rajoutés à $\cC$. On enlève l'élément $\xi a$ de $\cV'$
et on met dans $\cV'$ tous les éléments
de la forme $\xi \gamma_a b$ et de la forme $\xi \gamma_a^2 b$
pour $b$ différent de $a$. Ainsi, l'orbite d'ordre 3 sous $Ast$ formée
des trois éléments $\xi a$, $\xi \gamma_a a$ et
$\xi \gamma_a^2 a$ du $\cW$ réactualisé n'est pas dans $\cV'$.
Tous les autres éléments de $\xi \gamma_a \cV$ et $\xi \gamma_a^2 \cV$
que l'on vient de rajouter à $\cW$ sont dans $\cV'$.

A la ligne \ref{ligne2}, $\xi$ est dans $\cC$ et $\xi \gamma_a$
vient d'être rajouté à $\cC$. On enlève $\xi a$ de $\cV'$
et on met dans $\cV'$ tous les éléments de la forme $\xi \gamma_a b$
pour $b$ différent de $a^*$. Ainsi, l'orbite $(\xi a, \xi \gamma_a a)$
sous $Ast$ qui est d'ordre 2 n'est pas dans $\cV'$ et
$\gammap_{\xi a} =\id$ puisque $\xi$ et $\xi \gamma_a$
sont tous deux dans $\cC$. Tous les autres éléments de $\xi \gamma_a \cV$
que l'on vient de rajouter à $\cW$ sont dans $\cV'$.

À la ligne \ref{ligneend}, $\gamma_{b}$ est égal à l'identité
pour $b\in \cW - \cV'$. La manière dont les arcs géodésiques ont été insérés
dans $\cV'$ assure que la suite des extrémités des éléments de $\cV'$
est dans l'ordre circulaire et définit bien un polygone hyperbolique convexe.

Si $\cV'$ ne contient pas d'orbite d'ordre
3 pour $Ast'$, $(\cV', Ast',\mu')$ définit bien un symbole de Farey
à la ligne \ref{ligneend}.
Sinon, on transforme chaque orbite d'ordre 3 en
quatre chemins deux à deux échangés par $Ast'$. La bijection
$Ast'$ devient alors une involution.
Remarquons que les données de recollement $\gammap_A$, $\gammap_B$,
$\gammap_C$ ont été remplacées par $\gammap_A$ et $\gammap_C$
qui engendrent le même sous-groupe puisque $\gammap_A \gammap_B \gammap_C=\id$.

Le symbole $(\cV', Ast',\mu')$ définit bien alors un symbole de Farey.
Montrons qu'il est associé au groupe $\Gamma'$. Pour cela,
nous devons montrer d'après \cite{kulkarni} que le groupe engendré
par les données de recollement $\gammap_{\xi a}$
pour $\xi a \in \cV'$ est $\Gamma'$ à la ligne \ref{ligneend}.
Tout élément $\gamma$ de $\Gamma$ est de la forme
$\gamma_{a_1} \cdots \gamma_{a_i} \cdots \gamma_{a_n}$
avec $a_i \in \cV$. Définissons par récurrence la suite d'éléments de $\cC$ par
$\delta_1=1$, $\delta_{k+1}=\cl{\delta_k \gamma_{a_k}}$
et posons $b_k=\delta_{k}a_k$. Par définition,
$\gammap_{b_k}= \delta_k \gamma_{a_k}\delta_{k+1}^{-1}$
appartient à $\Gamma'$. De plus,
si $b_k$ n'appartient pas à $\cV'$, $\gammap_{b_k}$ est égal à 1
et peut être supprimé.
On voit facilement que
$\delta_{k+1}$ est un représentant de
$\gamma_{a_1}\cdots \gamma_{a_k}$.
Comme
$\gamma_{a_1}\cdots \gamma_{a_n}$ appartient à $\Gamma'$, on a
$\delta_{n+1}=1$.
On a donc
$$
  \gamma =\delta_1 \gamma_{a_1} \delta_2^{-1}
  \delta_2 \gamma_{a_1} \delta_3^{-1}
  \cdots
  \delta_k \gamma_{a_k} \delta_{k+1}^{-1}
  \cdots \delta_{n} \gamma_{a_n} \delta_{n+1}^{-1}
  =\gammap_{b_1} \cdots \gammap_{b_k} \cdots \gammap_{b_n} \ .
$$
Donc, les $\gammap_{b_k}$ engendrent $\Gamma'$, $\cF'$ est un symbole de Farey
de groupe $\Gamma'$.
\end{proof}

\begin{rem}
Lorsqu'il existe un test effectif d'appartenance à $\Gamma'$
pour un élément de $\Gamma$, on peut ainsi construire effectivement un symbole de Farey et
un polygone fondamental
associé à un sous-groupe $\Gamma$ d'indice fini de $\Psl_2(\ZZ)$
à partir d'un symbole de Farey associé à $\Psl_2(\ZZ)$.
\end{rem}
\begin{rem}
Si $\cD$ est le domaine fondamental associé à $\cF$, la réunion
des $\xi \cD$ pour $\xi \in \cC$ est un domaine fondamental associé à
$\Gamma'$.
Malheureusement, l'application $Ast'$ n'est pas toujours une involution
sur $\cV'$ et à la ligne \ref{ligneend}, on n'a pas toujours obtenu
un symbole de Farey. Il faut donc faire une "rectification" pour obtenir
un symbole de Farey. Ce cas ne peut pas se produire si $\cF$ n'a qu'un seul chemin
elliptique d'ordre 3, par exemple si $\cF$ est
un symbole de Farey associé à $\Psl_2(\ZZ)$.
\end{rem}
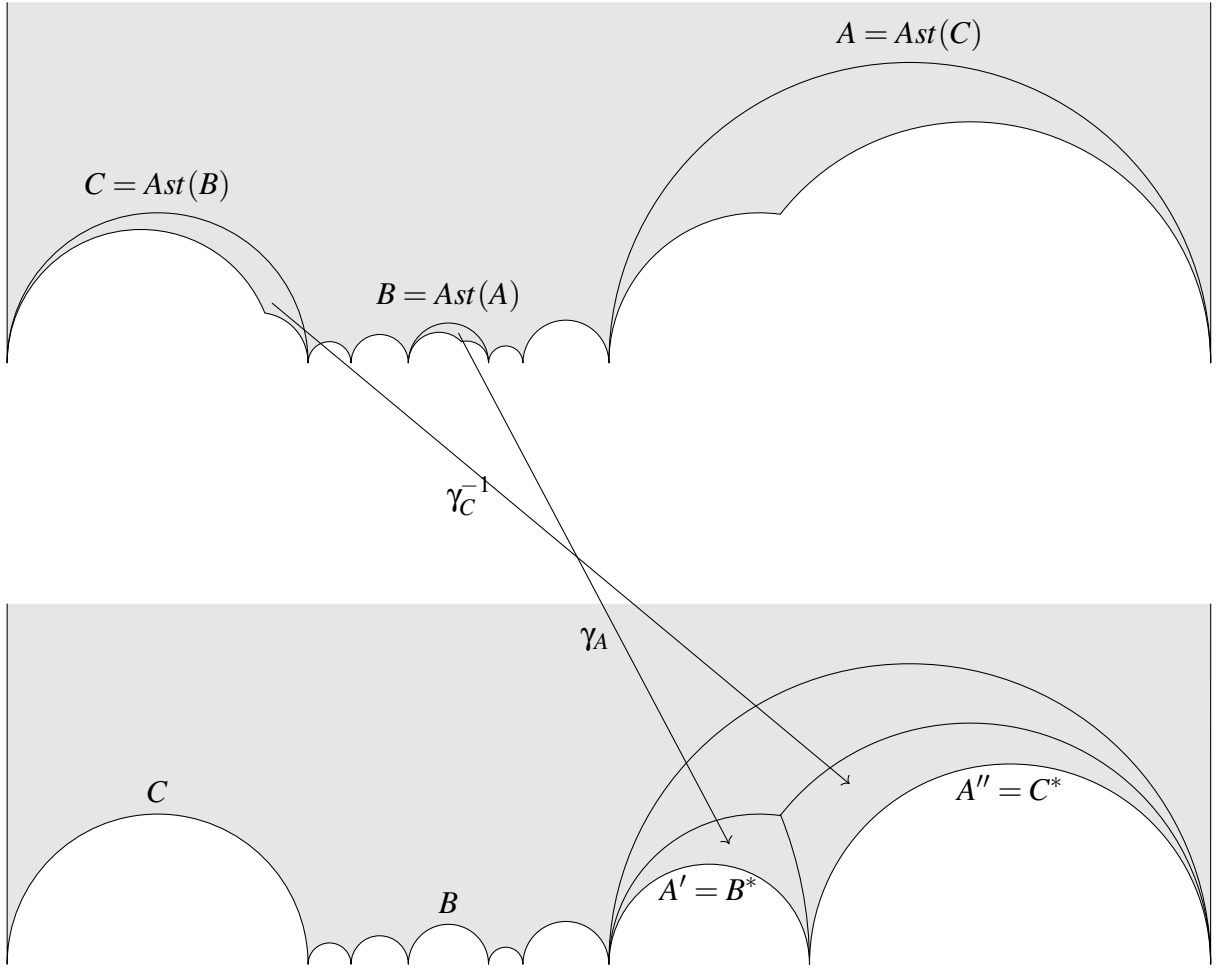
\begin{center}
\begin{tikzpicture}[scale=16]
\fill [black!10!white] (0,4/5) -- (0,1/2) arc (180:22:1/9) arc (82:0:1/24) arc (180:0:1/56) arc (180:0:1/42)
  arc (180:44:1/39) arc (104:0:1/55) arc (180:0:1/70) arc (180:0:1/28) arc (180:82:1/8) arc (142:0:1/5) -- (1,4/5);
\fill [black!10!white] (0,3/10) -- (0,0) arc (180:0:1/8) arc (180:0:1/56) arc (180:0:1/42)
  arc (180:0:1/30) arc (180:0:1/70) arc (180:0:1/28) arc (180:0:1/12) arc (180:0:1/6) -- (1,3/10);

\draw (0,4/5) -- (0,1/2) arc (180:22:1/9) arc (82:0:1/24);
\draw (0,1/2) arc (180:0:1/8) node [midway,above] {$C=Ast(B)$}
  arc (180:0:1/56) arc (180:0:1/42) arc (180:44:1/39) arc (104:0:1/55);
\draw (1/3,1/2) arc (180:0:1/30) node [midway,above] {$B=Ast(A)$}
  arc (180:0:1/70) arc (180:0:1/28) arc (180:82:1/8) arc (142:0:1/5);
\draw (1/2,1/2) arc (180:0:1/4) node [midway,above] {$A=Ast(C)$} -- (1,4/5);

\draw (0,3/10) -- (0,0) arc (180:0:1/8) node [midway,above] {$C$} arc (180:0:1/56)
  arc (180:0:1/42) arc (180:0:1/30) node [midway,above] {$B$} arc (180:0:1/70) arc (180:0:1/28)
  arc (180:0:1/4) -- (1,3/10);
\draw (1/2,0) arc (180:82:1/8) arc (142:0:1/5);
\draw (1/2,0) arc (180:0:1/12) node [midway, below] {$A'=B^*$}
  arc (180:0:1/6) node [midway, below] {$A''=C^*$};
\draw (2/3,0) arc (0:22:1/3);

\draw[->] (0.375,0.525) -- (0.6,0.1) node [pos=0.6,below,left] {$\gamma_A$};
\draw[->] (0.22,0.55) -- (0.7,0.15) node [pos=0.4,below,left] {$\gamma_C^{-1}$};
\end{tikzpicture}
\end{center}
\subsection{Conjugaison et opérateurs de Hecke}
Nous supposerons désormais que $\left\lbrace \cdot,\cdot \right\rbrace_{\Gamma,\cF}$
ne dépend pas de $\cF$, ce qui est vrai lorsque $V$ est $V_k$
muni de la forme bilinéaire $\langle \cdot,\cdot \rangle_{V_k}$.

\begin{prop} Soit $\epsilon= \smallmat{-1&0\\0&1}$.
Pour $\Phi_1 \in H^1(\epsilon\Gamma\epsilon^{-1},V_k)$ et
$\Phi \in \Hom_{\Gamma}(\Delta_0,V_k)$,
\begin{equation*}
\begin{split}
\left\lbrace\Phi_1|\epsilon,\Phi\right\rbrace_{\Gamma}&=
(-1)^{k-1}\left\lbrace\Phi_1,\Phi|\epsilon^{-1}\right\rbrace_{\epsilon\Gamma\epsilon^{-1}}\ .
\end{split}
\end{equation*}
En particulier, pour $f \in \hypk{\Fonct{N}{\QQ}}$,
\begin{equation*}
\left\lbrace\Psi_k(f|\epsilon),\Phi\right\rbrace_{\Gamma}
=\left\lbrace\Psi_k(f),\Phi|_k\epsilon^{-1}\right\rbrace_{\epsilon\Gamma\epsilon^{-1}} \ .
\end{equation*}
\end{prop}
\begin{proof}
Si $\cF=(\cV,*,\mu)$ est un symbole de Farey étendu de groupe $\Gamma$,
on obtient un symbole de Farey étendu $\epsilon \cF=(\cV',*, \mu')$
pour $\epsilon \Gamma \epsilon^{-1}$
défini par $\cV'=\epsilon \overline{\cV}$, $(\epsilon \overline{a})^*= \epsilon \overline{a^*}$,
$\gamma_{\epsilon \overline{a}}=\epsilon\gamma_a\epsilon^{-1}$ et
$\mu'(\epsilon \overline{a})=\mu(a)$.
On a
\begin{equation*}
\begin{split}
\left\lbrace\Phi_1|\epsilon,\Phi\right\rbrace_\Gamma&=
\frac{1}{2}\sum_{a\in \Vmod} \langle \Phi_1(\epsilon\gamma_a^{-1}\epsilon^{-1})|\epsilon,
  \Phi(a)\rangle_{V_k}=
\frac{(-1)^k}{2}\sum_{a\in \Vmod} \langle \Phi_1(\epsilon\gamma_a^{-1}\epsilon^{-1}),
  \Phi|_k\epsilon^{-1}(\epsilon a)\rangle_{V_k}\\
&=
\frac{(-1)^{k-1}}{2}\sum_{b\in \Vmod'} \langle \Phi_1(\gamma_b^{-1}),
  \Phi|_k\epsilon^{-1}(b)\rangle_{V_k} =
(-1)^{k-1} \left\lbrace\Phi_1,\Phi|_k\epsilon^{-1}\right\rbrace_{\epsilon\Gamma\epsilon^{-1}}
\ .\end{split}
\end{equation*}
La deuxième égalité se déduit alors de \eqref{epsilon}.
\end{proof}

Regardons le comportement des formes bilinéaires
$\left\lbrace \cdot,\cdot \right\rbrace_{\Gamma}$
par rapport à un opérateur de Hecke
$[\Gamma_1 \alpha \Gamma_2]$ où $\alpha$ est un élément de
$M_2(\ZZ)^+=\Gl_2(\QQ)^+ \cap M_2(\ZZ)$
et où $\Gamma_1$ et $\Gamma_2$ sont deux sous-groupes d'indice fini de $\Sl_2(\ZZ)$.

Écrivons $\Gamma_1 \alpha \Gamma_2 = \sqcup_{\xi\in \cC} \Gamma_1 \alpha \xi$
où $\cC$ est un système de représentants
de $\left(\Gamma_2 \cap \alpha^{-1} \Gamma_1 \alpha\right )\backslash \Gamma_2$.
Les définitions bien connues qui suivent ne dépendent pas du choix de $\cC$.
Nous utiliserons ensuite un système de représentants
obtenu par l'algorithme de la proposition \ref{prop:sousgroupe}.

On définit l'application $[\Gamma_1 \alpha \Gamma_2]$
$$\Hom_{\Gamma_1}(\Delta_0,V) \to \Hom_{\Gamma_2}(\Delta_0,V)$$
donnée par
$$\Phi|[\Gamma_1 \alpha \Gamma_2]=\sum_{\xi\in \cC} \Phi|\alpha\xi \ .$$
Si $\widetilde{\Phi}$ est un $1$-cocycle de $\Gamma_1$ à valeurs dans $V$, on définit
le $1$-cocycle $\widetilde{\Phi}|[\Gamma_1 \alpha \Gamma_2]$ de $\Gamma_2$ par
$$ (\widetilde{\Phi}|[\Gamma_1 \alpha \Gamma_2])(\gamma) =
\sum_{\xi\in \cC} \widetilde{\Phi}(\gamma_{\xi}) | \alpha \perm{\gamma}{\xi}
$$
si $\alpha \xi\gamma = \gamma_{\xi} \alpha \perm{\gamma}{\xi}$
avec $\gamma_{\xi} \in \Gamma_1$ et $\permm{\gamma}$ une permutation de $\cC$.
On a $\permm{\gamma^{-1}}=\permm{\gamma}^{-1}$ et
$(\gamma^{-1})_{\perm{\gamma}{\xi}}=(\gamma_\xi)^{-1}$.
On en déduit que
$$ (\widetilde{\Phi}|[\Gamma_1 \alpha \Gamma_2])(\gamma^{-1}) =
\sum_{\xi} \widetilde{\Phi}(\gamma_{\xi}^{-1}) | \alpha \xi \ .
$$
Si $\alpha \in M_2(\ZZ)^+$ et $\widetilde{\Phi} \in Z^1(\Gamma, V)$, on note
$\widetilde{\Phi}|\alpha$ le cocycle de $Z^1(\alpha^{-1} \Gamma \alpha, V)$ déduit
par transport de structure : pour
$\gamma \in \alpha^{-1} \Gamma \alpha$
$$ \widetilde{\Phi}|\alpha( \gamma) = \widetilde{\Phi}(\alpha \gamma \alpha^{-1})| \alpha \ .$$

\begin{prop}
\label{prop:hecke}
Soient $\widetilde{\Phi}_1 \in Z^1(\Gamma_1,V)$
et $\Phi_2\in \Hom_{\Gamma_2}(\Delta_0,V)$. Si $\alpha \in M_2(\ZZ)^+$, on a
$$
\left\lbrace \widetilde{\Phi}_1|[\Gamma_1 \alpha \Gamma_2], \Phi_2 \right\rbrace_{\Gamma_2}
=
\left\lbrace \widetilde{\Phi}_1,
\Phi_2|[\Gamma_2 \alpha^* \Gamma_1] \right\rbrace_{\Gamma_1} \ .
$$
\end{prop}
On utilise le fait que l'accouplement $\langle \cdot,\cdot\rangle_{V}$
vérifie
$\langle \gamma v_1,v_2\rangle_{V}=\langle v_1, \gamma^*v_2\rangle_{V}$
où
$\gamma^*=\det(\gamma) \gamma^{-1}$ pour $\gamma \in \Gl_2(\QQ)$.

La proposition \ref{prop:hecke} se déduira des lemmes \ref{heckecomp0}, \ref{heckecomp}
et du corollaire \ref{independance}.
\begin{lem}
\label{heckecomp0}
Soient $\widetilde{\Phi}_1 \in Z^1(\alpha\Gamma\alpha^{-1},V)$ et
$\Phi_2 \in \Hom_{\Gamma}(\Delta_0,V)$.
Alors,
\begin{equation*}
\begin{split}
  \left\lbrace \widetilde{\Phi}_1|\alpha,
  \Phi_2 \right\rbrace_{\Gamma}&=
  \left\lbrace \widetilde{\Phi}_1,
  \Phi_2|\alpha^* \right\rbrace _{\alpha\Gamma\alpha^{-1}}
\ .\end{split}
\end{equation*}
\end{lem}
\begin{proof}
Soit $\cF=(\cV,Ast,\mu)$ un symbole de Farey pour $\Gamma$. Alors,
$(\alpha \cV,Ast',\mu')$ est un symbole de Farey pour $\alpha \Gamma \alpha^{-1}$
avec la définition naturelle de $Ast'$ et de $\mu'$.
On a alors
\begin{equation*}
\begin{split}
\left\lbrace \widetilde{\Phi}_1|\alpha,
  \Phi_2 \right\rbrace_{\Gamma}&=
 \frac{1}{2} \sum_{a\in \Vmod} \langle (\widetilde{\Phi}_1|\alpha)(\gamma_a^{-1})
 ,\Phi_2(a)\rangle_{V}
=\frac{1}{2}
  \sum_{a \in \Vmod} \langle \widetilde{\Phi}_1(\alpha\gamma_a^{-1}\alpha^{-1})|\alpha
 ,\Phi_2(a)\rangle_{V}
 \\
 &=\frac{1}{2}
  \sum_{a\in \Vmod} \langle \widetilde{\Phi}_1(\alpha\gamma_a^{-1}\alpha^{-1})
 ,(\Phi_2| \alpha^*) (\alpha a)\rangle_{V}
 =
 \frac{1}{2}
  \sum_{b\in \alpha \Vmod} \langle \widetilde{\Phi}_1(\gamma_{b}^{-1})
 ,(\Phi_2| \alpha^*) (b)\rangle_{V}
 \\
&=
\left\lbrace \widetilde{\Phi}_1,
  \Phi_2|\alpha^* \right\rbrace_{\alpha\Gamma\alpha^{-1}}
\end{split}
\end{equation*}
car $ \alpha \Vmod=\widetilde{\alpha \cV}$.
\end{proof}
Donnons quelques notations valables pour les lemmes suivants.
Choisissons un symbole de Farey $\cF=(\cV, *, \mu)$ de
 $\Gamma_2$ et soit le symbole de Farey $\cF'=(\cV', *, \mu')$ de
$\Gamma_2 \cap \alpha^{-1} \Gamma_1 \alpha$
construit à partir de $\cF$ dans la proposition \ref{prop:sousgroupe}
dont on reprend les notations.
On pose $\Wmod=\cC \times \Vmod$.
On a pour $a\in\Vmod$
$$\delta_{\xi a}=\alpha\gamma_{\xi a}\alpha^{-1}
\in \Gamma_1 \cap \alpha \Gamma_2 \alpha^{-1}$$
et
$\alpha \perm{\gamma_a}{\xi} \gamma_a^{-1}=\delta_{\xi a}^{-1} \alpha \xi$
que
$\alpha\xi \gamma_a=\delta_{\xi a} \alpha \perm{\gamma_a}{\xi}$.
On a
\begin{equation}\label{decomp}
\cC \times \Vmod= \cC \times (\cV - \cV_{ell,3})
  \sqcup \widetilde{\cV'}_{ell,3} \sqcup \cW_3
\end{equation}
 où $\cW_3$ est la réunion des
  $\widetilde{\xi a\ }=\{\xi u_a, \xi v_a\}$ pour $\xi a $ dans une orbite d'ordre 3
  dans $\cC\times \cW$.
\begin{lem} \label{orbite3}
Pour une orbite $\{A,B,C\}$ d'ordre 3 pour $Ast'$ dans $\cW$,
on a pour $\Phi=\Phi_1|\alpha$
$$
\sum_{P \in \widetilde{\{A,B,C\}\ }}
\langle
\widetilde{\Phi}(\gamma_{P}^{-1}),
\Phi_2(P) \rangle_{V}=
\sum_{P \in \{A',A'', B, C\}}
\langle
\widetilde{\Phi}(\gamma_{P}^{-1}),
\Phi_2(P) \rangle_{V}\ .
$$
\end{lem}
\begin{proof}
La contribution de
$\widetilde{\{A, B, C\}\ }$
d'ordre 3 sous $Ast$ est

\begin{equation*}
\begin{split}
\frac{1}{2}\sum_{P \in \widetilde{\{A,B,C\}\ }}&
\langle \widetilde{\Phi}(\gamma_{P}^{-1}), \Phi_2(P) \rangle_{V}=
\langle \Phi(\gamma_{u_A}^{-1}),\Phi(u_A) \rangle_V
+
\langle \Phi(\gamma_{u_C}^{-1}),\Phi(u_C) \rangle_V
+
\langle \Phi(\gamma_{v_C}^{-1}),\Phi(v_C) \rangle_V \\
&=\frac{1}{3}
\left(
\langle \Phi(\gamma_{A}^{-1}),\Phi(A)+\Phi(A') \rangle_V
+
\langle \Phi(\gamma_{C}^{-1}),\Phi(C) + \Phi(C') \rangle_V
+
\langle \Phi(\gamma_{A}^{-1}\gamma_{C}^{-1}),\Phi(C) + \Phi(C'')\rangle_V
\right)\ .
\end{split}
\end{equation*}
On a
\begin{equation*}
\begin{split}
\langle \Phi(\gamma_{A}^{-1}\gamma_{C}^{-1}),\Phi(C) + \Phi(C'')\rangle_V
&=
\langle \Phi(\gamma_{A}^{-1}),\Phi(\gamma_{C}^{-1}C) + \Phi(\gamma_{C}^{-1}C'')\rangle_V
+
\langle \Phi(\gamma_{C}^{-1}),\Phi(C) + \Phi(C'')\rangle_V
\\
&=
\langle \Phi(\gamma_{A}^{-1}),-\Phi(A'') + \Phi(A')\rangle_V
+
\langle \Phi(\gamma_{C}^{-1}),\Phi(C) + \Phi(C'')\rangle_V
\ .\end{split}
\end{equation*}
Donc
\begin{equation*}
\begin{split}
\frac{1}{2}\sum_{P \in \widetilde{\{A,B,C\}\ }}
\langle \widetilde{\Phi}(\gamma_{P}^{-1}), \Phi_2(P) \rangle_{V}
&=\frac{1}{3}
\left(
\langle \Phi(\gamma_{A}^{-1}),3\Phi(A') \rangle_V
+
\langle \Phi(\gamma_{C}^{-1}),3\Phi(C)\rangle_V
\right)
\\
&=
\langle \Phi(\gamma_{A'}^{-1}),\Phi(A') \rangle_V
+
\langle \Phi(\gamma_{C}^{-1}),\Phi(C)\rangle_V \\
&=\frac{1}{2}
\sum_{P\in \{ A',A'',B,C\}}
\langle \widetilde{\Phi}_1(\gamma_{P}^{-1}),\Phi_2(P)\rangle_V
\ ,
\end{split}
\end{equation*}
ce qui termine la démonstration du lemme.
\end{proof}
\begin{lem}
\label{heckecomp}
Soient $\widetilde{\Phi}_1 \in Z^1(\Gamma_1,V)$ et $\Phi_2 \in \Hom_{\Gamma_2}(\Delta_0,V)$.
Alors,
\begin{equation*}
\begin{split}
\left\lbrace \widetilde{\Phi}_1|[\Gamma_1 \alpha \Gamma_2], \Phi_2 \right\rbrace_{\Gamma_2}
&=
\left\lbrace \widetilde{\Phi}_1|\alpha,
  \Phi_2 \right\rbrace_{\Gamma_2 \cap \alpha^{-1} \Gamma_1 \alpha}\\
\left\lbrace \widetilde{\Phi}_1, \Phi_2 |[\Gamma_2 \alpha \Gamma_1]\right\rbrace_{\Gamma_1}
&=
\left\lbrace \widetilde{\Phi}_1,
  \Phi_2|\alpha \right\rbrace_{\Gamma_1 \cap \alpha^{-1} \Gamma_2 \alpha}
\ .\end{split}
\end{equation*}
\end{lem}
\begin{proof}
Montrons la première égalité.
\begin{equation*}
\begin{split}
\left\lbrace \widetilde{\Phi}_1|[\Gamma_1 \alpha \Gamma_2], \Phi_2 \right\rbrace_{\Gamma_2}
&=
 \frac{1}{2}
\sum_{a\in \Vmod}\sum_{\xi\in \cC}
\langle \widetilde{\Phi}_1(\delta_{\xi a}^{-1})|\alpha \xi,
\Phi_2(a) \rangle_{V}
=\frac{1}{2}
\sum_{\xi a\in \cC \times \Vmod}\langle \widetilde{\Phi}_1(\delta_{\xi a}^{-1})|\alpha \xi,
\Phi_2(a) \rangle_{V}
\\
&=\frac{1}{2}
\sum_{\xi a\in \cC \times \Vmod}\langle
(\widetilde{\Phi}_1|\alpha)(\gamma_{\xi a}^{-1}),
\Phi_2(a)|\xi^{-1} \rangle_{V}
=\frac{1}{2}
\sum_{\xi a\in \cC \times \Vmod}\langle
(\widetilde{\Phi}_1|\alpha)(\gamma_{\xi a}^{-1}),
\Phi_2(\xi a)| \rangle_{V}\ .
\end{split}
\end{equation*}
Pour $a\notin \cV_{ell}$ et $\xi a \notin \Vmod'$, la contribution
de $\xi a$ est 0 car $\gamma_{\xi a}=1$. En utilisant la décomposition
\eqref{decomp}
de $\cC \times \Vmod$ et le lemme \ref{orbite3}, on obtient l'expression de
$\left\lbrace \widetilde{\Phi}_1|\alpha,\Phi_2
\right\rbrace_{\Gamma_2 \cap \alpha^{-1} \Gamma_1 \alpha}
$
calculée à l'aide du symbole de Farey étendu $\cF'$ et donc la première égalité du lemme
\ref{heckecomp}
\begin{equation*}
\left\lbrace \widetilde{\Phi}_1|[\Gamma_1 \alpha \Gamma_2], \Phi_2 \right\rbrace_{\Gamma_2}
=
\left\lbrace \widetilde{\Phi}_1|\alpha,\Phi_2
\right\rbrace_{\Gamma_2 \cap \alpha^{-1} \Gamma_1 \alpha}
\end{equation*}
par l'indépendance relative au symbole de Farey étendu.

Montrons la seconde égalité.
Pour ne pas introduire de nouvelles notations, démontrons-la en échangeant $\Gamma_1$
et $\Gamma_2$. Autrement dit, montrons que
si $\widetilde{\Phi}_2 \in Z^1(\Gamma_2,V)$ et $\Phi_1 \in \Hom_{\Gamma_1}(\Delta_0,V)$,
\begin{equation*}
\left\lbrace \widetilde{\Phi}_2, \Phi_1 |[\Gamma_1 \alpha \Gamma_2]\right\rbrace_{\Gamma_2}
=
\left\lbrace \widetilde{\Phi}_2,
  \Phi_1|\alpha \right\rbrace_{\Gamma_2 \cap \alpha^{-1} \Gamma_1 \alpha} \ .
\end{equation*}
On a
\begin{equation*}
\begin{split}
\left\lbrace \widetilde{\Phi}_2, \Phi_1 |[\Gamma_1 \alpha \Gamma_2]\right\rbrace_{\Gamma_2}
&=\frac{1}{2}\sum_{\xi a\in \cC\times \Vmod}
\langle \widetilde{\Phi}_2(\gamma_a^{-1}), \Phi_1 |(\alpha \xi)(a)\rangle_{V}
\\
&=\frac{1}{2}
\sum_{\xi a\in \cC\times \Vmod}
\langle \widetilde{\Phi}_2(\gamma_a^{-1})|\xi^{-1},
(\Phi_1|\alpha)(\xi a)\rangle_{V}
\\&=\frac{1}{2}
\sum_{\xi a\in \cC\times \Vmod}
\langle \widetilde{\Phi}_2(\gamma_a^{-1}\xi^{-1}) - \widetilde{\Phi}_2(\xi^{-1}),
(\Phi_1|\alpha)(\xi a)\rangle_{V} \ .
\end{split}
\end{equation*}
La somme sur $a$ des $(\Phi_1|\alpha)(\xi a)$ est nulle car on somme alors
sur un chemin fermé. Donc,
\begin{equation*}
\begin{split}
\left\lbrace \widetilde{\Phi}_2, \Phi_1 |[\Gamma_1 \alpha \Gamma_2]\right\rbrace_{\Gamma_2}
&=\frac{1}{2}
\sum_{\xi a\in \cC\times \Vmod}
\langle \widetilde{\Phi}_2(\gamma_a^{-1}\xi^{-1}),
(\Phi_1|\alpha)(\xi a)\rangle_{V}
\\&=\frac{1}{2}
\sum_{\xi a\in \cC\times \Vmod}\langle \widetilde{\Phi}_2(\perm{\gamma_a}{\xi}^{-1}\gamma_{\xi a}^{-1}),
(\Phi_1|\alpha)(\xi a)\rangle_{V}
\\&=
\frac{1}{2}\sum_{\xi a\in \cC\times \Vmod}\langle \widetilde{\Phi}_2(\perm{\gamma_a}{\xi}^{-1})|\gamma_{\xi a}^{-1}
+ \widetilde{\Phi}_2(\gamma_{\xi a}^{-1}),
(\Phi_1|\alpha)(\xi a)\rangle_{V}\\
&=\frac{1}{2}\sum_{a\in\Vmod}\sum_{\xi\in \cC}
-\langle \widetilde{\Phi}_2(\perm{\gamma_a}{\xi}^{-1})
,(\Phi_1|\alpha)(\perm{\gamma_a}{\xi} a^*)\rangle_{V}+
\sum_{\xi a\in \cC\times \Vmod}
\langle \widetilde{\Phi}_2(\gamma_{\xi a}^{-1}),
(\Phi_1|\alpha)(\xi a)\rangle_{V}
\\
&=\frac{1}{2}\sum_{a\in\Vmod}\sum_{\xi'\in \cC}
-\langle \widetilde{\Phi}_2(\xi)
,(\Phi_1|\alpha)(\xi a^*)\rangle_{V}+
\sum_{\xi a\in \cC\times \Vmod}
\langle \widetilde{\Phi}_2(\gamma_{\xi a}^{-1}),
(\Phi_1|\alpha)(\xi a)\rangle_{V}
\\
&=\frac{1}{2}\sum_{\xi a\in \cC\times \Vmod}
\langle \widetilde{\Phi}_2(\gamma_{\xi a}^{-1}),
(\Phi_1|\alpha)(\xi a)\rangle_{V}\ .
\end{split}
\end{equation*}
En utilisant la décomposition \eqref{decomp} de $\cC\times \Vmod$ comme précédemment et
le lemme \ref{orbite3}, on en déduit par l'indépendance relative au symbole de Farey étendu que
$$
\left\lbrace \widetilde{\Phi}_2, \Phi_1 |[\Gamma_1 \alpha \Gamma_2]\right\rbrace_{\Gamma_2}
=
\left\lbrace \widetilde{\Phi}_2, \Phi_1|\alpha\right\rbrace_{\Gamma_2 \cap \alpha^{-1} \Gamma_1 \alpha}
\ .$$
\end{proof}

\subsection{Comportement du module d'Eisenstein et du module parabolique}

Si $s \in \PP^1(\QQ)$, notons
$w(s)$ la largeur de la pointe $\Gamma s$.
Le résultat suivant est dans l'esprit de \cite{pasol}.
\begin{prop}
Si $\Phi \in\Hom_{\Gamma}(\Delta_0,V_k)$ est l'image
d'un élément $\Phi^{0}$ de $\Hom_{\Gamma}(\Delta,V_k)$,
$$\Phi^{0}=\sum_{s\in \Gamma \backslash \PP^1(\QQ)}c_s(\Phi^0) Eis_{s}$$
et si $f\in \hypk{\Fonct{\Gamma}{\QQ}}$, on a
$$\left\lbrace\Psi_k(f),\Phi\right\rbrace_{\Gamma}=
\sum_{s\in C(\Gamma)}
w(s) \left (\int f|\gamma_{s} \dd{\beta_k}{\beta_0}\right ) c_s(\Phi^0)
\ .$$
\end{prop}
\begin{proof}
D'après la proposition \ref{noncusp},
on a
\begin{equation*}
\begin{split}
\left\lbrace\Psi_k(f),\Phi\right\rbrace_{\Gamma}&=-
\sum_{s\in C(\Gamma)}
\langle \widetilde{\Psi^{(s)}}(\tau_s), \Phi^{0}(\{s\}) \rangle\ .
\end{split}
\end{equation*}
On choisit le cocycle $\widetilde{\Psi_k(f)^{(s)}}$ défini par
$\widetilde{\Psi_k(f)^{(s)}}(\gamma)=
\Psi_k(f)([Z_s,\gamma^{-1}Z_s])$ avec
 $Z_s=\pi_{s}(\gamma_{s} 0))=\gamma_{s} \pi_\infty(0)$
 et $\gamma_s\infty=s$.
Alors,
\begin{equation*}
\begin{split}
\langle \Psi_k(f)([Z_{s},\tau_s^{-1}Z_{s}],\Phi^0(\{s\})\rangle_{V_k}&=
\langle \Psi_k(f)(\gamma_{s}[\pi_\infty(0),\gamma_{s}^{-1}
\tau_s^{-1}\gamma_{s}
\pi_\infty(0)]),\Phi^0(\{s\})\rangle_{V_k}\\
&=\langle \Psi_k(f|\gamma_{s})([\pi_\infty(0),
\pi_\infty(\gamma_{s}^{-1}\tau_s^{-1}\gamma_{s} 0])|
\gamma_{s}^{-1},\Phi^0(\{s\})\rangle_{V_k}
\\
&=\langle \Psi_k(f|\gamma_{s})([\pi_\infty(0),
\pi_\infty(\gamma_{s}^{-1}\tau_s^{-1}\gamma_{s} 0)]),
\Phi^0(\{s\})|\gamma_{s}\rangle_{V_k} \ .
\end{split}
\end{equation*}
On a $\Phi^0({s})|\gamma_{s}=c_{s}(\Phi^0)x^{k-2}$.
D'autre part, $\tau_s$ est le générateur positif du stabilisateur
de la pointe $s$
et $\gamma_{s}^{-1}\tau_s\gamma_{s} 0$ est égal à
la largeur $w(s)$ de la pointe $s$. Comme
\begin{equation*}
\begin{split}\Psi_k(f|\gamma_{s})([\pi_\infty(0),
\pi_\infty(\gamma_{s}^{-1}\tau_s^{-1}\gamma_{s} 0)])
&=
\Psi_k(f|\gamma_{s})([\pi_\infty(0),-\pi_\infty(w(s))])\\
&=
\frac{1}{k-1}
\left (\int f|\gamma_{s} \dd{\beta_k}{\beta_0}\right )
\frac{(-w(s)x+y)^{k-1}-y^{k-1}}{x}
\ ,
\end{split}
\end{equation*}
on en déduit l'égalité
$$
\langle \Psi_k(f)([Z_{s},\tau_s^{-1}Z_{s}],\Phi^0(\{s\})\rangle_{V_k}=
-w(s)\left (\int f|\gamma_s \dd{\beta_k}{\beta_0}\right )c_{s}(\Phi_0)
$$
et la proposition.
\end{proof}
\begin{cor}\label{cor:eis}
La restriction de $\left\lbrace \cdot,\cdot\right\rbrace_{\Gamma}$
à $\EiskQ{k,\Gamma} \times \Hom_\Gamma(\Delta, V_k)$
(resp. à $\EiskQ{k,\Gamma} \times \Hom_\Gamma(\Delta, V_k)/\QQ$ lorsque $k=2$)
est non dégénérée.
\end{cor}
\begin{proof}
Si $F=\Eis_{k,\Gamma}(f)\in \EiskQ{k,\Gamma}$ est orthogonal à
$\Hom_\Gamma(\Delta, V_k)$, alors $\int f|\gamma_s \dd{\beta_k}{\beta_0}$
est nul pour tout $s\in C(\Gamma)$. On en déduit que $F=0$
(proposition \ref{prop:ortheisenstein}).
Les deux espaces $\EiskQ{k,\Gamma}$ et $\Hom_\Gamma(\Delta, V_k)$
(resp.
$\EiskQ{k,\Gamma}$ et $\Hom_\Gamma(\Delta, V_k)/
\Hom_\Gamma(\QQ, \QQ)$ lorsque $k=2$)
ayant même dimension, le corollaire s'en déduit.
\end{proof}
\begin{prop}\label{prop:dualeisenstein}
L'espace vectoriel $\cS$ des éléments $\Phi \in\Hom_{\Gamma}(\Delta_0,V_k)$
tels que $\left\lbrace\Psi_k(f),\Phi\right\rbrace_{\Gamma}=0$ pour
tout $f \in \hypk{\Fonct{\Gamma}{\QQ}}$ est un sous-espace vectoriel
stable par les opérateurs de Hecke et d'intersection nulle avec
l'image de $\Hom_{\Gamma}(\Delta,V_k)$. Son image dans
$H^1(\Gamma,V_k)$ est $H^1_{par}(\Gamma,V_k)$.
De plus, $\RR\otimes \cS$ est égal à l'image de $S_k(\Gamma)$ dans
$\Hom_{\Gamma}(\Delta,V_k(\RR))$ par $\Per_\RR$.
\end{prop}
\begin{proof}
L'espapce $\cS$ est stable par les opérateurs de Hecke d'après la proposition
\ref{prop:hecke}. Son intersection avec l'image de $\Hom_{\Gamma}(\Delta_0,V_k)$
est nulle par le corollaire \ref{cor:eis}.
Son image dans $H^1(\Gamma,V_k)$ est $H^1_{par}(\Gamma,V_k)$
par la proposition \ref{prop:suiteexacte}.
Enfin, son tensorisé avec $\RR$ contient l'image de $S_k(\Gamma)$ par
$\Per_\RR$ par la proposition \ref{prop:ortheisenstein} et la compatibilité
avec le produit de Petersson.
\end{proof}

\begin{cor}
L'espace vectoriel des éléments $\Phi \in\Hom_{\Gamma_0(N)}(\Delta_0,V_k)$
tels que $\left\lbrace\Psi_k(f),\Phi\right\rbrace_{\Gamma}=0$ pour
tout $f \in \VN$ (resp. $f \in \VN-\{\charun_{(0,0)}\}$ lorsque $k=2$)
est un sous-espace vectoriel stable par les opérateurs de Hecke et d'intersection nulle avec
l'image de $\Hom_{\Gamma}(\Delta,V_k)$. Son image dans
$H^1(\Gamma,V_k)$ est $H^1_{par}(\Gamma,V_k)$.
\end{cor}
\begin{proof}
Le corollaire est une réécriture de la proposition
\ref{prop:dualeisenstein} pour le système de générateurs
donné dans la proposition \ref{gen} pour $\Gamma=\Gamma_0(N)$.
\end{proof}

\subsection{Le cas de $\Sl_2(\ZZ)$}

Appliquons ce qui précède au groupe $\Gamma=\Sl_2(\ZZ)$.
Comme $-\id\in \Gamma$, on peut supposer $k$ pair. Prenons le symbole de Farey
$\cV$ formé de $a_1=a_1^*=\{\infty,0\}$ et $a_2=a_2^*=\{0,\infty\}$ avec les données de recollement
respectives $\sigma= \smallmat{0&-1\\1&0}$ et
$\tau= \smallmat{0&-1\\1&-1}$. Nous aurons aussi besoin de $T=\tau^{-1} \sigma=\smallmat{1&1\\0&1}$.
La formule (\ref{def:petersson1}) s'écrit
\begin{equation*}
\begin{split}
\left\lbrace \Phi_1, \Phi_2\right\rbrace_{\Sl_2(\ZZ)}&=
\frac{1}{2} \langle
  \widetilde{\Phi}_1(\sigma^{-1}),\Phi_2 (\{\infty,0\})\rangle_{V}
+ \frac{1}{3}
\langle \widetilde{\Phi}_1(\tau^{-1})+\widetilde{\Phi}_1(\tau^{-2}),
\Phi_2 (\{0,\infty\})\rangle_{V}
\\
&= \langle
\frac{1}{2}\widetilde{\Phi}_1(\sigma^{-1}) - \frac{1}{3}
\widetilde{\Phi}_1(\tau^{-1})-\frac{1}{3}\widetilde{\Phi}_1(\tau^{-2}),
\Phi_2 (\{\infty,0\})\rangle_{V}\\
\end{split}
\end{equation*}
où $\widetilde{\Phi}_1$ est un 1-cocycle sur $\Gamma=SL_2(\ZZ)$ représentant un élément
$\Phi_1$ de $H^1(\Gamma,V)$ et $\Phi_2\in\Hom_\Gamma(\Delta_0, V)$.
Si $\Phi_1\in\Hom_\Gamma(\KKK_0, V)$, on peut donc encore écrire
\begin{equation*}
\begin{split}
\left\lbrace \Phi_1, \Phi_2\right\rbrace_{\Sl_2(\ZZ)}&=\frac16\langle \Phi_1(3[\pi_\infty(0),\sigma\pi_\infty(0)] -
2[\pi_\infty(0),\tau\pi_\infty(0)]
-2[\pi_\infty(0),\tau^{2}\pi_\infty(0)]),\Phi_2 (\{\infty,0\})\rangle_{V}\\
&=\frac16\langle \Phi_1
\left((1 + 2\tau^2)\{\infty,0\} + 2(\tau -1) [0,1]_\infty\right)
,\Phi_2 (\{\infty,0\})\rangle_{V}
\end{split}
\end{equation*}
en utilisant les relation $\tau^2(\pi_0(\infty))=\pi_\infty(1)$ et $\sigma(\pi_\infty(0))=\pi_0(\infty)=\tau(\pi_\infty(1))$.

On a dans $\Delta_0$ les relations \[(1+\sigma)\{\infty,0\}=(1+\tau+\tau^2)\{\infty,0\}=0\]
d'où l'on déduit, pour tout $\delta\in\KKK_0$
\begin{equation*}
\begin{split}
\langle \Phi_1(\delta)|(1 + \tau + \tau^2)
,\Phi_2 (\{\infty,0\})\rangle_{V}&
=\langle \Phi_1(\delta)
,\Phi_2 (\{\infty,0\})|(1 + \tau + \tau^2)\rangle_{V}=0\\
\langle \Phi_1(\delta)|(1+\sigma)
,\Phi_2 (\{\infty,0\})\rangle_{V}&=\langle \Phi_1(\delta)
,\Phi_2 (\{\infty,0\})|(1+\sigma)\rangle_{V}=0\ .
\end{split}
\end{equation*}
On peut donc réécrire
\begin{equation}
\begin{split}
\label{simplif}
\left\lbrace \Phi_1, \Phi_2\right\rbrace_{\Sl_2(\ZZ)}&=
\frac{1}{6}\langle
 \Phi_1(\{\infty,0\})|(\tau-\tau^{-1}) + 2 \Phi_1([0,1]_\infty)|(\tau^{-1}-1), \Phi_2(\{\infty,0\})
\rangle_V
\end{split}
\end{equation}
ou encore, comme $\tau^{-1}=T\sigma$ et $\tau=\sigma T^{-1}$, en termes de $T$,
\begin{equation}\label{accouplement}
\left\lbrace \Phi_1, \Phi_2\right\rbrace_{\Sl_2(\ZZ)}
=\frac{1}{6}\left(
\langle \Phi_1(\{\infty,0\})|(T-T^{-1})
-2\Phi_1([0,1]_\infty)|(1+T),\Phi_2 (\{\infty,0\})\rangle_{V}
\right )
\end{equation}
pour $\Phi_1\in \Hom_{\Sl_2(\ZZ)}(\KKK_0,V_k)$ et $\Phi_2\in \Hom_{\Sl_2(\ZZ)}(\Delta_0,V_k)$.
Dans le cas où $\Phi_1\in \Hom_{\Sl_2(\ZZ)}(\Delta_0,V_k)$,
on obtient la formule de Haberland (\cite[Folgerung, p. 278]{haberland}):
\begin{equation}
\label{accouplement2}
\begin{split}\left\lbrace \Phi_1, \Phi_2\right\rbrace_{\Sl_2(\ZZ)}
&=\frac{1}{6}
\langle \Phi_1(\{\infty,0\})|(T-T^{-1}),\Phi_2 (\{\infty,0\})\rangle_{V}
\\
&=\frac{1}{6}\left(
\langle \Phi_1(\{\infty,0\})|T,\Phi_2 (\{\infty,0\})\rangle_{V} -
\langle\Phi_1(\{\infty,0\}),\Phi_2 (\{\infty,0\})|T\rangle_{V}
\right )\ .
\end{split}
\end{equation}
Écrivons pour $\Phi \in \Hom_{\Sl_2(\ZZ)}( \Delta_0,V_k)$
\begin{equation*}
\Phi(\{\infty,0\})
=\sum_{j=0}^{k-2}\binom{k-2}{j}r_j(\Phi) x^j y^{k-2-j}\ .
\end{equation*}
Donc
\begin{equation*}
\begin{split}
\Phi_1(\{\infty,0\})|(T-T^{-1})&=\Phi_1(\{\infty,0\})(x,y-x)-\Phi_1(\{\infty,0\})(x,y+x)\\
&=-2\sum_{m=0}^{k-2}\left(\sum_{\substack{0\leq n\leq m\\n\not\equiv m\pmod2}}
  \frac{(k-2)!}{n!(m-n)!(k-2-m)!}r_n(\Phi_1)\right) x^m y^{k-2-m}\ .
\end{split}
\end{equation*}
On a, en utilisant la première équation de \eqref{accouplement2},
\begin{equation*}
\begin{split}
\left\lbrace \Phi_1, \Phi_2\right\rbrace_{\Sl_2(\ZZ)}&=
\frac13
\sum_{\substack{0\leq n\leq m\leq k-2\\n\not\equiv m\pmod2}}
  (-1)^n\frac{(k-2)!}{n! (m-n)!(k-2-m)!}r_n(\Phi_1)r_{k-2-m}(\Phi_2)
  \ .
\end{split}
\end{equation*}

Si $\Phi_1=\Per(F_1)$ et $\Phi_2=\Per(F_2)$ proviennent de formes
paraboliques $F_1$ et $F_2$, on a
\begin{equation*}
\begin{split}
\Per(F_l)(\{\infty,0\})&=\int_{i\infty}^0F_l(t)(tx+y)^{k-2}dt
=\sum_{j=0}^{k-2}\binom{k-2}{j}r_j(F_l) x^j y^{k-2-j}
\end{split}
\end{equation*}
avec
\begin{equation*}
r_j(F_l)= \int_{i\infty}^0 F_l(\tau) \tau^j d\tau=r_j(\Phi_l)\ .
\end{equation*}
On retrouve la formule classique (\cite[Folgerung 1, p. 280]{haberland})
\begin{equation*}
\begin{split}
\left\lbrace  F_1, F_2\right\rbrace_{\Sl_2(\ZZ)}&=
\frac1{3(2i)^{k-1}}
\sum_{\substack{0\leq n\leq m \leq k-2\\n\not\equiv m\pmod2}}
(-1)^m\frac{(k-2)!}{n!(m-n)!(k-2-m)!}
r_n(F_1) \overline{r_{k-2-m}(F_2)}\\
&=
\frac1{3(2i)^{k-1}}\sum_{\substack{0\leq n\leq m\leq k-2\\n\not\equiv m\pmod2}}
(-1)^m\binom{k-2}{n,m-n,k-2-m}
r_n(F_1) \overline{r_{k-2-m}(F_2)}\ .
\end{split}
\end{equation*}

\begin{cor}\label{corN1}
Pour $m$ pair et compris entre 0 et $k-2$, posons
\begin{equation*}
\lambda_{k,m}=2\binom{k-1}{m}\frac{B_k}{k(k-1)}+ \sum_{\substack{0\leq n\leq k-2-m \\n\textrm{ impair}}}
\binom{k-2}{n,m,k-2-m-n}\frac{B_{k-1-n}}{k-1-n}
 \frac{B_{n+1}}{n+1}\ .
\end{equation*}
Soit $\Phi\in \Hom_{\Sl_2(\ZZ)}(\Delta_0,V_k)$.
Alors
$$\left\lbrace\Psi_k(\charun),\Phi\right\rbrace_{\Sl_2(\ZZ)}=\frac{1}{3}\sum_{\substack{0\leq j\leq k-2\\m\textrm{ pair}}}\lambda_{k,m} r_m(\Phi)\ .$$
Si $F$ est une forme parabolique pour $\Sl_2(\ZZ)$, on a
\begin{equation*}
\sum_{\substack{0\leq m\leq k-2\\m\textrm{ impair}}}\binom{k-2}m r_m(F)=0
\end{equation*}
et
\begin{equation}
\label{relation}
\sum_{\substack{0\leq m\leq k-2\\m\textrm{ pair}}}\lambda_{k,m} r_m(F)=0
\ .\end{equation}
\end{cor}
\begin{proof}
Prenons pour $\Phi_1$ l'élément d'Eisenstein
$\Psi_k(\charun)\in \Hom_{\Sl_2(\ZZ)}(\KKK_0,V_k)$
avec $\charun$ la fonction sur $(\ZZ/1\ZZ)^2$ valant 1
et $\Phi_2= \Per(F)$.
On a (voir la proposition \ref{prop:dualeisenstein})
\begin{equation*}
\left\lbrace \Psi_k(\charun), \Per(F)\right\rbrace_{\Sl_2(\ZZ)}=0
\ .
\end{equation*}
Il s'agit donc de calculer $\left\lbrace \Psi_k(\charun), \Per(F)\right\rbrace_{\Sl_2(\ZZ)}$
en utilisant la formule \eqref{accouplement}.
Rappelons que l'on a
\begin{equation*}
\begin{split}
\Psi_k(\charun)(\{\infty,0\})&=-\sum_{\substack{0\leq j\leq k-2\\j\textrm{ impair}}}
\binom{k-2}{j}\frac{B_{k-1-j}}{k-1-j} \frac{B_{j+1}}{j+1}x^j y^{k-2-j}+\frac{1}{-i\pi}\zeta'(2-k)(x^{k-2}-y^{k-2})\\
\Psi_k(\charun)([0,1]_\infty)&=
\frac{-B_k}{k(k-1)}\frac{(x+y)^{k-1} - y^{k-1}}{x}\ .
\end{split}
\end{equation*}
D'après le théorème \ref{def:biendefini}
et la proposition \ref{invariance0}, la contribution du terme
\hbox{$\frac{1}{-i\pi}\zeta'(2-k)(x^{k-2}-y^{k-2})$} est nulle.
En posant $S=x^{k-2}-y^{k-2}$ et $Q=\sum_{j=0}^{k-2}b_jx^jy^{k-2-j}$, on a
\begin{equation*}
\langle S|(T-T^{-1}),Q\rangle_{V_k}
=-2\langle\sum_{\substack{0\leq j\leq k-2\\j\textrm{ impair}}}\binom{k-2}jx^jy^{k-2-j},Q\rangle_{V_k}
=2\sum_{\substack{0\leq j\leq k-2\\j\textrm{ impair}}}b_j\ .
\end{equation*}
On en déduit que la somme des coefficients de degré impair de $\Phi_2(\{\infty,0\})$ est nulle.

En posant
\begin{equation*}
R=-\sum_{\substack{0\leq n\leq k-2\\n\textrm{ impair}}}
\binom{k-2}{n}\frac{B_{k-1-n}}{k-1-n} \frac{B_{n+1}}{n+1}x^ny^{k-2-n}\ ,
\end{equation*}
 on calcule
\begin{equation*}
\begin{split}
R|(T-T^{-1})&= R(x,y-x)- R(x,y+x)
\\&=
2\sum_{\substack{0\leq n\leq k-2\\n\textrm{ impair}}}
  \binom{k-2}{n}\frac{B_{k-1-n}}{k-1-n} \frac{B_{n+1}}{n+1}
\sum_{\substack{n+m\leq k-2\\m\not\equiv n\pmod2}}\binom{k-2-n}{m}x^{k-2-m}y^{m}\\
&=
2\sum_{\substack{0\leq m\leq k-2\\m\textrm{ pair}}}
\left(
 \sum_{\substack{0\leq n+m \leq k-2 \\n\textrm{ impair}}}
\binom{k-2}{n,m,k-2-m-n}\frac{B_{k-1-n}}{k-1-n}
  \frac{B_{n+1}}{n+1}
 \right)x^{k-2-m} y^{m}\ .
\end{split}
\end{equation*}

Calculons enfin $\Psi_k(\charun)([0,1]_\infty)|(T+1)$. Pour $P= \frac{(x+y)^{k-1} - y^{k-1}}{x}$, on a
\begin{equation*}
P|(T+ 1)=\frac{y^{k-1} -(y-x)^{k-1}+(y+x)^{k-1} - y^{k-1}}{x}
=2\sum_{\substack{0\leq m\leq k-2\\m\textrm{ pair}}}\binom{k-1}{m}x^{k-2-m}y^{m}
\ .\end{equation*}
Pour $m$ pair, le coefficient de $x^{k-2-m}y^{m}$ dans $R|(T-T^{-1})-2\frac{-B_k}{k(k-1)}P|(1+T)$ est donc
\begin{equation*}
4\frac{B_k}{k(k-1)}\binom{k-1}{m}+2 \sum_{\substack{0\leq n+ m\leq k-2 \\n\textrm{ impair}}}
\binom{k-2}{n,m,k-2-m-n}\frac{B_{k-1-n}}{k-1-n}
 \frac{B_{n+1}}{n+1}\ ,
\end{equation*}
d'où la formule pour $\lambda_{k,m}$. Ce qui termine la démonstration
du corollaire \ref{corN1}.
\end{proof}

\appendix
\section{Calculs classiques sur les séries d'Eisenstein}

\subsection{\texorpdfstring{Rappels sur les fonctions $\zeta$}{Lg}}
Si $g$ est une fonction sur $\NN$, on pose
$$ L(s, g) = \sum_{n>0} \frac{g(n)}{n^s}\ .$$
Cette définition s'étend à une fonction
sur $\ZZ/N\ZZ$ par relèvement de manière naturelle.
\begin{prop} (\cite[chap. XIV, thm. 2.1]{lang} ou \cite[app. A]{katz}).
Si $g$ est une fonction sur $\ZZ/N\ZZ$, $L(s,g)$ converge pour $\re(s) > 1$ et
se prolonge en une fonction méromorphe sur le plan complexe ayant
au plus un pôle simple en $s=1$ de résidu $\frac{\widehat{g}(0)}{N}$.
Les fonctions $L$ de $g$ et $\widehat{g}$
vérifient l'équation fonctionnelle
\begin{equation}
\label{eqfonctgen}
L(1-s,g)=(2\pi)^{-s}N^{s-1}\Gamma(s)
\big (L(s,\widehat{g}) e^{\pi i s/2}+ L(s,\widehat{g}^-) e^{-\pi i s/2}\big )
\ .\end{equation}
\end{prop}
On pose
$\CN{N}{h}=\frac{1}{N^{h}}\frac{(-2\pi i)^h}{(h-1)!}$
pour $h$ entier $\geq 1$.
\begin{cor}
Soit $g$ une fonction sur $\ZZ/N\ZZ$.
\begin{enumerate}
\item Pour $h$ entier $>1$,
\begin{equation}
\label{eqfonct}
\begin{split}
NL(1-h,g)&=
\CN{N}{h}^{-1} L(h,\widehat{g} + (-1)^h\widehat{g}^-)\\
L(1-h,\widehat{g})&=
\CN{N}{h}^{-1} L(h,g^- + (-1)^hg)\ .
\end{split}
\end{equation}
En particulier, $$L(1-h,g) = (-1)^h L(1-h, g^-)\ .$$
\item
Si
$$L(s,g) = \frac{1}{N}\frac{\hat{g}(0)}{s-1} + L^*(1,g) + O(s-1)$$
est le développement de $L(s,g)$ au voisinage de $s= 1$
avec $L^*(1,g)\in \CC$, on a
\begin{equation}
\label{eq:a}
\begin{split}
L(0,\widehat{g})+ \frac{1}{2} \widehat{g}(0) &=
\CN{N}{1}^{-1} (L^*(1,g^-) - L^*(1,g)) \\
L(0, g+g^-)&= -g(0)\\
L(0, g-g^-)&= 2 L(0,g) + g(0) \ .
\end{split}
\end{equation}
\end{enumerate}
\end{cor}
\begin{proof} Pour $h>1$,
la première égalité est une conséquence de l'équation fonctionnelle
\eqref{eqfonctgen};
la deuxième égalité se déduit de la première en échangeant les rôles de $g$ et de
$\widehat{g}$ et en utilisant le fait que $\widehat{\widehat{g\, }}=N g^-$.

Pour $h=1$, l'équation fonctionnelle \eqref{eqfonctgen} implique que
$$L^*(1,\widehat{g}) - L^*(1,\widehat{g}^-)= -\pi i g(0) - 2\pi i L(0,g)
\ .$$
En remplaçant $g$ par $\widehat{g}$, on obtient
$$L^*(1,g^-) - L^*(1,g)= \frac{1}{N}\left(-\pi i\widehat{g}(0) - 2\pi i L(0,\widehat{g})\right)
\ ,$$
ce qui donne la première identité.
La deuxième égalité se déduit de la première en l'appliquant à
$\widehat{g}$ et à $\widehat{g}^-$. Pour la troisième, on écrit simplement
que $\widehat{g}-\widehat{g}^-=2 \widehat{g}-(\widehat{g}+ \widehat{g}^-)$
et on utilise la deuxième égalité.
\end{proof}
Pour $h\neq 2$, posons $L^*(h-1,g)=L(h-1,g)$ pour simplifier l'écriture
du corollaire suivant.
\begin{cor}
Pour $h\geq 1$,
\begin{equation}
\label{eqderiv}
\begin{split}
\CN{N}{h-1}^{-1}L^*(h-1,g)&=
\frac{1}{-2\pi i}L'(2-h,\widehat{g}^- + (-1)^h \widehat{g})
+ \frac{1}{4}L(2-h, \widehat{g}^--(-1)^h \widehat{g})\\
&=
\frac{1}{-2\pi i}L'(2-h,\widehat{g}^- + (-1)^h \widehat{g})
+ \frac{1}{2}L(2-h, \widehat{g}^-)\ .
\end{split}
\end{equation}
\end{cor}
\begin{proof}
Rappelons que $$\Gamma(2-h+s) \Gamma(h-1-s) = \frac{\pi}{\sin \pi (2-h+s)}
=(-1)^h\frac{\pi}{\sin(\pi s)} \ .$$
En appliquant l'équation fonctionnelle \eqref{eqfonctgen} en $2-h+s$, on a
\begin{equation*}
\begin{split}
(2\pi)^{2-h+s} &N^{h-1 -s} L(h-1-s,g) \\
&=\frac{(-1)^{h}\pi}{\Gamma(h-1-s)\sin(\pi s)}
\left ( L(2-h+s, \widehat{g}) e^{\pi i \frac{s-h+2}{2}} +
L(2-h+s, \widehat{g}^-) e^{-\pi i \frac{s-h+2}{2}}
\right)\\
&=
\frac{i^{h-2}\pi}{\Gamma(h-1-s)\sin(\pi s)}
\left (L(2-h+s, \widehat{g}) e^{\pi i \frac{s}{2}} +
(-1)^h L(2-h+s, \widehat{g}^-) e^{-\pi i \frac{s}{2}}
\right)
\ ,\end{split}
\end{equation*}
d'où
\begin{equation*}
\begin{split}
\frac{N^{h-1}}{(-2\pi i)^{h-2}}(2\pi)^s &N^{-s} \Gamma(h-1-s)L(h-1-s,g)
=\pi\frac{ L(2-h+s, \widehat{g}) e^{\pi i \frac{s}{2}} +
(-1)^h L(2-h+s, \widehat{g}^-) e^{-\pi i \frac{s}{2}}}
{\sin(\pi s)}\end{split}
\end{equation*}
et au voisinage de $s=0$,
\begin{equation*}
\begin{split}
\frac{N^{h-1}(h-2)!}{(2\pi i)^{h-2}}L(h-1-s,g)+O(s)=
\pi\frac{L(2-h+s, \widehat{g}) e^{\pi i \frac{s}{2}} +
(-1)^h L(2-h+s, \widehat{g}^-) e^{-\pi i \frac{s}{2}}}
{\sin(\pi s)}\ .
\end{split}
\end{equation*}
On a
\begin{equation*}
\begin{cases}
L(2-h+s, \widehat{g}) e^{\pi i \frac{s}{2}}&=
L(2-h, \widehat{g})
+ (L'(2-h, \widehat{g}) + \frac{\pi i}{2}L(2-h, \widehat{g})) s + O(s^2)\\
L(2-h+s, \widehat{g}^-) e^{-\pi i \frac{s}{2}}&=
L(2-h, \widehat{g}^-)
+ (L'(2-h, \widehat{g}^-) - \frac{\pi i}{2}L(2-h, \widehat{g}^-)) s + O(s^2)\\
\sin(\pi s)&= \pi s + O(s^2)\ .
\end{cases}
\end{equation*}
Pour $h>2$ (resp. $h=2)$,
$L(2-h, \widehat{g}) + (-1)^h L(2-h, \widehat{g}^-)$ est nul (resp.
est égal à $-g(0)$ et les coefficients de $\frac{1}{s}$ se simplifient).
Donc,
\begin{equation*}
\begin{split}
(-2\pi i)\CN{N}{h-1}^{-1} L^*(h-1,g)&=(-1)^h\left(
L'\left(2-h, \widehat{g}+(-1)^h\widehat{g}^-\right) +
\frac{\pi i}{2}L\left(2-h, \widehat{g}-(-1)^h\widehat{g}^-\right)
\right)
\\& =
L'\left(2-h, \widehat{g}^-+(-1)^h\widehat{g}\right) -
\frac{\pi i}{2}L\left(2-h, \widehat{g}^--(-1)^h\widehat{g}\right)
\end{split}
\end{equation*}
et
\begin{equation*}
\begin{split}
\CN{N}{h-1}^{-1} L^*(h-1,g)& =
\frac{1}{-2\pi i}L'\left(2-h, \widehat{g}^-+(-1)^h\widehat{g} \right) +
\frac{1}{4}L\left(2-h, \widehat{g}^--(-1)^h\widehat{g}\right)
\\
&=
\frac{1}{-2\pi i}L'\left(2-h, \widehat{g}^-+(-1)^h\widehat{g} \right) +
\frac{1}{2}L\left(2-h, \widehat{g}^-\right)\ .
\end{split}
\end{equation*}
\end{proof}

\subsection{Développement de Fourier et transformée de Mellin des séries d'Eisenstein}
\begin{lem}
Soit $f$ une fonction sur $(\ZZ/N\ZZ)^2$.
La valeur de $\eisenmod{k}{f}=\CN{N}{k}^{-1} E_{k,f}$
en la pointe $\infty$ est
$$\Ikf{k}{f}=L(1-k,P_2(f)(0,\cdot)^-)$$
où $P_2(f)(0,\cdot)$ désigne la fonction $x \mapsto P_2(f)(0,x)$
sur $\ZZ/N\ZZ$.
Son $q$-développement est
$$\eisenmod{k}{f}(\tau)= L(1-k,P_2(f)(0,\cdot)^-)+ \sum_{n \geq 1, m\geq 1}
\left(P_2(f)(n,-m) + (-1)^k P_2(f^-)(n,-m)\right) m^{k-1}q_N^{n m}
$$
avec $q_N=exp(\frac{2i\pi\tau}{N})$.
Lorsque $f$ est décomposé, autrement dit $f(x_1,x_2)=f_1(x_1) f_2(x_2)$,
la transformée de Mellin $\mellin_f(s)$ de $\eisenmod{k}{f} - \Ikf{k}{f}$,
prolongement analytique de
\begin{equation*}
\begin{split}
\int_{0}^{i\infty} \left(\eisenmod{k}{f}(\tau) -
\Ikf{k}{f}\right)\tau^{s} \frac{d\tau}{\tau} \ ,
\end{split}
\end{equation*}
est égale à
\begin{equation*}
\begin{split}
\mellin_f(s)&=(\frac{N}{-2i\pi})^{s}\Gamma(s)
\left(L(s, f_1) L(s-k+1,\widehat{f_2}^-) + (-1)^k L(s, f_1^-) L(s-k+1,\widehat{f_2})
\right)\ .
\end{split}
\end{equation*}
\end{lem}
\begin{proof}
Le terme constant $\Ikf{k}{f}$ vérifie
\begin{equation*}
\begin{split}
\CN{N}{k}\Ikf{k}{f}&=
\sum_{d\neq 0} f(0,d) d^{-k}
=
\sum_{d>0} f(0,d) d^{-k} + (-1)^k \sum_{d>0} f(0,-d) d^{-k}
\\&= L(k,f(0,\cdot)) + (-1)^k L(k,f^-(0,\cdot)) \ .
\end{split}
\end{equation*}
La formule \eqref{eqfonct}
implique que
$$\Ikf{k}{f}=L(1-k,P_2(f)^-(0,\cdot))\ .$$
Le $q$-développement est une conséquence de la formule classique
\[\lim_{N\to\infty} \sum_{d=-N}^N\frac1{z+d}=\frac\pi{\tan\pi z}=(-i\pi)\left(1+2\sum_{m=1}^\infty\exp(2i \pi m z)\right)\]
pour $z$ dans le
demi-plan de Poincaré et, par dérivations successives,
$$\sum_{d\in \ZZ} (z+d)^{-k} = \frac{(-2\pi i)^{k}}{(k-1)!}\sum_{m=1}^\infty m^{k-1}\exp(2i \pi m z)
\ .$$
On a en effet
\begin{equation*}
\begin{split}
&\sum_{n\neq 0, d\equiv d_0 \bmod N}f(n,d) (nz + d) ^{-k}=
\sum_{n\neq 0, d\in \ZZ} f(n,d_0) (nz + d_0+dN) ^{-k}\\
&\quad =
N^{-k}\left(\sum_{n \geq1, d\in \ZZ} f(n,d_0) (\frac{nz +d_0}{N} + d) ^{-k}
+
\sum_{n\geq 1, d\in \ZZ} f(-n,d_0) (-1)^k (\frac{nz-d_0}{N} + d) ^{-k}\right )\\
&\quad=
C_{N,k}
\sum_{n\geq 1} \left (f(n,d_0)
\sum_{m=1}^\infty m^{k-1}\exp(2i \pi \frac{md_0}{N})q_N^{mn}
+ (-1)^k f(-n,d_0)
\sum_{m=1}^\infty m^{k-1}\exp(-2i \pi \frac{md_0}{N})q_N^{mn}
\right) \ .
\end{split}
\end{equation*}
D'où
\begin{equation*}
\begin{split}
\eisenmod{k}{f} - \Ikf{k}{f}&=
\sum_{m\geq 1,n\geq 1}
\left(\sum_{d_0\bmod N}\left (f(n,-d_0) + (-1)^k f(-n,d_0)
\right)\exp(-2i \pi \frac{md_0}{N})\right)m^{k-1}
q_N^{mn}\\
&=\sum_{n \geq 1, m\geq 1}
\left(P_2(f)(n,-m) + (-1)^k P_2(f^-)(n,-m)\right) m^{k-1}q_N^{n m} \ .
\end{split}
\end{equation*}
Faisons le calcul de la transformée de Mellin.
Si $g_1$ et $g_2$ sont des fonctions sur $\ZZ/N\ZZ$
et si
\begin{equation*}
\begin{split}
h=\sum_{n \geq 1} \sum_{m \geq 1} g_1(n) g_2(m) m^{k-1} q_N^{n m}
\end{split}
\end{equation*}
la transformée de Mellin de $h$ est le prolongement analytique de
\begin{equation*}
\begin{split}
\int_{0}^{i\infty} h(\tau) \tau^{s} \frac{d\tau}{\tau}&=
i^s\sum_{n \geq 1} \sum_{m \geq 1} \int_0^\infty g_1(n) g_2(m) m^{k-1}
y^s \exp(-2\pi \frac{nm y}{N})\frac{dy}{y}\\
\\&=
(\frac{N}{-2i\pi})^s\Gamma(s)
\sum_{n \geq 1} \sum_{m \geq 1}
g_1(n) g_2(m) m^{k-1-s} n^{-s}
\\&=
(\frac{N}{-2i\pi})^s\Gamma(s)L(s,g_1) L(s-k+1, g_2) \ .
\end{split}
\end{equation*}
Quand $f$ est de la forme $f_1 \otimes f_2$, on a
$P_2(f)(n,-m)=f_1(n) \widehat{f_2}^-(m)$ et
$P_2(f^-)(n,-m)=f_1^-(n) \widehat{f_2}(m)$
et on applique la formule précédente au $q$-développement de
$\eisenmod{k}{f}(\tau) - \Ikf{k}{f}$. Cela termine la démonstration du lemme.
\end{proof}
\begin{prop}\label{psi}
Soient $k \geq 2$ et $f\in \hypk{Fonc((\ZZ/N\ZZ)^2, \CC)}$.
Alors,
$$\Ikf{k}{N^{-1} \widehat{f}}=\int f^- \dd{\beta_k}{\beta_0}$$
et
\begin{equation*}
\periode_{N^{-1}\widehat{f}}(j+1)=
 (-1)^{j+1} \int f^-\dd{\beta_{k-j-1}}{\beta_{j+1}}
 + \begin{cases}
 \phantom{-}\int \widehat{f}\ \dd{\beta_0 }{\beta'_{k-1}}
  &\text{ pour $j=0$}
\\
 -\int \widehat{f}\ \dd{\beta'_{k-1}}{\beta_0}
&\text{ pour $j=k-2$ et $k\neq 2$} \ .
\end{cases}
\end{equation*}

\end{prop}

\begin{proof}
Il suffit de démontrer les formules pour une fonction $f$
de la forme $f=f_1\otimes f_2$. On a alors
$\widehat{f\ }=N^{-1}\widehat{f_2}\otimes \widehat{f_1}^-$
et
$P_2(\widehat{f})(0,x)=\widehat{f_2}(0)f_1(x)$,
d'où
\begin{equation*}
\begin{split}
\Ikf{k}{N^{-1}\widehat{f\ }}&=N^{-1}\widehat{f_2}(0)L(1-k,f_1^-)
=\int f_1(-x_1)f_2(-x_2) d\beta_{k}(x_1)d\beta_{0}(x_2)
\\
&=\int f^- \dd{\beta_{k}}{\beta_0}\ .
\end{split}
\end{equation*}
Montrons les formules sur $\periode_{\widehat{f}}(j+1)$.
On a \begin{equation*}
\begin{split}
\mellin_{\widehat{f}}(s)&=(\frac{N}{-2i\pi})^{s}\Gamma(s)
\left(L(s, \widehat{f_2}) L(s-k+1,f_1^-) +
(-1)^k L(s, \widehat{f_2}^-) L(s-k+1,f_1)
\right)
\end{split}
\end{equation*}
en appliquant le lemme précédent.
D'où
\begin{equation*}
\begin{split}
N\CN{N}{j+1}\periode_{N^{-1}\widehat{f}}(j+1)=
L(j+1, \widehat{f_2}) L(2-k+j,f_1^-) +
(-1)^{k} L(j+1, \widehat{f_2}^-) L(2-k+j,f_1)\ .
\end{split}
\end{equation*}
Pour $0 <j < k-2$, on a
$$L(2-k+j,f_1^-)= (-1)^{j-k+1}L(2-k+j,f_1)\ .$$
D'où,
\begin{equation*}
\begin{split}
N\CN{N}{j+1}\periode_{N^{-1}\widehat{f}}(j+1)&=
L(j+1, \widehat{f_2}) L(2-k+j,f_1^-) + (-1)^{j+1}
L(j+1, \widehat{f_2}^-) L(2-k+j,f_1^-)\\
&=
L(j+1, \widehat{f_2} + (-1)^{j+1}\widehat{f_2}^-)L(2-k+j,f_1^-)\\
&= N\CN{N}{j+1} L(-j,f_2) L(2-k+j,f_1^-) \ .
\end{split}
\end{equation*}
D'où
\begin{equation*}
\begin{split}
\periode_{N^{-1}\widehat{f}}(j+1)&=
L(-j,f_2) L(2-k+j,f_1^-)\\
&=(-1)^{j+1}\int_{(\ZZ/N\ZZ)^2} f_1(-x_1) f_2(-x_2)
d\beta_{k-j-1}(x_1)d\beta_{j+1}(x_2) \\
&=(-1)^{j+1}\int_{(\ZZ/N\ZZ)^2}f^-\dd{\beta_{k-j-1}}{\beta_{j+1}} \ .
\end{split}
\end{equation*}

Passons aux cas particuliers.
Si $j=0$, le résidu de $L(s+1,N^{-1}\widehat{f_2})$ en $s=0$ est $f_2(0)$.
On a alors au voisinage de $s=0$
\begin{equation*}
\begin{split}
N\CN{N}{1}\periode_{N^{-1}\widehat{f}}(1)&=
\left(\frac{f_2(0)}{s} + L^*(1,\widehat{f_2}) + O(s)\right)
\left(L(2-k,f_1^-)+ L'(2-k,f_1^-) s + O(s^2)\right)
\\ &\quad + (-1)^k \left(\frac{f_2(0)}{s} + L^*(1,\widehat{f_2}^-) + O(s)\right)
\left(L(2-k,f_1 )+ L'(2-k,f_1) s + O(s^2)\right)
\\
&
=f_2(0)\left(\frac{L(2-k,f_1^- + (-1)^k f_1)}{s}
 +L'(2-k,f_1^-+ (-1)^kf_1) \right )\\
 &\quad+ L^*(1,\widehat{f_2}) L(2-k,f_1^-) +(-1)^k L^*(1,\widehat{f_2}^-) L(2-k,f_1)
 + O(s) \ .
\end{split}
\end{equation*}
Comme $L(2-k,f_1^- + (-1)^k f_1)=0$,
\begin{equation*}
\begin{split}
N\CN{N}{1} \periode_{N^{-1}\widehat{f}}(1)&=
 f_2(0) L'(2-k,f_1^-+ (-1)^kf_1) +
\left(L^*(1,\widehat{f_2}) - L^*(1,\widehat{f_2}^-)\right) L(2-k,f_1^-) \ .
\\
&=
f_2(0) L'(2-k,f_1^-+ (-1)^kf_1) +
N\CN{N}{1}
\left(L(0,f_2)+\frac{1}{2} f_2(0))\right) L(2-k,f_1^-) \ .
\end{split}
\end{equation*}
En utilisant le corollaire \ref{eq:a} et le fait que
$N\CN{N}{1}=-2\pi i$,
\begin{equation*}
\begin{split}
\periode_{N^{-1}\widehat{f}}(1)&=
\frac{1}{-2i\pi} f_2(0)L'(2-k,f_1^-+ (-1)^kf_1)
+\left(
L(0,f_2)+\frac{1}{2} f_2(0) \right)L(2-k,f_1^-) \ .
\end{split}
\end{equation*}
En introduisant les distributions de Bernoulli,
\begin{equation*}
\begin{split}
\periode_{N^{-1}\widehat{f}}(1)&=
\int_{(\ZZ/N\ZZ)^2} N^{-1}\widehat{f_2}(x_1) \widehat{f_1}(-x_2) d \beta_0(x_1)
d\beta'_{k-1}(x_2)
-\int_{(\ZZ/N\ZZ)^2} f_1(-x_1)f_2(-x_2)
d \beta_{k-1}(x_1)d \beta_1(x_2)
\\
&=
\int_{(\ZZ/N\ZZ)^2} \widehat{f}\dd{\beta_0}{\beta'_{k-1}}
-\int_{(\ZZ/N\ZZ)^2} f^- \dd{\beta_{k-1}}{\beta_1}\ .
\end{split}
\end{equation*}
D'où la formule pour $j=0$.
Si $j=k-2$ et $k>2$, on applique la formule du corollaire \ref{eqderiv}:
\begin{equation*}
\begin{split}
\periode_{N^{-1}\widehat{f}}(k-1)&=
  \CN{N}{k-1}^{-1}\left(L(k-1, \widehat{f_2}) L(0,f_1^-) + (-1)^{k}
 L(k-1, \widehat{f_2}^-) L(0,f_1)\right)\\
 &=
\frac{1}{-2\pi i}\left(
   L'(2-k,f_2 + (-1)^kf_2^-) L(0,f_1^-)
+
  (-1)^k L'(2-k,f_2^- + (-1)^kf_2) L(0,f_1)
\right)\\&
\quad\quad\quad +
\frac{1}{2} L(2-k,f_2) L(0,f_1^- - f_1)
\\
&=
\frac{1}{-2\pi i}
   L'(2-k,f_2 + (-1)^kf_2^-) L(0,f_1+f_1^-)
+
  \frac{1}{2} L(2-k,f_2) L(0,f_1^- - f_1)
\\
&=- \frac{1}{-2\pi i}
   L'(2-k,f_2 + (-1)^k f_2^-)f_1(0)
+ L(2-k,f_2)
\left( L(0,f_1^-) + \frac{1}{2} f_1(0) \right)\ .
\end{split}
\end{equation*}
D'où
\begin{equation*}
\begin{split}
\periode_{N^{-1}\widehat{f}}(k-1)
=&-\int_{(\ZZ/N\ZZ)^2} N^{-1}\widehat{f_2}(x_1)\widehat{f_1}(-x_2)
  d\beta'_{k-1}(x_1)d\beta_0(x_2)\\
&+(-1)^{k-1}
\int_{(\ZZ/N\ZZ)^2} f_1(-x_1) f_2(-x_2)
d\beta_1(x_1) d\beta_{k-1}(x_2) \\
=&-\int_{(\ZZ/N\ZZ)^2} \widehat{f}
d\beta'_{k-1}d\beta_0
+(-1)^{k-1}\int_{(\ZZ/N\ZZ)^2} f^- \dd{\beta_1}{\beta_{k-1}} \ .
\end{split}
\end{equation*}
\end{proof}

\bibliographystyle{smfplain}
\selectbiblanguage{french}
\bibliography{petersson}

\providecommand{\bysame}{\leavevmode ---\ }
\providecommand{\og}{``}
\providecommand{\fg}{''}
\providecommand{\smfandname}{\&}
\providecommand{\smfedsname}{\'eds.}
\providecommand{\smfedname}{\'ed.}
\providecommand{\smfmastersthesisname}{M\'emoire}
\providecommand{\smfphdthesisname}{Th\`ese}
\begin{thebibliography}{10}

\bibitem{farey}
{\scshape K.~Belabas, D.~Bernardi {\normalfont \smfandname} B.~Perrin-Riou} --
  {\og Polygones fondamentaux d'une courbe modulaire\fg}, \emph{Publications
  Mathématiques de Besançon} (à paraitre 2020).

\bibitem{cohen}
{\scshape H.~Cohen} -- {\og Haberland's formula and numerical computation of
  petersson scalar products\fg}, in \emph{ANTS-X Conference Proceedings, San
  Diego}, 2012.

\bibitem{diamond}
{\scshape F.~Diamond {\normalfont \smfandname} J.~Shurman} -- \emph{Modular
  forms, elliptic curves, and modular curves}, Graduate Texts in Mathematics,
  vol. 228, Springer-Verlag New York, 2005.

\bibitem{eichler}
{\scshape M.~Eichler} -- {\og Eine verallgemeinerung der abelschen
  integrale\fg}, \emph{Math. Zeitschrift} \textbf{67} (1957), p.~267--298.

\bibitem{haberland}
{\scshape K.~Haberland} -- {\og Perioden von modulformen einer variabler und
  gruppencholologie\fg}, \emph{I. Math. Nachr.} \textbf{112} (1983),
  p.~246--283.

\bibitem{HW}
{\scshape G.~H. Hardy {\normalfont \smfandname} E.~M. Wright} -- \emph{An
  introduction to the theory of numbers}, fifth \smfedname, The Clarendon Press
  Oxford University Press, 1979 (1st Edition 1938).

\bibitem{heumann}
{\scshape J.~Heumann {\normalfont \smfandname} V.~Vatsal} -- {\og Modular
  symbols, eisenstein series and congruences\fg}, \emph{Journal de Théorie des
  Nombres de Bordeaux} \textbf{26} (2014), p.~709--756.

\bibitem{hida}
{\scshape H.~Hida} -- {\og Congruences of cusp forms and special values of
  their zeta functions\fg}, \emph{Invent. Math.} \textbf{63} (1981),
  p.~225--261.

\bibitem{katz}
{\scshape N.~M. Katz} -- {\og The eisenstein measure and p-adic
  interpolation\fg}, \emph{American Journal of Mathematics} \textbf{99} (1977),
  p.~238--311.

\bibitem{zagierkohnen}
{\scshape W.~Kohnen {\normalfont \smfandname} D.~Zagier} -- {\og Modular forms
  with rational points\fg}, in \emph{Modular forms} (R.~Rankin, \smfedname),
  Ellis Horwood series in math. and its applications, 1984.

\bibitem{kubert}
{\scshape D.~S. Kubert} -- {\og The universal ordinary distribution\fg},
  \emph{Bull. Soc. Math.} \textbf{107} (1979), p.~179--202.

\bibitem{kulkarni}
{\scshape R.~S. Kulkarni} -- {\og An arithmetic-geometric method in the study
  of the subgroups of the modular group\fg}, \emph{American Journal of
  Mathematics} \textbf{113} (1991), p.~1053--1133.

\bibitem{lang}
{\scshape S.~Lang} -- \emph{Introduction to modular forms. with appendices by
  d. zagier and w. feit, corrected reprint of the 1976 original}, Springer,
  Berlin, 1995.

\bibitem{pasol}
{\scshape V.~Pasol {\normalfont \smfandname} A.~A. Popa} -- {\og Modular forms
  and period polynomials\fg}, \emph{Proc. London Math. Soc.} \textbf{107}
  (2013), p.~1--31.

\bibitem{PS}
{\scshape R.~Pollack {\normalfont \smfandname} G.~Stevens} -- {\og
  Overconvergent modular symbols and p-adic l-functions\fg}, \emph{Annales
  scientifiques de l'ENS} \textbf{44} (2011), p.~1--42.

\bibitem{shimura}
{\scshape G.~Shimura} -- {\og Sur les intégrales attachées aux formes
  automorphes\fg}, \emph{Journal of the Mathematical Society of Japan}
  \textbf{11} (1959), p.~291--311.

\bibitem{stevens}
{\scshape G.~Stevens} -- \emph{Arithmetic on modular curves}, Progress in
  Mathematics, vol.~20, Birkhauser Boston, Inc, 1982.

\bibitem{stevenscupcap}
\bysame , {\og The eisenstein measure and real quadratic fields\fg}, in
  \emph{Théorie des nombres (Quebec, PQ, 1987)}, de Gruyter, Berlin, 1989,
  p.~887--927.

\bibitem{pari}
{\scshape {The PARI Group}} -- {\og {PARI}/{GP} version \texttt{2.12.0}\fg},
  \url{http://pari.math.u-bordeaux.fr}, 2019.

\bibitem{weil}
{\scshape A.~Weil} -- \emph{Elliptic functions according to eisenstein and
  kronecker}, 2nd \smfedname, Classics in Math., Springer-Verlag, Berlin, 1999
  (1st Edition 1976).

\bibitem{wittaker}
{\scshape E.~T. Wittaker {\normalfont \smfandname} G.~N. Watson} -- \emph{A
  course of modern analysis, 4th edition}, Cambridge Mathematical Library,
  1996.

\bibitem{zagier}
{\scshape D.~Zagier} -- {\og Modular parametrizations of elliptic curves\fg},
  \emph{Canad. Math. Bull.} \textbf{28} (1985), p.~372--384.

\end{thebibliography}
\nocite{HW}
\nocite{hida}
\nocite{zagierkohnen}

\end{document}